\documentclass[a4paper, 11pt,reqno]{amsart}
\usepackage{amsfonts, amsthm, amssymb, amsmath, stackengine, scalerel}
\usepackage{mathrsfs,array,tikz-cd}
\usepackage{eucal,fullpage,times,color,enumerate,accents, comment}
\usepackage[all]{xy}
\usepackage{xr}
\usepackage{url}
\usepackage{enumitem}
\usepackage[new]{old-arrows}
\usepackage{extpfeil}
\usepackage{turnstile} 
\usepackage{verbatim}  
\usetikzlibrary{graphs,decorations.pathmorphing,decorations.markings,shapes,arrows}
\input xy
\xyoption{all}
\usepackage{float}
\usepackage{subfig} 
\usepackage{hyperref}

\newcommand{\calO}{{\mathcal{O}}}

\newcommand{\calE}{\mathcal{E}}

\newcommand{\calD}{{\mathcal{D}}}

\newcommand{\calF}{\mathcal{F}}
\newcommand{\calG}{\mathcal{G}}

\newcommand{\calC}{\mathcal{C}}
\newcommand{\calS}{\mathcal{S}}
\newcommand{\calP}{\mathcal{P}}

\newcommand{\Tor}{\mathrm{Tor}}
\newcommand{\Spec}{\mathrm{Spec}}

\newcommand{\Set}{\mathrm{Set}}

\newcommand{\ord}{\mathrm{ord}}

\newcommand{\R}{\mathbb{R}}

\newcommand{\id}{\mathrm{id}}

\newcommand{\opp}{\mathrm{op}}    

\newcommand{\Loc}{\mathrm{Loc}}

\newcommand{\Frm}{\mathrm{Frm}}
\newcommand{\argu}{\text{(---)}}

\newcommand{\frap}{\mathfrak{p}}

\newcommand{\ISpec}{\mathrm{ISpec}}
\newcommand{\LSpec}{\mathrm{LSpec}}

\newcommand{\M}{\mathcal{M}}
\newcommand{\m}{\mathfrak{m}}
\newcommand{\baseS}{\mathcal{S}}

\newcommand{\thT}{\mathbb{T}}

\newcommand{\thU}{\mathbb{U}}
\newcommand{\places}{\mathrm{places}}

\newcommand{\D}{\mathcal{D}}

\newcommand{\Z}{\mathbb{Z}}
\newcommand{\N}{\mathbb{N}}
\newcommand{\Des}{\mathrm{Des}}
\newcommand{\E}{\mathcal{E}}
\newcommand{\B}{\mathrm{B}}
\newcommand{\Q}{\mathbb{Q}}
\newcommand{\spec}{\mathrm{spec}}

  \DeclareFontFamily{U}{wncy}{}
\DeclareFontShape{U}{wncy}{m}{n}{<->wncyr10}{}
\DeclareSymbolFont{mcy}{U}{wncy}{m}{n}
\DeclareMathSymbol{\Sh}{\mathord}{mcy}{"58}

\DeclareSymbolFont{matha}{OML}{txmi}{m}{it}
\DeclareMathSymbol{\varv}{\mathord}{matha}{118}

\newcommand{\lranglet}[2]{\langle#1,#2\rangle}

\def\forkindep{\mathrel{\raise0.2ex\hbox{\ooalign{\hidewidth$\vert$\hidewidth\cr\raise-0.9ex\hbox{$\smile$}}}}}

\newcommand{\OC}{{{[\mathbb{O}]}}}
\newcommand{\BG}{\mathrm{BG}}

\newcommand{\MING}[1]{{\color{blue} {\tiny \bf (M:)} {\bf #1}}}

\makeatletter
\newcommand*\bigcdot{\mathpalette\bigcdot@{.5}}
\newcommand*\bigcdot@[2]{\mathbin{\vcenter{\hbox{\scalebox{#2}{$\m@th#1\bullet$}}}}}
\makeatother

\newcommand{\LDes}{\mathrm{LDes}}

\newcommand{\J}{\mathfrak{J}}
\newcommand{\K}{\mathfrak{K}}
\newcommand{\qint}{\mathbb{Q}_{(0,1]}}

\newcommand{\gav}{{|\cdot|}}

\newsavebox{\pullback}
\sbox\pullback{%
	\begin{tikzpicture}%
	\draw (0,0) -- (1ex,0ex);%
	\draw (1ex,0ex) -- (1ex,1ex);%
	\end{tikzpicture}}

\newcommand{\RIdl}{\mathsf{RIdl}}

\DeclareMathOperator*{\colim}{colim}

\newcommand{\equalizer}[2]{\xymatrix@1{#1 \ar@<.4ex>[r] \ar@<-0.4ex>[r] & {\ } #2}}

\newcommand{\adjunction}[4]{\xymatrix@1{#1{\ } \ar@<-0.3ex>[r]_{ {\scriptstyle #2}} & {\ } #3 \ar@<-0.3ex>[l]_{ {\scriptstyle #4}}}}

\usepackage{epigraph}

\usepackage{xcolor}
\usepackage{csquotes}
\usepackage{lipsum}

\definecolor{quotemark}{gray}{0.7}
\makeatletter
\newlength\origparskip

\newcommand{\fquote}{%
	\@ifnextchar[{\fquote@i}{\fquote@i[]}
}

\def\fquote@i[#1]{%
	\@ifnextchar[{\fquote@ii{#1}}{\fquote@ii{#1}[]}
}%

\def\fquote@ii#1[#2]{%
	\def\pqm@tempa{#1}%
	\def\pqm@tempb{#2}%
	\noindent
	\list
	{}
	{\setlength{\leftmargin}{0.3\textwidth}%
		\setlength{\rightmargin}{0.1\textwidth}%
		\setlength{\origparskip}{\parskip}}%
	\item[]%
	\begin{picture}(0,0)%
	\put(-15,-8){\makebox(0,0){\scalebox{4}{%
				\textcolor{quotemark}{\textquotedblright}}}}%
	\end{picture}%
	\begingroup
	\itshape
	\ignorespaces}%

\def\endfquote{%
	\endgroup
	\par
	\raggedleft
	\ifx\pqm@tempa\empty
	\else
	{\bfseries --- \pqm@tempa\par}%
	\setlength{\parskip}{\origparskip}%
	\ifx\pqm@tempb\empty
	\else
	(\pqm@tempb)%
	\fi
	\fi
	\par
	\endlist}
\makeatother

\begin{document}
\bibliographystyle{alpha}
\newtheorem{theorem}{Theorem}[section]
\newtheorem*{theorem*}{Theorem}
\newtheorem*{condition*}{Condition}
\newtheorem*{definition*}{Definition}
\newtheorem*{corollary*}{Corollary}
\newtheorem{proposition}[theorem]{Proposition}
\newtheorem{lemma}[theorem]{Lemma}
\newtheorem{corollary}[theorem]{Corollary}
\newtheorem{claim}[theorem]{Claim}
\newtheorem{conclusion}[theorem]{Conclusion}
\newtheorem{hypothesis}[theorem]{Hypothesis}
\newtheorem{conjecture}[theorem]{Conjecture}
\newtheorem{setup}[theorem]{Setup}
\newtheorem{summarytheorem}{Summary Theorem}[section]

\newtheorem{maintheorem}{Theorem}
\renewcommand*{\themaintheorem}{\Alph{maintheorem}}
\newtheorem*{theorem:SINGLEPRIME}{Theorem~\ref{thm:DESCENTsinglePRIME}}
\newtheorem*{theorem:ARCHPRIME}{Theorem~\ref{thm:ARCHIMEDEANPLACE}}
\newtheorem*{theorem:OstrowskiQ}{Theorem~\ref{thm:ostrowskiQ}}

\theoremstyle{definition}
\newtheorem{definition}[theorem]{Definition}
\newtheorem{question}[theorem]{Question}
\newtheorem{action}[theorem]{Action Item}
\newtheorem{answer}[theorem]{Answer}
\newtheorem*{problem*}{Problem}
\newtheorem{goal}[theorem]{Goal}
\newtheorem{remark}[theorem]{Remark}
\newtheorem{observation}[theorem]{Observation}
\newtheorem{discussion}[theorem]{Discussion}
\newtheorem{guess}[theorem]{Guess}
\newtheorem{example}[theorem]{Example}
\newtheorem{condition}[theorem]{Condition}
\newtheorem{warning}[theorem]{Warning}
\newtheorem{notation}[theorem]{Notation}
\newtheorem{construction}[theorem]{Construction}

\newtheorem{problem}[theorem]{Problem}
\newtheorem{fact}[theorem]{Fact}
\newtheorem{thesis}[theorem]{Thesis}
\newtheorem{convention}[theorem]{Convention}
\newtheorem{summary}[theorem]{Summary}

\title{The Archimedean Place is a Blurred Interval at Infinity} 
\author{Ming Ng}
\thanks{\noindent Research partially supported by EPSRC Grant EP/V028812/1.}


\begin{abstract} Classically, the places of $\mathbb{Q}$ are often regarded as a one-point compactification of $\mathrm{Spec}(\mathbb{Z})$, with the real place corresponding to a formal ``prime'' added at infinity. Re-examining this picture from a topos-theoretic perspective reveals a subtler geometry: while the non-Archimedean places are identified with singletons indexed by the non-zero prime ideals of $\mathbb{Z}$, the Archimedean place is represented by the space of upper reals  $\overleftarrow{[0,1]}$, which may be informally thought of as the unit interval equipped with a non-Hausdorff topology. 
	
	On a technical level, our analysis brings together geometric logic and descent techniques from topos theory, distinguishing standard descent from lax descent toposes both at the level of sheaves and of the geometric theories they classify. More broadly, this paper brings into conversation two parallel distinctions: on the number-theoretic side, between Archimedean and non-Archimedean phenomena, and on the topos-theoretic side, between standard and lax descent. Looked at from a high level, these perspectives begin to converge on a common theme: how should the connected and the disconnected interact?
\end{abstract}
\maketitle

\section{Introduction}\label{sec:Intro}

\subsection{Motivation from Number Theory}\label{sec:NT-motivation}
 Much of the theory-building in number theory has been guided by a deep tension: while it is important to treat all the completions of the rationals $\mathbb{Q}$ (or indeed, any global field) symmetrically, there also exists clear disanalogies between the $p$-adics and the reals. The depth of these disanalogies can be measured by the fact that many powerful technologies work well in one setting but not the other. Indeed, as Mazur muses \cite{Maz}: 

\begin{quote}``A major theme in the development of Number Theory has been to try to bring $\mathbb{R}$ somewhat more into line with the $p$-adic fields; a major mystery is why $\mathbb{R}$ resists this attempt so strenuously.''
\end{quote}
This leads to a natural question, which will guide the investigations of this paper.

\begin{question}\label{qn:TENSION} What is the right perspective from which to understand this tension? That is, how can we treat the $p$-adics and the reals symmetrically whilst also accommodating their differences? 
\end{question}

One classical response is provided by the well-known analogy between number fields and function fields, which lies at the foundations of Arakelov geometry. If $C$ is a smooth affine curve (over an algebraically closed field $k$), it admits a unique completion to a smooth projective curve $\overline{C}$, constructed by adding a finite number of points. On the other hand, the arithmetic analogue starts with 
$$S=\Spec(\calO_K)\,,$$
the spectrum of the ring of integers $\calO_K$ of a number field $K$. The closed points of $S$ correspond to the non-Archimedean places of $K$, and so we may ``compactify'' $S$ by adding the Archimedean places.

This analogy turns out to be surprisingly robust. For instance, both settings admit a natural product formula. For any non-zero rational function $f\in k(C)$, we have
$$ \sum_{p\in\overline{C}(k)}\ord_p(f)=0\,,$$
i.e. the {\em degree} of principal divisor is 0. Similarly, for any non-zero element $f\in K$ in the number field,
$$\sum_{v\in \Lambda_K} \log|f|_v =0\,,$$
where $\Lambda_K$ denotes the set of places of $K$, and $|\cdot|_v$ is a suitably normalised absolute value at $v$. Thus, in both settings, the local data attached to a global object assemble to satisfy a global compatibility relation. This gives first evidence that both compactified spaces support a common arithmetic structure. Indeed, Arakelov theory further develops this picture: by incorporating the Archimedean places, one can develop an intersection theory of {\em arithmetic divisors}, mirroring the classical intersection theory of divisors on a Riemannian surface, see e.g. \cite{Arakelov,ArakelovDiophantine}.

Nevertheless, this analogy has a few important limitations. The first concerns the structure of the local fields. Given any point $p$ on a smooth projective curve $\overline{C}$, the completion of the function field $k(C)$ at $p$ is isomorphic to the field of formal Laurent series, 
$$\widehat{k(C)}_p\cong k((t)).$$
 Stated informally: in the function fields case, the local fields all look alike. This is in stark contrast with local fields arising from the rationals $\Q$: in particular, $\R$ is not isomorphic to any $\Q_p$, and $\Q_p\not\cong \Q_q$ for distinct primes $p$ and $q$. 
 
 There is also a fundamental difference at the global level. The smooth compactification of an affine curve is a genuinely geometric construction: the added points become part of a smooth projective curve with a global geometric structure. By contrast, the Arakelov compactification of $\Spec(\calO_K)$ simply adjoins the set of Archimedean places to the non-Archimedean ones. Thus, while there are compelling reasons to morally view the Arakelov compactification as a compactified affine curve, the function field analogy itself does not supply a geometric space whose underlying points recover the places.

These issues frame the central perspective of the present paper: rather than beginning with the set of places of $\Q$ and asking how to equip it with a geometric structure, we ask whether the space itself can be recovered from the algebraic and topological structure implicit within the places themselves. As we shall later see, geometric logic, with its deep connections to topos theory, provides powerful yet sensitive tools for extracting the required information. To develop this remark, let us shift gears and first explain the perspective from logic.

\subsection{The View from Logic} The existence of serious interactions between number theory and logic may come as a surprise to the uninitiated, but they are certainly not new.  To quote some recent successes, one might point to applications of $o$-minimality to the Andr\'{e}-Oort conjecture \cite{Pil11,Tsi18}, or perhaps the use of model-theoretic tools to establish the topological tameness of Berkovich spaces \cite{HruLoe}. These breakthrough results not only underscore the power of logical tools in other fields, they also add a deep nuance to our understanding of logical complexity. 

Stability is the classic model-theoretic frontier: structures which are {\em stable} are regarded as well-behaved, whereas {\em unstable} structures typically signal some form of chaos or complexity. Yet many structures arising naturally in the model theory of valued fields are unstable, whilst still exhibiting significant forms of tameness. At a high level, much of the recent work in the area can be understood as seeking out new logical frameworks in which this tameness becomes visible (e.g. stable domination, continuous logic etc.). A particularly relevant example for us is the framework of {\em Globally Valued Fields} \cite{GVF}, introduced by Ben Yaacov and Hrushovski as the theory of fields with multiple valuations satisfying a global product formula. This framework can be understood as providing a model-theoretic answer to Question~\ref{qn:TENSION}, but notice that it moves beyond classical first-order logic: globally valued fields are axiomatised in unbounded continuous logic, as developed in \cite{BY08}. 

\smallskip

This paper takes a related but different approach through {\em geometric logic}. Geometric logic differs substantially in character from the classical first-order logic that underlies much of model theory. Syntactically, geometric formulas are built from \textbf{finite} conjunctions, \textbf{arbitrary} disjunctions, and existential quantification, but do {\em not} allow \textbf{negation} in general. This mirrors how the opens of a topological space are closed under \textbf{finite} intersections and \textbf{arbitrary} unions, but not necessarily closed under \textbf{complements}. This gives rise to an unusual logic -- one which is positive, infinitary, and possesses an intrinsically topological character. 

These syntactic differences have significant consequences. For one, the compactness theorem is no longer available since the logic is now infinitary. Moreover, geometric logic is incomplete: for instance, there exists non-trivial geometric theories having no models in $\Set$.\footnote{For those familiar with locale theory, this corresponds to the fact that not all frames are spatial; see \cite[\S 2.4]{Vi07}.} For the model theorist, this signals that many familiar tools of first-order model theory do not translate into our setting.\footnote{This is in contrast to, e.g. unbounded continuous logic, which still has a compactness theorem \cite[\S 3]{BY08}.} For the topos theorist, however, this failure of completeness points in a different direction: one must enlarge the class of models beyond those found in $\Set$ and consider models in arbitrary Grothendieck toposes. This key shift in perspective leads to a fundamental structure theorem in topos theory, which we use throughout this paper:

\begin{theorem}[{{\cite[D3.1]{Elephant}, \cite[\S 5.3]{Vi07}}}]\label{thm:structure-thm}  Every geometric theory has a classifying topos. Conversely, every Grothendieck topos is the classifying topos of some geometric theory.
\end{theorem}


We shall leave the topos-theoretic details as a dark grey box, if not a black one. In broad strokes, the structure theorem gives two complementary ways of viewing a Grothendieck topos: 
\begin{align*}
 \textbf{Sheaf-theoretic.} &\qquad \E\simeq \textbf{Sh}(\calC,J),\,\text{a category of sheaves on a site;}\\
 \textbf{Logical.} &\qquad \E\simeq \baseS[\thT],\,\text{the classifying topos of a geometric theory}\,\thT\,.
\end{align*}
The first perspective originates in algebraic geometry, while the second allows us to regard the topos as a generalised space of models. In more detail: the {\em classifying topos} $\baseS[\thT]$ is characterised by the universal property that, for any Grothendieck topos $\calF$, there is a natural equivalence of categories 
$$\textbf{Geom}(\calF,\baseS[\thT])\simeq \thT\mathrm{-mod}(\calF)\,.$$
That is, {\em geometric morphisms} $f\colon \calF\to \baseS[\thT]$ correspond to {\em $\thT$-models} in $\calF$. Here, geometric morphisms should be viewed as generalised points, analogous to how continuous maps $\ast\to X$ correspond to points of a topological space $X$. This motivates the following informal picture, which we use throughout this paper.

\begin{definition}[Point-free Space]\label{def:ptFREEspace}\hfill
	\begin{enumerate}[label=(\roman*)]
		\item A \emph{(point-free) space} is described by a geometric theory $\thT$, whose models play the role of its points.\footnote{More precisely, the associated space is represented by a pseudo-functor $[\thT]\colon\mathfrak{Top}^{\opp}\to\mathfrak{CAT}$ , assigning to each Grothendieck topos $\calF$ its category of $\thT$-models; see \cite[Section B4.2]{Elephant} for details.}
		To emphasise this intuition, we write ``$x\in X$'' to mean that $x$ is a model of the theory associated to the point-free space $X$.
		
		\item 
		A {\em map} $f\colon X\rightarrow Y$ of point-free spaces is defined by a {\em geometric construction} assigning to each point $x\in X$ a point $f(x)\in Y$; this assignment is functorial in $x$.
	\end{enumerate}	 
\end{definition}

The lay reader may, without too much harm, think of a {\em point-free space} as an ordinary topological space, with {\em point-free maps} playing the role of continuous maps between them. Indeed, our suggestive point-set notation is meant to reinforce this intuition. Of course, this requires interpretation: since geometric logic is incomplete, a point-free space may not have any $\Set$-based points. Here ``$x\in X$'' really refers to the {\em generic model} of the corresponding theory, corresponding to the identity map $\id\colon X\to X$; when we reason pointwise, we are thus using point-set notation as shorthand for reasoning with these generic models. For more context behind the point-set notation, see \cite[Remark 5.3]{Moerdijk}.


The main subtlety is that, since points are now $\thT$-models as opposed to mere elements of a set, maps between point-free spaces are given by modifications of $\thT$-models via {\em geometric constructions}.  Specifically, these are constructions preserved under pullback along geometric morphisms, constituting a strict regime of {\em constructive mathematics} we call  ``geometric mathematics'' or ``geometricity''.\footnote{Categorically, geometric constructions are built from finite limits and arbitrary colimits.}  Here is the methodological upshot: we may reason about point-free spaces using ``point-set intuition'' provided that the arguments we make adhere to the constraints of geometricity (e.g. no invoking Law of Excluded Middle or other classical principles). For additional details, particularly regarding the topos theory underlying Definition~\ref{def:ptFREEspace}, we refer to \cite{Vi07,VickersPtfreePtwise}.

This unusual marriage between topology and logic gives rise to a quite different regime of topology, what we call \emph{point-free topology}.\footnote{There are various usages of the term ``point-free topology'' in the literature (e.g. in locale theory, formal topology), but to our knowledge these are all subsumed by the topos-theoretic definition -- see e.g. \cite{ViSublocales}.} In classical point-set topology, a space begins with an underlying set of points, which is then equipped with additional topological structure. By contrast, point-free topology reverses the order: the points are models of a geometric theory, and the theory itself determines the corresponding point-free space through its classifying topos. In particular, the underlying set of points is no longer the primitive data, but rather one manifestation of a richer geometric framework. It is this shift in perspective that allows us to reconsider Question~\ref{qn:TENSION} without taking the underlying set of places as the starting point, in effect pulling the problem away from the set theory. 

 
\subsection{Results}\label{sec:results} The present paper builds on previous work \cite{NV,NVOstrowski}, and applies descent techniques to study the places of $\Q$. A striking feature of our approach is that, rather than treating the places as a discrete set of points, we ask how they might be characterised as a point-free space. To set up the main theorems of this paper, we first review some of our previous results, fixing a few conventions along the way.

\begin{convention} Throughout this paper, and unless stated otherwise, the term ``topos'' means  {\em Grothendieck} topos, while ``space'' and ``map'' mean {\em point-free} space and {\em point-free} map in the sense of Definition~\ref{def:ptFREEspace}. 
\end{convention}

\begin{convention}\label{conv:geometricity} The term ``geometric'' is used in the sense of {\em geometric constructions}, as introduced above. Moreover, all our main results are proved {\em geometrically}, i.e. we provide a {\em constructive} proof within geometric mathematics valid over any elementary base topos $\baseS$ with nno. 	
\end{convention}

\begin{convention} Although Question~\ref{qn:TENSION} can be stated for general number fields $K$, this paper will focus on the special case where $K=\Q$.
\end{convention}

A useful starting point is the Dedekind reals. In geometric logic, one can axiomatise the theory of Dedekind reals, whose $\Set$-based models recover the usual topological space of Dedekind reals \cite{Vi3}.\footnote{This is another departure from classical model theory, where Dedekind reals typically arise as {\em types} as opposed to models of a first-order theory. The language of filters can be helpful in appreciating these different manifestations of the Dedekind reals; see \cite[\S 1.1]{NgBerk} for details.} One of the main contributions of \cite{NV} was to show that one can, in the geometric context, define real exponentiation as a point-free map.

\begin{maintheorem}[{{\cite{NV}}}]\label{thm:xzetaDed} There exists an exponentiation map on the Dedekind reals
$$	\exp\colon (0,\infty)\times \R \to (0,\infty)
\text{,}$$
	satisfying the usual exponent laws
	\begin{gather}
	x^{\zeta+\zeta'}=x^\zeta x^{\zeta'}\text{, }
	\quad x^0=1
	\nonumber 
	\\
	x^{\zeta\cdot \zeta'}=(x^\zeta)^{\zeta'}\text{, }
	\quad x^1 = x \nonumber 
	\\
	(xy)^\zeta = x^\zeta y^\zeta\text{, }
	\quad 1^\zeta = 1. \nonumber
	\end{gather}
\end{maintheorem}

The relevance of real exponentiation lies in its role in defining the places of $\Q$. Classically, an {\em absolute value on $\Q$} is defined as a multiplicative norm
$$|\cdot|\colon \Q\to [0,\infty)\,$$
satisfying the triangle inequality. Each absolute value determines a metric completion of $\Q$, giving rise to the familiar local fields $\R$ and $\Q_p$. It is therefore natural to ask when two absolute values determine the same completion up to topological isomorphism. This turns out to have an elegant answer: two absolute values $|\cdot|_1,|\cdot|_2$ determine isomorphic completions just in case there exists $\alpha\in(0,1]$ such that $$|\cdot|_1=|\cdot|_2^\alpha \qquad\text{or}\qquad |\cdot|_2=|\cdot|_1^\alpha\,.$$
A {\em place} is defined as an equivalence class of absolute values under this relation. 

Our paper \cite{NVOstrowski} re-examines this picture in the geometric context. We first show that the theory of absolute values on $\Q$ can be axiomatised in geometric logic.\footnote{Compare this, perhaps, with the axiomatisation of globally valued fields in unbounded continuous logic \cite{GVF}.} The geometric account of exponentiation from Theorem~\ref{thm:xzetaDed} then provides the operation needed to formulate the equivalence relation defining places. We then give a geometric proof of several Ostrowski-type theorems, in particular the following:

\begin{maintheorem}[{{\cite[Theorem 5.8]{NVOstrowski}}}]\label{thm:ostrowskiQ} Let $\gav$ be a non-trivial absolute value on $\mathbb{Q}$. 
\underline{Then}, one of the following must hold:
	\begin{enumerate}[label=(\roman*)]
		\item $\gav=\gav^\alpha_{\infty}$ for some $\alpha\in(0,1]$; or
		\item $\gav=\gav_p^\alpha$ for some $\alpha\in (0,\infty)$ and some (unique) prime $p\in\mathbb{N}_+$. 
	\end{enumerate} 
Moreover, the spaces of absolute values take the form
 $$[av_A]\cong (0,1]\qquad \text{and}\qquad [av_{NA};p]\cong (0,\infty),$$ where $[av_A]$ denotes the space of (non-trivial) Archimedean absolute values, and $[av_{NA};p]$ denotes the space of non-Archimedean absolute values associated to prime $p$.
\end{maintheorem}

This classification frames the strategy of the present paper. Rather than constructing the entire space of places at once, we first isolate a single place and study its associated topos. As anticipated in Section~\ref{sec:NT-motivation}, one might expect each place to correspond to a singleton, carrying no further intrinsic structure. Our first theorem confirms this expectation in the non-Archimedean case. 

\begin{maintheorem}[Theorem~\ref{thm:DESCENTsinglePRIME}]\label{thm:thmC} For any non-Archimedean place of $\Q$, let $\calD$ denote its associated topos. Then, 
	$$\mathcal{D}\simeq\baseS\{\ast\}.$$
That is, $\calD$ is equivalent to the category of sheaves on a singleton, and so every non-Archimedean place is represented by a singleton space.
\end{maintheorem}

However, here comes the big surprise. When we consider the Archimedean (i.e. the real) place, the topos theory reveals a richer structure.

\begin{maintheorem}[Theorem~\ref{thm:ARCHIMEDEANPLACE}] \label{thm:thmD} Let $\D'$ denote the topos of the Archimedean place of $\Q$. Then,
	 $$\D'\simeq\baseS\overleftarrow{[0,1]},$$
where $\overleftarrow{[0,1]}$ denotes the space of {\em upper reals} bounded between 0 and 1, and $\baseS\overleftarrow{[0,1]}$ its category of sheaves.
\end{maintheorem}

Taken together, these results challenge the classical Arakelov picture of the places of $\Q$ as a kind of one-point compactification of $\Spec(\Z)$:  while the non-Archimedean places are singletons by Theorem~\ref{thm:thmC}, the Archimedean place instead carries the structure of a blurred unit interval equipped with a non-Hausdorff topology (Theorem~\ref{thm:thmD}). Moreover, we emphasise that the topos of places should not be confused with the Berkovich spectrum of $\Z$: the latter parametrises absolute values themselves, whereas here we have already passed to their equivalence classes defining places. In particular, unlike the Berkovich picture, the non-Archimedean places define singletons in our setting.

What are the implications of Theorems~\ref{thm:thmC} and \ref{thm:thmD}, e.g. for arithmetic geometry? At this critical juncture, our understanding is still incomplete and many interesting questions remain, some potentially quite deep. Nonetheless, some partial answers are explored in the final two sections of the paper. 
\begin{itemize}
	\item Section~\ref{sec:ARCHvsNONARCH} gives a topos-theoretic account of the differences between the Archimedean and non-Archimedean places. In our language, the Archimedean place exhibits non-trivial forking in its sheaves whereas the topos corresponding to a non-Archimedean place eliminates all forking behaviour. This is summarised in Conclusion~\ref{conc:forking}.
	\item  Section~\ref{sec:strangewoods} is expository, and discusses some interesting points of contact with other work in the literature. In particular, it brings into focus a theme that has implicitly guided our analysis thus far: how do the connected and disconnected interact? 
	
\end{itemize}

\subsection*{Acknowledgements} This is work from my thesis \cite{NgThesis}. It is a pleasure to acknowledge that Theorem~\ref{thm:thmD} is joint work with Steve Vickers, and to thank him for his support in incorporating it in the present paper. Thanks also to A. Connes, J.I. Burgos Gil, E. Finster, and N. Hultberg for various helpful comments.
\tableofcontents

\section{Preliminaries in Point-free Topology}\label{sec:prelim} As mentioned above, there are two complementary ways of viewing a topos: one may study a topos in terms of its objects (i.e. sheaves), or in terms of the geometric morphisms into it (viewed as generalised points or models). The guiding methodology of this paper lies in leveraging these two perspectives against each other. In particular, we mainly focus on {\em localic spaces}, a well-behaved setting where the connections between point-free and classical point-set topology become especially transparent.

Section~\ref{sec:localic} recalls the basic framework of localic spaces, introduces the localic reals and $\ISpec(\Z)$ as the key examples of our paper, and collects some technical results needed later, including the Lifting Lemma~\ref{lem:liftingsheaves}. Section~\ref{sec:Des} reviews standard and lax descent, reformulating the relevant constructions in the language of point-free topology.

\subsection{Localic Spaces}\label{sec:localic} Logic plays a primarily background role in this paper, so we will be brief on the  details. At a high level, a {\em first-order theory} $\thT$ is a formal description of a mathematical structure --- a description built from specific logical connectives according to specific rules. This paper focuses on {\em geometric logic}, a fragment of logic built out of finite conjunctions, arbitrary (possibly infinite) disjunctions, and existential quantification, whilst forbidding nested implications and universal quantification. Any first-order theory $\thT$ defined within geometric logic is called a {\em geometric theory}.

We can now unpack Definition~\ref{def:ptFREEspace}. Recall that a {\em point-free space} is presented by a geometric theory $\thT$, with its models play the role of points. Here, a $\thT$-model is
a structure $M$ satisfying the axioms of $\thT$. The appropriate morphism between such points are the {\em specialisation morphisms} $\alpha\colon M\to M'$, given by homomorphisms of $\thT$-models.\footnote{Categorically, a $\thT$-model corresponds to a geometric
	morphism between toposes, and a specialisation morphism corresponds to a
	geometric natural transformation.} A {\em map} between point-free spaces is then given by a geometric construction on models (Convention~\ref{conv:geometricity}), and such constructions are functorial with
respect to specialisation morphisms. This functoriality will play a key
role in our analysis, particularly in Section~\ref{sec:Arch}.

\subsubsection{Propositional Theories and Frames}\label{subsec:prop-th-frame} A geometric theory $\thT$ is called {\em propositional} if its signature has no sorts --- in other words, its axioms are constructed only from constant symbols, $\top$ (true), finite $\land$ and arbitrary $\bigvee$. This is directly analogous to the way the opens of a topological space are generated by finite intersections and arbitrary unions, which is made precise via the notion of a {\em frame}.

	\begin{definition}[Frame]\label{def:frame} \hfill
	\begin{enumerate}[label=(\roman*)]
		\item 	A \emph{frame} is a complete lattice $A$ possessing all small joins $\bigvee$ and all finite meets $\land$, such that the following distributivity law holds
		\[a\land \bigvee S = \bigvee \{a\land b \,|\, b\in S\} \qquad \text{for}\,\,  a\in A, S\subseteq A.\]
		\item 	A \emph{frame homomorphism} is a function between frames that preserves arbitrary joins and finite meets.
	\end{enumerate}	
\end{definition}

Clearly, the opens of any topological space $X$ form a frame, which we denote as $\Omega_X$. On the side of logic, the {\em Lindenbaum algebra} of a propositional theory $\thT$ --- i.e. the lattice of its formulae modulo provable equivalence --- is likewise a frame. We denote the Lindenbaum algebra of $\thT$ as $\Omega_\thT$, and regard its elements as the opens of the corresponding localic space. Moreover, if $\thT$ and $\thT'$ are propositional, then a point-free map between their associated spaces 
$f\colon [\thT]\to[\thT']$ correspond to frame homomorphisms between their Lindenbaum algebras $f^{-1}\colon \Omega_{\thT'}\to \Omega_{\thT}$ (note the reversal of arrow direction). For details on the underlying topos theory, see \cite[\S 2.2]{Vi07}, \cite[C1.3-4]{Elephant}.


\subsubsection{Essentially Propositonal Theories} This paper works with a slightly larger class of theories. We call a geometric theory $\thT$ {\em essentially propositional} if its classifying topos is such that $\baseS[\thT]\simeq\baseS[\thT']$ for some propositional theory $\thT$. Informally, this means that the space of $\thT$-models is equivalent to the space of $\thT'$-models, even if {\em a priori} $\thT$ has more expressive power than $\thT$. This motivates the following key definition:

\begin{definition}[Localic Space]\label{def:localic} If $\thT$ is essentially propositional, we call the space of its models a {\em localic space}. Localic spaces form a category $\Loc$, whose morphisms are the point-free maps. Whenever we wish to emphasise that a theory need not be essentially propositional, we speak of a {\em generalised space}.
\end{definition}

Localic spaces occupy a conceptual sweet spot in point-free topology, and are the primary objects of interest in this paper. On one hand, the connection with frames remains intact, and so many familiar topological notions continue to make sense in the localic setting. (In particular, sheaves on a localic space may be characterised as \'{e}tale bundles; we return to this point in Section~\ref{sec:etale-bundles}.) On the other hand, the added flexibility of allowing sorts provides a more natural language for describing mathematical structures of interest to us, whilst still retaining the point-free character of the theory.

An important source of such theories is given by geometric theories whose sorts are free algebra constructions, such as $\N$, $\Z$ or $\Q$.\footnote{This marks another departure from classical model theory. In classical first-order logic, L\"{o}wenheim-Skolem implies that no set of axioms can force a structure to be isomorphic to some fixed infinite model. In the presence of infinite disjunctions, however, this is no longer the case -- see \cite[\S 3.4]{Vi07}, in particular the analysis of the geometric theory of $\N$.} Here we will focus on two main examples: localic reals (Section~\ref{subsec:localicreals}), and the coZariski spectrum $\ISpec(\Z)$ (Section~\ref{subsec:localicprimes}).

\subsubsection{The Localic Reals}\label{subsec:localicreals} This paper uses two different types of reals: Dedekind reals and the one-sided reals. Both reals are built from the rationals but in different ways; this results in different topologies and therefore different subtleties in their analysis.

\begin{definition}[Dedekind Reals]\label{def:DedR} The localic space $\R$ is a space whose points are the Dedekind cuts of $\Q$. Concretely, a point in $\R$ is a pair $(L,R)$ of subsets of $\Q$ satisfying the axioms:
	\begin{enumerate}
		\item $L$ and $R$ are {\em inhabited}, i.e. there exists $q\in L$ and $q\in R$;
		\item $R$ is a {\em rounded upper set}, i.e. $q\in R \iff $ there exists $q'\in R$ with $q'<q$;
		\item $L$ is a {\em rounded lower set}, i.e. $q\in L\iff$ there exists $q'\in L$ with $q<q'$;
		\item $L$ and $R$ are \emph{separated} as well as {\em located}, 
		\[ q\in L,\ r\in R \Longrightarrow q<r,
		\qquad\text{and}\qquad 
		q<r \Longrightarrow q\in L\text{ or }r\in R.
		\] 
 	\end{enumerate}
{\em Convention.} We typically denote a point of $\R$ simply by $x$, rather than specifying the pair $(L_x,R_x)$. We write $q<x$ to mean $q\in L_x$ and $x<r$ to mean $r\in R_x$.
\end{definition}

Informally, a Dedekind real approximates a number from above and below. By contrast, a {\em one-sided real} approximates the number from only one direction. Classically, the two notions are essentially equivalent; in the geometric setting, however, they give rise to localic spaces equipped with different topologies (= different Lindenbaum algebras). More explicitly:

\begin{definition}[One-Sided Reals]\label{def:one-sided} There are two kinds of one-sided reals.\footnote{Notice the definition allows the upper (resp. lower) reals to be empty, which correspond to $\infty$ (resp. $-\infty$).}
	\begin{enumerate}[label=(\roman*)]
		\item The space of \emph{upper reals}, denoted $\overleftarrow{[-\infty,\infty]}$, has as points the rounded upper subsets of $\Q$.
		\item The space of \emph{lower reals}, denoted $\overrightarrow{[-\infty,\infty]}$, has as points the rounded lower subsets of $\Q$.
	\end{enumerate}
The points of either space are subsets of $\Q$, and are therefore ordered by subset inclusion. We call this the {\em specialisation order}, denoted $\sqsubseteq$. For lower reals, this agrees with the usual numerical order: 
$$x\sqsubseteq y\iff x\leq y,$$
whereas for upper reals it is reversed: 
$$x\sqsubseteq y \iff y\leq x.$$
The arrows in $\overleftarrow{[-\infty,\infty]}$ and $\overrightarrow{[-\infty,\infty]}$ are meant to indicate the direction of refinement under the specialisation order. The specialisation order in turn induces the {\em Scott topology} on the respective spaces.	
\end{definition}

\begin{remark} Of course, the term ``specialisation order'' reflects the fact that $\sqsubseteq$ defines the {\em specialisation morphisms} between the points of the one-sided reals, as defined in the introductory remarks of Section~\ref{sec:localic}.
\end{remark}

By imposing additional axioms on Definition~\ref{def:DedR} one obtains the usual interval subspaces of the Dedekind reals -- e.g. requiring the left Dedekind section $L$ to contain 0 gives the open interval of positive reals $(0,\infty)$. Similarly, one can define intervals of one-sided reals, but there is an important subtlety: these subspaces inherit the Scott topology from the original space, and must therefore be closed under directed joins along the specialisation order. In other words, every one-sided interval must be closed at the arrowhead, and so for instance we write $\overleftarrow{[0,1]}$ --- as opposed to $\overleftarrow{(0,1]}$ --- for the space of upper reals between 0 and 1.

\subsubsection{The Lifting Lemma} Although perhaps unusual to the classical mathematician, the one-sided reals arise as natural examples of a well-known construction in domain theory known as \textit{rounded ideal completions} \cite{Smyth,Vi8}.

\begin{definition}[Rounded Ideal Completions]\label{def:Rstructure}
	Let $(Y,\prec)$ be a set equipped with a dense transitive order $\prec$.\footnote{Thus, whenever $q\prec q'$, there exists $q''$ with $q\prec q''\prec q'$. We emphasise that we do not require $\prec$ to be a strict order.}
	A subset $I\subseteq Y$ is an \emph{ideal} if it is downward closed and every finite subset of $I$ has an upper bound in $I$. If
	\[
	I_q:=\{q'\in Y\mid q'\prec q\}
	\]
	is an ideal for every $q\in Y$, we call $(Y,\prec)$ an \emph{$R$-structure}. The \emph{rounded ideal completion} of an $R$-structure is the space
	\[
	\RIdl(Y,\prec)
	\]
	of all ideals of $Y$.  $\RIdl(Y,\prec)$ carries a specialisation order $\sqsubseteq$ defined by subset inclusion, which in turn determines the Scott topology on the space.
	
	
\end{definition}


The terminology ``rounded'' is justified by the density of $\prec$: every ideal is automatically rounded in the sense that  for any $q\in S$, there exists $q'\in S$ such that $q\prec q'$. Moreover, the fact that $\prec$ is not required to be a strict order allows us to represent subspaces one-sided reals as rounded ideal completions, where $Y$ is some subset of $\mathbb{Q}$ and $\prec$ is the standard order $<$ on $\mathbb{Q}$ except possibly reversed or modified to permit edge cases. This fact was used extensively in \cite{NV}. The most relevant example for us is the following.

\begin{example}\label{ex:ridlreals} Denote $\qint:=\{q\in\mathbb{Q} \,|\,0<q\leq 1 \}$ and define $x\prec y$ iff $x>y$ or $x=y=1$. Then 
	$$\overleftarrow{[0,1]}\cong \RIdl(\mathbb{Q}_{(0,1]},\prec).$$
	
\end{example}

\begin{convention}\label{conv:[0,1]ridl} By Example~\ref{ex:ridlreals}, a point $\gamma\in\overleftarrow{[0,1]}$ may be viewed either as an upper real in the usual sense of Definition~\ref{def:one-sided}, or a rounded ideal $I_\gamma\in\RIdl(\qint,\prec)$ in the sense of Definition~\ref{def:Rstructure}. This paper will use both representations interchangeably, depending on convenience.
	
\end{convention}

The language of rounded ideal completions allows us to reduce many questions about the one-sided reals to questions about the rationals, which are comparatively easier to work with. The following lemmas develop this remark.

\begin{lemma}\label{lem:filteredCOLIMITpts} Let $f\colon X\rightarrow Y$ be a map of (generalised) spaces. Then $f$ preserves filtered colimits of points
	$$f(\colim_{i\in J} x_i)\cong \colim_{i\in J} f(x_i)\,.$$
\end{lemma}
\begin{proof} This translates \cite[Theorem 1.37]{Vi07}.
\end{proof}

\begin{lemma}\label{lem:EPIridl} Let $\RIdl(Y,\prec)$ be the rounded ideal completion of $R$-structure $(Y,\prec)$. Then, the map
	\begin{align}\label{eq:CANONridlMAP}
	\psi\colon Y &\longrightarrow \RIdl(Y,\prec) \\ 
	q&\longmapsto  I_q:=\{q'\in Y\,|\, q'\prec q\},\nonumber 
	\end{align}
is an epimorphism of spaces.
\end{lemma}
\begin{proof} It is clear that the canonical map $\psi$ is well-defined since $Y$ is an $R$-structure (and so $I_q$ is an ideal of $Y$ by definition). To show that $\psi$ is an epimorphism, suppose we have two maps $g_1,g_2\colon \RIdl(Y,\prec)\rightarrow Z$ such that $g_1 \psi\cong g_2\psi$. Now every ideal $I\in \RIdl(Y,\prec)$ can be represented as a directed join $I=\bigcup_{q\in I} I_q$, which is a filtered colimit. Hence, apply Lemma~\ref{lem:filteredCOLIMITpts} to compute
	\[g_1(I)=g_1(\displaystyle\bigcup_{q\in I} I_q) \cong \displaystyle\colim_{q\in I}g_1(I_q)\cong \displaystyle\colim_{q\in I}g_1 \psi(q)\qquad \text{and}\qquad g_2(I)\cong \displaystyle\colim_{q\in I}g_2\psi(q). \]
	Since $g_1\circ  \psi \cong  g_2 \circ \psi$ by hypothesis, it follows that $g_1\cong g_2$, showing that $\psi$ is indeed an epimorphism.
\end{proof}

This sets up the following key lemma. 

\begin{lemma}[Lifting Lemma]\label{lem:liftingsheaves} As our setup,
	
	\begin{itemize}
		\item Let $X$ be a space;
		\item Let $(Y,\prec)$ be an $R$-structure.
	\end{itemize}
	Then, the epimorphism from Lemma~\ref{lem:EPIridl}
	induces an equivalence between:
	\begin{enumerate}[label=(\roman*)]
		\item A map $f\colon Y\!\longrightarrow X$ equipped with a system of transition maps $$\left\{\theta_{q'q}\colon f(q')\rightarrow f(q) \right\}_{q'\prec q}\,,
	 $$
	 where each $\theta_{q'q}$ is a specialisation morphism between the corresponding points of $X$.\footnote{{\em Categorical interpretation.} Generalised points of $X$ correspond to geometric morphisms into the topos $\baseS[X]$, so the transition maps here correspond an indexed system of geometric transformations between $\baseS[X]$-valued geometric morphisms.}  Moreover, these transition maps are required to satisfy the following continuity conditions:
		\begin{itemize}
			
			\item \emph{{\bf (Cocycle condition).}} If $q''\prec q'\prec q$ then $$\theta_{q''q}=\theta_{q'q}\circ \theta_{q''q'}\,;$$
			\item \emph{{
				\bf (Colimit condition).}} The induced map $$\theta_{q}\colon\displaystyle\colim_{q'\prec q}f(q')\rightarrow f(q)$$ is an isomorphism. 
		\end{itemize}
		\item A map $\overline{f}\colon\RIdl(Y,\prec)\longrightarrow X$. 
	\end{enumerate}	
\end{lemma}


\begin{proof} We divide the proof into two main steps.
	\subsubsection*{Step 1: Transforming the given map.} Suppose $f\colon Y\!\longrightarrow X$ is a map equipped with transition maps satisfying the continuity conditions of the lemma. For any $I\in\RIdl(Y,\prec)$, the collection \[ \{f(q)\}_{q\in I} \] forms a filtered diagram: filteredness follows from the ideal property of $I$, while the cocycle condition ensures that the transition maps are compatible. Since toposes possess all filtered colimits of their points 
	 \cite[Corollary 7.14]{J0}, we may therefore define
	\begin{align}\label{eq:GET[0,1]}
	\overline{f}\colon\RIdl(Y,\prec)&\longrightarrow X\\
	I&\longmapsto \displaystyle\colim_{q\in I}f(q).\nonumber
	\end{align}
	Conversely, suppose we have a map $\overline{f}\colon \RIdl(Y,\prec)\rightarrow X$. We can then define a map
	\begin{align}\label{eq:GETqint}
	f\colon Y&\longrightarrow X\\
	q&\longmapsto \overline{f}(I_q) \nonumber
	\end{align}
	where $I_q$ as in Equation~\eqref{eq:CANONridlMAP}. By Definition~\ref{def:ptFREEspace}, this assignment is functorial in $q$. Hence, since $q'\prec q$ implies $I_{q'}\sqsubseteq I_{q}$ in $\RIdl(Y,\prec)$, $\overline{f}$ sends this morphism between points to a morphism
	$$
	\theta_{q'q}\colon
	f(q')
	\longrightarrow
	f(q)\,,
	$$
	which defines a system of transition maps satisfying the cocycle condition. Finally, the colimit condition follows by applying Lemma~\ref{lem:filteredCOLIMITpts}, which yields
	\[\colim_{q'\prec q}f(q')=\colim_{q'\prec q}\overline{f}(I_{q'})\cong \overline{f}(I_q)=f(q).\]

	\subsubsection*{Step 2: Proving equivalence} Suppose we are given $f\colon Y\rightarrow X$. Following Step 1, define
	\begin{align*}
	g\colon Y &\longrightarrow X\\
	q&\longmapsto \overline{f}\circ\psi(q)\,,
	\end{align*}
	where $\overline{f}$ is defined as in Equation~\eqref{eq:GET[0,1]}, and $\psi$ is the epimorphism from Lemma~\ref{lem:EPIridl}. Unpacking definitions, \[g(q)=\overline{f}\circ \psi(q)= \overline{f}(I_q)=\displaystyle\colim_{q'\prec q}f(q')\cong f(q),\]
	where the final isomorphism follows from the colimit condition. Thus, $g\cong f$.
	
	Conversely, suppose we are given $\overline{f}\colon\RIdl(Y,\prec)\rightarrow X$. Following Step 1 again, define 
	\begin{align*}
	\overline{g}\colon\RIdl(Y,\prec)&\longrightarrow X\\
	I&\longmapsto \colim_{q\in I}\overline{f}(I_q)\,\,.
	\end{align*}
	For every $q\in Y$, Lemma~\ref{lem:filteredCOLIMITpts} gives
	$$
	\overline g\circ \psi(q)
	=
	\colim_{q'\prec q}\overline f(I_{q'})
	\cong
	\overline f(I_q)
	=
	\overline f\circ \psi(q).
	$$
Since $\psi$ is an epimorphism by Lemma~\ref{lem:EPIridl}, this gives $\overline{f}\cong \overline{g}$, finishing the proof.\end{proof}

\begin{remark} The Lifting Lemma~\ref{lem:liftingsheaves} generalises \cite[Lemma 1.29]{NV}: notice here we no longer require $X$ to be localic. 
\end{remark}

\subsubsection{Localic Spectra}\label{subsec:localicprimes} Let us revisit the discussion of Arakelov compactification from the Introduction.  Classically, the spectrum of a commutative ring
$R$, denoted $\spec(R)$, is the set of prime ideals of $R$, equipped with the Zariski topology. In the point-free setting, however, points and topology are specified simultaneously: a space is presented by a  geometric theory $\thT$, whose models are its points and whose Lindenbaum
algebra is its frame of opens. Thus, to construct a point-free spectrum one
must first choose an appropriate geometric presentation.
\smallskip

There are two natural choices, corresponding to the Zariski and coZariski
topologies \cite{ColeSpectra,JohnstoneSpectra}. 


\begin{example}[The Zariski Spectrum]\label{ex:LSpec} Let $R$ denote a commutative ring with $1$. The \emph{Zariski Spectrum} $\LSpec(R)$ for $R$ is the space of prime filters of $R$.
In $\Set$, a prime filter is equivalently the
complement of a non-trivial prime ideal. Thus the classical points of
$\LSpec(R)$ are the points of $\spec(R)$, with the usual Zariski topology
generated by
\[
D(a)=\{\frap\in\spec(R)\mid a\notin\frap\}.
\]
\end{example}

\begin{example}[The coZariski Spectrum]\label{ex:ISpec} The \emph{coZariski Spectrum} $\ISpec(R)$ for $R$ is the space whose points are the \emph{prime ideals} of $R$. Classically, its underlying set is again
	$\spec(R)$, but now equipped with the coZariski topology generated by the sub-basic opens
	\[
	V(a)=\{\frap\in\spec(R)\mid a\in\frap\}.
	\]

\end{example} 

\begin{remark} As defined, $\LSpec(R)$ and $\ISpec(R)$ are generalised spaces, not necessarily localic. However, this paper will only focus on the special case $R=\Z$. Since $\Z$ is a free algebra construction, this justifies regarding $\LSpec(\Z)$ and $\ISpec(\Z)$ as localic spaces. For more details on the role of free algebra constructions in defining geometric theories, see the discussion after Definition~\ref{def:localic} or \cite[\S 3.4]{Vi07}.
\end{remark}

The upshot is that while $\LSpec(R)$ and $\ISpec(R)$ have the same classical points, their
point-free presentations, and thus their respective topologies, are different. A key insight from \cite{NVOstrowski} is that these topological differences have strong implications for the point-free analysis of absolute values.

\begin{discussion}\label{dis:ISpec} The geometric theory of absolute values on $\Z$ has been axiomatised in \cite[\S 2]{NVOstrowski}, and thus has a corresponding point-free space (cf. Theorem~\ref{thm:structure-thm}). Define $[av_U]$ to be the space of absolute values on $\Q$ satisfying the ultrametric inequality. In light of the Ostrowski's Theorem~\ref{thm:ostrowskiQ}, one expects a quotient map
	\begin{equation}
	\mathrm{quot}\colon [av_{U}]\longrightarrow [\mathrm{places}_{U}]
	\end{equation}
	sending an ultrametric absolute value to its corresponding place. However, if $[\mathrm{places}_U]=\LSpec(\Z)$, then \cite[Observation 5.12]{NVOstrowski} points out that no such map exists. By contrast, the natural map
	\begin{align} 
	\mathcal{I}\colon [av_{U}]&\longrightarrow \ISpec(\Z)\\
	|\cdot| &\longmapsto \{n \,\big|\, |n| <1\} \nonumber 
	\end{align}
	is well-defined, and can be easily checked to send both trivial and non-trivial non-Archimedean absolute values to their corresponding prime ideals. 
	
\end{discussion}

Discussion~\ref{dis:ISpec} thus justifies the following convention.

\begin{convention} Hereafter, the ``space of primes of $\Z$'' means $\ISpec(\Z)$, the coZariski spectrum.
\end{convention}



\subsubsection{Sheaves on Localic Spaces}\label{sec:etale-bundles} Thus far, we have primarily viewed localic spaces as point-free spaces in the sense of Definition~\ref{def:ptFREEspace}. We now turn to their complementary description as categories of sheaves. Since we shall
move freely between these two perspectives, let us fix some notation.

	\begin{convention}\label{conv:TvsST} Let $\thT$ be a geometric theory. We write $\baseS[\thT]$ for its classifying topos, viewed as a category of objects, and $[\thT]$ for its corresponding space of models. 
	Thus
		\[
		F\in\baseS[\thT]
		\]
		denotes an object of the classifying topos (i.e. a sheaf), whereas
		\[
		x\in[\thT]
		\]
		denotes a point (i.e.\ a geometric morphism into $\baseS[\thT]$).
\end{convention}

When we restrict to localic spaces, their sheaves have a particularly nice form: namely, as \'{e}tale bundles.

\begin{definition}[\'{E}tale Bundles]\label{def:etale} Let $f\colon Y\rightarrow X$ be a map of localic spaces. 
\begin{enumerate}[label=(\roman*)]
	\item We call $f$ an {\em open map} if its corresponding frame homomorphism 
	$$f^{-1}\colon\Omega_{X}\rightarrow\Omega_Y\,$$
	has a left adjoint $f_!$ satisfying the Frobenius reciprocity condition:
	\[f_!(U\land f^{-1}(V))=f_!(U)\land V,\qquad\text{for all}\, U\in\Omega_X,V\in\Omega_{Y}.\]
	\item We call $f$ \emph{\'{e}tale} if the maps $f$ and its diagonal $\Delta\colon Y\to Y\times_X Y$ are both open. We call any \'{e}tale map with codomain $X$ an \emph{\'{e}tale bundle on $X$}. 
\end{enumerate}
\end{definition}

\begin{discussion}\label{dis:fibrewiseDISCRETE} Informally, a bundle $p\colon Y\to X$ is just a map but thought of in the opposite direction: for each $x\in X$, we define a fibre space $p^{-1}(x)$ over $x$.\footnote{In point-set topology, one also requires the fibres to vary continuously over the base-point, which must be checked separately. In point-free topology, continuity is automatic so long as one works geometrically. This is because if the generic fibre $p^{-1}(x)$ is constructed geometrically, then it is preserved by pullback and thus extends continuously over all points in $X$; see \cite[\S 9]{VickersPtfreePtwise}.} In the special case $X=\{*\}$, the canonical projection $f\colon Y\to \{*\}$ is \'{e}tale iff $Y$ is a discrete space \cite[Theorem V.5.1]{JoyalTierney}.  This gives rise to a useful slogan: 
$$\text{\'{E}tale bundles} = \text{Fibrewise discrete bundles}.$$
\end{discussion}

One can also characterise sheaves of (generalised) spaces via the so-called {\em object classifier}.

\begin{definition} The {\em theory of objects}, denoted $\mathbb{O}$, has only one sort, with no function symbols, predicates or axioms. Its defining property is that, for any topos $\E$, models of $\mathbb{O}$ in $\E$ are precisely the objects of $\E$. We denote its corresponding space as $[\mathbb{O}]$, and call it the \emph{object classifier}.\footnote{Warning: not to be confused with the \emph{subobject classifier}, which is a single \emph{object} living in each topos $\E$. By contrast, the \emph{object classifier} $[\mathbb{O}]$ is a generalised {\em
	space} parametrising {\em all} objects in {\em all} toposes.} 
\end{definition}

For context and later quotation, the following fact summarises several equivalent characterisations of sheaves on a localic space. 

\begin{fact}\label{fact:sheaves} Let $X$ be a localic space. A sheaf $F\in\baseS X$ admits the following equivalent characterisations:
	\begin{enumerate}[label=(\roman*)]
		\item A functor $F\colon \Omega_X^{\opp}\to \Set$ subject to the standard gluing conditions; 
		\item A  point-free map $F\colon X\to \OC$ from $X$ to the object classifier.
		\item An \'{e}tale bundle $f\colon Y\to X$.
		\item A local homeomorphism $f\colon Y\to X$.
	\end{enumerate}
\end{fact}
\begin{proof} \hfill 
	\begin{itemize}
		\item[(i):] This is the standard definition -- see e.g. \cite[Example A.2.1.8]{Elephant}.
		\item[(i)$\iff$ (ii):] 	Denote $\mathrm{FinSet}$ as (a small skeleton of) the category of finite sets. By  \cite[Example B3.2.9]{Elephant}, the functor category $[\mathrm{FinSet},\Set]$ satisfies the following property
		\[\textbf{Geom}(\E, [\mathrm{FinSet},\Set])\simeq \E\]
		for any topos $\E$. Moreover, $[\mathrm{FinSet},\Set]$ is equivalent to the classifying topos of $\mathbb{O}$ \cite[Example 1.74]{Vi07}. Thus an object of $\baseS X$ is equivalently a point-free map $X\to[\mathbb{O}]$, justifying the name ``object classifier''.\footnote{For those interested in the topos-validity of our analysis: any 2-category $\mathfrak{BTop}/\baseS$ of bounded toposes over $\baseS$ has an object classifier so long as $\baseS$ is an elementary topos with natural number object \cite[Theorem B4.2.11]{Elephant}.}
		\item[(i)$\iff$ (iii):] This is \cite[Proposition VI.3.3]{JoyalTierney}.
			\item[(iii)$\iff$ (iv):] See \cite[Proposition 1.49]{Vi07}; for the relevant definition of local homeomorphism in this context, either read the statement of the cited Proposition or see \cite[C1.5]{Elephant}. 
	\end{itemize}
\end{proof}

Since \'{e}tale bundles are fibrewise discrete (Discussion~\ref{dis:fibrewiseDISCRETE}), one easily checks that Fact~\ref{fact:sheaves} implies: 

\begin{corollary}\label{cor:Set} $\baseS\{\ast\}=\Set$.
\end{corollary}

Finally, we apply the Lifting Lemma~\ref{lem:liftingsheaves} to obtain yet another characterisation of the sheaves on the upper real interval $\overleftarrow{[0,1]}$. This will later play a key role in our analysis of the Archimedean place. 

\begin{observation}\label{obs:CHARof[0,1]SHEAVES} Let $F\in\baseS\overleftarrow{[0,1]}$. 	\underline{Then}, $F$ can be equivalently characterised as:
	\begin{enumerate}[label=(\roman*)]
		\item $F$ is a sheaf on $\overleftarrow{[0,1]}$;
		\item  $F\colon \overleftarrow{[0,1]}\rightarrow[\mathbb{O}]$; 
		\item $F\colon\qint\rightarrow[\mathbb{O}]$ is a map satisfying the continuity conditions of Lemma~\ref{lem:liftingsheaves}. 
	\end{enumerate}
\end{observation}
\begin{proof} The equivalence of (i) and (ii) is Fact~\ref{fact:sheaves}. For (ii)$\iff$(iii), use the isomorphism $\overleftarrow{[0,1]}\cong\RIdl(\qint,\prec)$ from Example~\ref{ex:ridlreals}, together with the Lifting Lemma~\ref{lem:liftingsheaves}.  
\end{proof}

\subsection{Descent in Topos Theory}\label{sec:Des}  By way of motivation, 
consider a discrete set $X$ and a discrete group $G$ acting upon it. Classically, the quotient is obtained by identifying the set of points lying in the same orbit:
\[X/G := \{\mathrm{Orb}(x)\,|\, x\in X\}.\]
For point-free spaces, a direct translation is complicated by the fact that generalised points live in different toposes. One could attempt to recover the above picture by restricting to just the classical points (i.e. the $\Set$-based models), but these are in general insufficient to capture the geometry of the entire space -- and in fact, may not exist at all. 

Descent techniques supply the appropriate analogue: from a suitable diagram of spaces, they produce a new space realising the desired quotient together with its induced geometric structure. From the  topos-theoretic perspective, this corresponds to constructing certain finite colimits in the 2-category of toposes $\mathfrak{Top}$; see \cite[B3.4]{Elephant}. This section recalls the standard and lax descent constructions, together with Moerdijk's Stability Theorem~\ref{thm:MoerdijkSTABILITY}. Much of the material is standard \cite{JoyalTierney,Moerdijk}, except that we recast the constructions in the language of point-free topology.

\begin{convention} Caligraphic letters $\E,\E'$ \dots denote generalised spaces, and $\baseS\E,\baseS\E'$ \dots denote their corresponding categories of sheaves (cf. Convention~\ref{conv:TvsST}). A point-free map of spaces
 $$f\colon \E\to\E'$$ corresponds to, under the point-free space/topos correspondence, to a geometric morphism whose inverse image functor we denote by
	$$f^{\ast}\colon \baseS\E'\longrightarrow\baseS\E\,.$$ 
Note the reversal of arrows, mirroring the frame homomorphisms in Section~\ref{subsec:prop-th-frame}.\footnote{See \cite{VickersPtfreePtwise} for details, in particular the discussion on why classifying toposes may be regarded as generalised frames (what Vickers calls ``Giraud frames'').}	When restricting to localic spaces, we use standard capital letters $X,Y,\ldots$ and $\baseS X,\baseS Y$ for their corresponding categories of sheaves. We write $\Loc$ for the category of localic spaces.
\end{convention}

\subsubsection{Standard Descent} A \textit{2-truncated simplicial space} $\E_{\bullet}$ is a diagram of (generalised) spaces of the form
\begin{equation}\label{eq:2truncSIMPspace}
\begin{tikzcd}\mathcal{E}_2 \ar[r, bend right, shift right=1ex, swap, "\widehat{d}_2"] \ar[r,swap,"\widehat{d}_1"] \ar[r,bend left, shift left=1ex,  "\widehat{d}_0"]& \mathcal{E}_1 \ar[r, bend right,shift right=1ex, swap, "d_1"] \ar[r, bend left, shift left=1ex, "d_0"] & \mathcal{E}_0 \ar[l, "s_0"] \end{tikzcd}
\end{equation} 
We require that the maps in Equation~\eqref{eq:2truncSIMPspace} commute up to isomorphism.\footnote{For context on why things only commute up to isomorphism, see Discussion~\ref{dis:descentAUTOMORPHISMS}.}


\begin{construction}[Standard Descent]\label{cons:stdDESCENT} Given any 2-truncated simplicial space $\E_{\bullet}$, passing to the corresponding category of sheaves yields the diagram 
	\begin{equation}\label{eq:stdDESCENTsimpTOPOS}
	\begin{tikzcd}\baseS\E_2 & \baseS\E_1 \ar[l, bend right, shift right=1ex, swap, "\widehat{d}^*_0"] \ar[l,"\widehat{d}^*_1"] \ar[l,bend left, shift left=1ex,  "\widehat{d}^*_2"] \ar[r, swap,"s^*_0"] & \ar[l, bend right,shift right=1ex, swap, "d^*_0"] \ar[l, bend left, shift left=1ex, "d^*_1"] \baseS\E_0 \end{tikzcd}\,\,\,\,.
	\end{equation}
The {\em universal descent cocone} of Diagram~\eqref{eq:stdDESCENTsimpTOPOS} is a pair $(\Des,p^*)$, whereby:
	\begin{enumerate}
	\item $\Des$ is the {\em descent category}, defined by
	\begin{itemize}
		\item[] \textbf{Objects:} $(F,\theta)$, where $F$ is an object of $\baseS\mathcal{E}_0$
		and $\theta:d_0^*(F)\xrightarrow{\sim} d^*_1(F)$, also known as the \textit{descent data}, is an isomorphism satisfying the identities
		\begin{enumerate}[label=(\roman*)]
			\item \emph{(Unit Condition).}  $s^*_0(\theta)\cong \id$;
			\item \emph{(Cocycle Condition).}  $\widehat{d}^*_0(\theta)\circ \widehat{d}^*_2(\theta)\cong \widehat{d}_1^*(\theta)$.
		\end{enumerate}
	\smallskip
	
		\item[] \textbf{Morphisms:} $\alpha\colon (F,\theta)\rightarrow (F',\xi)$, where  $u\colon F\rightarrow F'$ is a morphism in $\baseS\mathcal{E}_0$ that is compatible with the descent data, i.e. $d_1^*(u)\circ\theta=\xi\circ d_0^*(u)$.
	\end{itemize}
	\item  $p^*$ is the forgetful functor
	\begin{align*}
	p^*\colon \Des&\longrightarrow \baseS\E_0\\
	(F,\theta)&\longmapsto F
	\end{align*}
\end{enumerate}
\end{construction} 

\begin{remark}\label{rem:univ-descent} The claim that $(\Des,p^*)$ defines the universal descent cocone of Diagram~\eqref{eq:stdDESCENTsimpTOPOS} follows from the standard theory of descent. Although $\Des$ as defined above is just a category, it is in fact a topos; see \cite[\S 3]{Moerdijk}. One then checks that $p^*$ is the inverse image functor of a geometric morphism $p\colon \calS\E_0\to \Des$, and that $(\Des,p)$ defines the  pseudo-colimit of Diagram~\eqref{eq:stdDESCENTsimpTOPOS} in $\mathfrak{Top}$ (the 2-category of toposes).\footnote{{\em Aside.} One should take care to distinguish here between the inverse image functors, which reverse the direction of the corresponding maps, and the geometric morphisms themselves. The latter are the morphisms in $\mathfrak{Top}$ with respect to which the colimit is formed, and hence point in the same direction as the corresponding point-free maps.} 
\end{remark}

\begin{remark} By Remark~\ref{rem:univ-descent}, $\Des$ defines a topos; it therefore classifies a geometric theory. Hence, translating Construction~\ref{cons:stdDESCENT} back into the language of point-free topology, we obtain a cocone 
	\begin{equation}
	p\colon \mathcal{E}_{\bullet}\longrightarrow [\thT_{\Des}],
	\end{equation}
	where $[\thT_\Des]$ denotes the point-free space corresponding to the descent topos $\Des$. Informally, $[\thT_{\Des}]$ may be viewed as the coequaliser of $d_0$ and $d_1$, subject to the additional descent conditions encoded by the simplicial structure.
	
\end{remark}

Two main examples of Construction~\ref{cons:stdDESCENT} will be important for this paper. The first revisits our motivating example of quotienting a space by a group action.

\begin{example}[Descent Topos of a Groupoid]\label{ex:groupoidDESCENT} Let $G:=(G_0,G_1)$ be a groupoid in $\Loc$, with domain, codomain, unit and multiplication maps
	\begin{equation}\label{eq:groupoidEXAMPLE}
	\begin{tikzcd} G_1\times_{G_0}G_1 \ar[r, bend right, shift right=1ex, swap, "\pi_0"] \ar[r,swap,"\m"] \ar[r,bend left, shift left=1ex,  "\pi_1"]& G_1 \ar[r, bend right,shift right=1ex, swap, "d_1"] \ar[r, bend left, shift left=1ex, "d_0"] & G_0 \ar[l, "s"]  
	\end{tikzcd}
	\end{equation}
Here the pullback is taken along $d_0$ and $d_1$; such pullbacks always exist \cite[Theorem B4.1.1]{Elephant}. The groupoid axioms ensure that $G$ determines a 2-truncated simplicial space, and hence Construction~\ref{cons:stdDESCENT} produces a descent topos, which we denote $\BG$. 

\end{example}

\begin{remark} Notice the colimit here is taken in the 2-category of toposes $\mathfrak{Top}$, and {\em not} in $\Loc$ (or in the corresponding category of localic toposes). In particular, although $G$ is a diagram of localic spaces, the descent topos itself may not be localic; see Theorem~\ref{thm:BG}.
\end{remark}

\begin{discussion} The point-free space $[\thT_{\BG}]$ corresponding to $\BG$ is equipped with a canonical cocone from the simplicial Diagram~\eqref{eq:groupoidEXAMPLE}, and in particular coequalises the maps $d_0$ and $d_1$. In the language of (generalised) points, let $g$ be a point of $G_1$ and write
	$$x:= d_0(g)\qquad gx:=d_1(g)\,.$$
In this sense, $g$ may be regarded as an arrow from $x$ to $gx$. The multiplication map
$$\m\colon G_1\times_{G_0}G_1\longrightarrow G_1$$
encodes composition of such arrows, while the unit and inverse maps give the identity and inverse arrows. Consequently, the coequaliser $[\thT_\BG]$ identifies the source and target of every arrow: $p(x)\cong p(gx)$. In this sense, $[\thT_\BG]$ may be viewed as the universal space in which the $G_1$-action on $G_0$ is identified, and thus as the quotient space of $G_0$ by the $G_1$-action.
\end{discussion}

The second example demonstrates how descent isolates important structure-preserving properties of point-free maps between spaces. 

\begin{example}[\v{C}ech groupoid]\label{ex:DescentGEOMMORPHISM} Let $\Phi\colon \E'\rightarrow \E$ be a map of (generalised) spaces. Its {\em \v{C}ech groupoid} is the diagram below constructed via iterated pullbacks
	\begin{equation}\label{eq:PHIiteratedPULLBACKS} \begin{tikzcd} \E'\times_{\E} \E'\times_{\E} \E' \ar[r, bend right, shift right=1ex, swap,"\pi_{12}"] \ar[r,"\pi_{02}"] \ar[r,bend left, shift left=1ex,  "\pi_{01}"]& \E'\times_\E \E' \ar[r, bend right,shift right=1ex, swap, "\pi_1"] \ar[r, bend left, shift left=1ex, "\pi_0"] &\E' \ar[l, swap, "\Delta"] \ar[r,"\Phi"] & \E
	\end{tikzcd}\,\,.
	\end{equation}
Here, the diagonal map $\Delta\colon \E'\rightarrow \E'\times_\E \E'$ sends $x\mapsto (x,x)$; the remaining arrows are the evident projection maps, e.g. $\pi_{01}(x_0,x_1,x_2)=(x_0,x_1).$ Diagram~\eqref{eq:PHIiteratedPULLBACKS} is clearly a $2$-truncated simplicial space -- in fact, a groupoid, with the inverse map given by interchanging the two projections. Applying Construction~\ref{cons:stdDESCENT} therefore gives a descent category, which we denote by $\Des(\Phi)$.
\end{example}

\begin{definition}[Effective Descent]\label{def:eff-descent} Continuing with Example~\ref{ex:DescentGEOMMORPHISM}, the inverse image functor
	\begin{equation}
	\Phi^*\colon \baseS\E\rightarrow \baseS\E'
	\end{equation}
induces a comparison functor $\chi\colon \baseS\E\rightarrow \Des(\Phi)$ such that the diagram below commutes
	\begin{equation}
	\begin{tikzcd}
	\baseS\E \ar[rr,"\chi"] \ar[dr,swap,"\Phi^*"] && \Des(\Phi) \ar[dl,"U"]\\
	& \baseS\E'
	\end{tikzcd}\,,
	\end{equation}
where $U$ is the forgetful functor $U(F,\theta)=F$. 
 We say that $\Phi$ satisfies \emph{effective descent}  when the induced functor $\chi$ is an equivalence.
\end{definition}

The usefulness of Definition~\ref{def:eff-descent} lies in the fact that the abstract descent topos can be hard to identify in general. When $\Phi$ satisfies effective descent, the descent topos is equivalent to a topos whose structure we already understand, thereby streamlining the analysis. This motivates a fundamental theorem of Joyal--Tierney, which identifies a broad class of morphisms for which effective descent is guaranteed.
\begin{theorem}[{{\cite[Theorem VIII.2.1]{JoyalTierney}}}]\label{thm:JTopenSURJ} Open surjections of toposes are effective descent morphisms. 
\end{theorem}

\begin{remark}\label{rem:openSURJ} For later reference, let us translate Theorem~\ref{thm:JTopenSURJ} to the localic setting. A localic map $f\colon Y\to X$ is an open surjection precisely when the corresponding frame homomorphism 
$$f^{-1}\colon \Omega_X \to \Omega_Y$$
is injective and admits a left adjoint satisfying the Frobenius reciprocity condition of Definition~\ref{def:etale}. By \cite[Prop.~C1.5.1 and Theorem~C1.5.4]{Elephant}, this frame-theoretic characterisation is equivalent to the associated geometric morphism being an open surjection.
\end{remark}

Remarkably, the reach of these ideas goes considerably further. An important discovery of Joyal--Tierney shows that descent provides a presentation of arbitrary toposes in terms of localic groupoids. Recall that an open localic groupoid $G$ is one whose $d_0,d_1$ maps are open (in the sense of Definition~\ref{def:etale}). 

\begin{theorem}\label{thm:BG} For any topos $\calE$, there exists an open localic groupoid $G$ such that 
	$$\E\simeq \BG\,.$$
\end{theorem}
\begin{proof} Joyal-Tierney prove in \cite[Theorem VIII.3.2]{JoyalTierney} that every topos equivalent to the category of \'{e}tale $G$-spaces for some open localic groupoid $G$. On the other hand, Moerdijk \cite[\S 4.2 and \S 5.2]{Moerdijk} identifies the category of \'{e}tale $G$-spaces with the descent category $\BG$.
\end{proof}

\begin{discussion}[Descent and Automorphisms]\label{dis:descentAUTOMORPHISMS} Let us pause to appreciate Theorem~\ref{thm:BG}. It tells us \emph{all} toposes arise as the descent category of some open localic groupoid $G$. As remarked by Johnstone \cite[C5.1]{Elephant}, this resonates with an informal picture, dating back to Grothendieck's work on \'{e}tale cohomology of schemes, that a topos is ``a space whose points have enough internal structure to allow them to possess non-trivial automorphisms''.

The groupoid presentation makes this intuition precise at a structural level: $G_0$ records the underlying space, while $G_1$ records additional isomorphisms between its points. In the special case where the groupoid is the trivial groupoid on a locale $X$, there are no such isomorphisms, and $BG\simeq\baseS X$ is localic.\footnote{To see why, compare the characterisation of $\BG$ as the category of \'{e}tale $G$-spaces \cite[\S 5.2]{Moerdijk} with the characterisation of sheaves on a localic space as \'{e}tale bundles (cf. Fact~\ref{fact:sheaves}).} By contrast, if we consider a connected atomic topos $\baseS\E$ with a global point $q\colon \Set\to \calS\E$, then one can show that $\baseS\E\simeq \BG$, where $G$ is the localic group of automorphisms of the point $q$ \cite[Remark C5.2.14(c)]{Elephant}.
\end{discussion}

\subsubsection{Working Internally and Base-Change} The preceding constructions are compatible with base-change. Let $\Phi\colon X'\to X$ be a localic map, and write $\mathbf{Loc}(X)$ for the slice category of localic spaces over $X$. Pullback along $\Phi$ gives a functor \footnote{Establishing the existence of such pullbacks requires work: it cannot in general be obtained simply by applying the inverse image functor $\Phi^*$ to an internal locale $L$. Rather, one has to construct $\Phi^{\#}(L)$ explicitly, e.g. via the technology of frame presentations \cite[\S 1.6]{Moerdijk} or GRD-systems \cite[\S 5]{ViPowerlocaleEXP}.}  
$$\Phi^{\#}\colon\mathbf{Loc}(X)\longrightarrow\mathbf{Loc}(X').$$
Thus, if $G$ is an open localic groupoid internal to $\mathbf{Loc}(X)$, then pulling back its structure maps along $\Phi$ gives an open localic groupoid
 $\Phi^\#(G)$ internal to $\mathbf{Loc}(X')$. The descent construction of Example~\ref{ex:groupoidDESCENT} applies equally well in this internal setting; we write $B(\baseS X,G)$ for the resulting descent topos. Moerdijk's stability theorem then says that this construction commutes with base-change.

\begin{theorem}[{{Moerdijk's Stability Theorem \cite[Theorem 6.7]{Moerdijk}}}]\label{thm:MoerdijkSTABILITY} Let $\Phi\colon X'\to X$ be a map of localic spaces, and let $G$ be an open localic groupoid internal to $\mathbf{Loc}(X)$. \underline{Then}, there is a canonical equivalence
	\[B(\baseS X',\Phi^{\#}(G))\xrightarrow{\sim} \baseS X'\times_{\baseS X} B(\baseS X,G)\,\,.\]
\end{theorem}

In Moerdijk's original statement of Theorem~\ref{thm:MoerdijkSTABILITY}, $\mathbf{Loc}(X)$ is taken to be the category of internal locales in the topos $\baseS X$, but this can be shown to be equivalent to the slice category of localic spaces over $X$ \cite[C1.6]{Elephant}. This gives a concrete point-free interpretation of working internally to a topos, which we shall use throughout the paper. The following convention develops this point-free picture to identify an important methodological principle, already used in \cite{NV}.


\begin{convention}[``Fixing $x$'']\label{conv:fixingx} Suppose we wish to construct a map with multiple arguments, such as
	\[
	f\colon[\thT]\times [\thU]\rightarrow [\thU'].
	\]
We shall often say ``fix $x\in [\thT]$'', before constructing a map
	\[
	f_{x}\colon[\thU]\rightarrow [\thU']
	\text{.}
	\]
Formally, fixing $x$ means working over $[\thT]$, or equivalently, internally to the topos $\baseS[\thT]$. Thus $f_x$ corresponds to a morphism
$$\langle p,f\rangle\colon[\thT]\times[\thU]\longrightarrow[\thT]\times[\thU']$$
over $[\thT]$, where $p$ denotes the projection to $[\thT]$. Equivalently, it is a morphism in the slice category $\mathbf{Loc}([\thT])$:
\begin{equation}
\begin{tikzcd}
{[\thT]\times[\thU]}
\ar[rr, "{\langle p, f\rangle}"]
\ar[dr, swap, "p"]
&& {[\thT]\times[\thU']}
\ar[dl, "p"]
\\
& {[\thT]}
\end{tikzcd}.
\end{equation}
Further context can be found in \cite[\S 9]{VickersPtfreePtwise}.
\end{convention}

\subsubsection{Lax Descent} There is also an important variant of Construction~\ref{cons:stdDESCENT} known as {\em lax descent}.

\begin{construction}[Lax Descent]\label{cons:LAXdescent} Given a 2-truncated simplicial space $\mathcal{E}_{\bullet}$, the \emph{lax descent category} of $\mathcal{E}_\bullet$, which we denote $\LDes$, is defined exactly as $\Des$ of Construction~\ref{cons:stdDESCENT} except that the descent data $\theta$ is no longer required to be an isomorphism.
\end{construction}

The following remarks clarify the distinction between standard vs. lax descent by examining the significance of requiring the descent data to be an isomorphism. 

\begin{remark}[Coinserters vs. Coequalisers]\label{rem:coinsert} Let $\mathcal{C}$ be a category, and consider the following diagram in $\mathcal{C}$:
	\begin{equation}\label{eq:COEQvsCOINSERT}
	\begin{tikzcd}
	A  \ar[r,shift left=.75ex,"f"]\ar[r,shift right=.75ex,swap,"g"] & B
	\end{tikzcd} 
	\end{equation}
The \textit{coequaliser} of $(f,g)$, when it exists, gives the universal morphism $h\colon B\rightarrow C$ satisfying $hf=hg$. A weaker (and less familiar) construction is the \emph{coinserter} of $(f,g)$, which gives the universal morphism $h\colon B\rightarrow C$ equipped with a 2-cell $hf\rightarrow hg$. In particular, unlike the coequaliser, the 2-cell of the coinserter is not required to be invertible.
	
In our setting, the standard descent may be regarded as a (pseudo-)coequaliser in $\mathfrak{Top}$ subject to the relevant descent conditions, whereas lax descent realises the corresponding (pseudo-)coinserter. The two constructions therefore agree in some cases, but differ in general.
\end{remark}

\begin{discussion}[Standard Descent \& Group Completion]\label{dis:GroupoidReflection} Consider an internal category $\mathbb{M}$ in the topos $\Set$ which we represent as 
	\begin{equation}\label{eq:groupCOMPex}
	\begin{tikzcd} \mathbb{M}_2\ar[r, bend right, shift right=1ex] \ar[r] \ar[r,bend left, shift left=1ex]&  \mathbb{M}_1 \ar[r, bend right,shift right=1ex] \ar[r, bend left, shift left=1ex] & \mathbb{M}_0 \ar[l] \end{tikzcd}
	\end{equation}	
	where the arrows commute in the obvious way. We may therefore regard $\mathbb{M}_1$ as a discrete monoid acting on the set $\mathbb{M}_0$. Unpacking Remark~\ref{rem:coinsert}, the lax descent construction retains the monoid action, whereas standard descent forces the action maps to become invertible and thus replaces the monoid by its groupoid/group completion; see \cite[Example B3.4.14]{Elephant}. 
	
	This association of standard descent with group completion is suggestive, particularly because group completion sometimes entails a loss of information -- a well-known example is the monoid $\left(\mathbb{N}\cup\{\infty\},+\right)$, whose group completion trivialises since $n+\infty=\infty$ for every $n$.
	This alerts us to the same possibility when applying standard descent: namely, that requiring the descent data to be invertible may collapse important distinctions preserved by lax descent. 
	
	
\end{discussion}

\begin{remark}\label{rem:LaxTopos} It is natural to ask if $\LDes$ is a topos, just as in Construction~\ref{cons:stdDESCENT}? The answer is yes, although the result appears to be folklore; see \cite[Remark B3.4.10]{Elephant} for some details. In any case, we note that our Theorem~\ref{thm:thmD} gives a direct proof that our lax descent category of interest is indeed a topos.
\end{remark}


\section{The Global Picture}\label{sec:global-places} 

Section~\ref{sec:Intro} gave the number-theoretic context, and discussed the subtleties in reconciling the non-Archimedean and Archimedean aspects of the picture (Question~\ref{qn:TENSION}); Section~\ref{sec:prelim} laid the topos-theoretic groundwork for this paper, with particular emphasis on the point-free perspective. This section now brings these perspectives together by constructing the topos of places of $\Q$. 

\smallskip

Recall that an {\em absolute value} is a map
\begin{align}
|\cdot|\colon\mathbb{Q}&\longrightarrow [0,\infty)\\
x&\longmapsto |x| \nonumber 
\end{align}
that is multiplicative, positive definite (i.e. $|x|=0$ iff $x=0$) and satisfies the triangle inequality. This was axiomatised in \cite[\S 2]{NVOstrowski} by an essentially propositional geometric theory. Thus, following Definition~\ref{def:localic}, its corresponding point-free space, denoted $$[av]$$ 
is a localic space. Next, given two absolute values $|\cdot|_1,|\cdot|_2$, define an equivalence relation $\sim$ where 
\begin{equation}\label{eq:place-1}
|\cdot|_1\sim|\cdot|_2 \iff \exists \alpha\in(0,1]\,.\bigg( |x|^{\alpha}_1=|x|_2 \,\,\text{or}\,\, |x|^{\alpha}_2=|x|_1, \,\text{for all non-zero}\, x\in\mathbb{Q} \bigg). 
\end{equation}
An equivalence class with respect to $\sim$ is called a \textit{place}. As it turns out, two absolute values belong to the same place if and only if they yield homeomorphic completions of $\Q$.\footnote{For the constructive mathematician: this is not only true classically, but geometrically as well, but we omit the details here.} Moreover, notice that exponentiation satisfies the obvious unit and associativity identities
\begin{equation}\label{eq:place-2}
|\cdot|^1=|\cdot|\qquad\text{and}\qquad (|\cdot|^\alpha)^\lambda = |\cdot|^{\alpha\cdot\lambda}\qquad\text{for any}\,\alpha,\lambda\in (0,1]\,\,.
\end{equation}

Let us reformulate this data as a 2-truncated simplicial space:

\begin{equation}\label{eq:PLACES-topos}
\begin{tikzcd} (0,1] \times (0,1]\times [av] \ar[r, bend right, shift right=1ex, swap, "\pi_{12}"] \ar[r,swap,"\m\times\id"] \ar[r,bend left, shift left=1ex,  "\pi_{02}"]& (0,1]\times  {{[av]}}  \ar[r, bend right,shift right=1ex, swap, "\exp"] \ar[r, bend left, shift left=1ex, "\pi"] & {{[av]}} \ar[l, "s"] \end{tikzcd} \,\,, 
\end{equation}
with $(0,1]$ as the obvious interval of Dedekind reals. The projection map $\pi$ sends $(\alpha,|\cdot|)\mapsto |\cdot|$, while the map $\exp$ sends $(\alpha,|\cdot|)\mapsto|\cdot|^\alpha$. The unit map is given by $s(|\cdot|)=(1,|\cdot|)$, while $\pi_{02},\pi_{12}$ are the obvious projections. Finally, the map $\m\times\id$ records composition of exponentiations, so that
$$
(\alpha,|\cdot|),(\lambda,|\cdot|^\alpha)
\longmapsto
(\alpha\lambda,|\cdot|)\,\,.
$$
The unit and associativity identities~\eqref{eq:place-2} ensure that these maps satisfy the required simplicial identities.
\smallskip

Let us briefly verify Diagram~\eqref{eq:PLACES-topos} is well-defined. Both the absolute values {\em and} the interval of Dedekind reals $(0,1]$ are axiomatisable by a geometric theory, and thus each determines a point-free space by the Structure Theorem~\ref{thm:structure-thm}. Their products and pullbacks exist in the category of point-free spaces by \cite[Theorem B4.1.1]{Elephant}. Finally, the exponentiation map 
$$\exp\colon(0,1]\times[av]\longrightarrow[av],
\qquad
(\alpha,|\cdot|)\longmapsto|\cdot|^\alpha,$$
is a point-free map because real exponentiation is geometric (Theorem~\ref{thm:xzetaDed}). Combined with the simplicial identities given by Equation~\eqref{eq:place-2}, conclude that Diagram~\eqref{eq:PLACES-topos} indeed defines a 2-truncated simplicial space. Consequently, and importantly for us, its lax descent topos exists. 

\begin{definition}[Topos of Places]\label{def:places} Let $\calP$ denote the lax descent topos associated to Diagram~\eqref{eq:PLACES-topos}. We call $\calP$ the {\em topos of places}, and write 
	$$[\mathrm{places}]$$
for its corresponding point-free space. 
\end{definition}

We now have the existence of a topos of places, giving the beginnings of an answer to Question~\ref{qn:TENSION} from the topos-theoretic perspective. The next action item would be to characterise its corresponding point-free space of places. Sections~\ref{sec:NonArch} and \ref{sec:Arch} analyse its Archimedean and non-Archimedean aspects separately, before comparing their sheaf structure in Section~\ref{sec:ARCHvsNONARCH}.

\section{Non-Archimedean Places}\label{sec:NonArch}

We now analyse the non-Archimedean places. Following the conventions of \cite{NVOstrowski}, an absolute value 
$$|\cdot|\colon\Q\to [0,\infty)\,$$
is called {\em non-Archimedean} if there exists a (unique) prime $p\in\N_+$ such that $|p|<1$.\footnote{Notice this excludes the trivial absoute value.} Applying Ostrowski's Theorem~\ref{thm:ostrowskiQ}, we obtain the following basic observation.

\begin{observation}\label{obs:stillnonARCH} Fix some prime $p$. Let $|\cdot|$ be an non-Archimedean absolute value such that $|p|<1$.
	\begin{enumerate}[label=(\roman*)]
		\item For every $\alpha\in(0,\infty)$, the absolute value $|\cdot|^\alpha$ is again non-Archimedean, with $|p|^\alpha<1$. In particular, exponentiation defines a group action of $(0,\infty)$ on the space 
	$$	[av_{NA};p]$$
		of non-Archimedean absolute values satisfying $|p|<1$.
		\item There is a unique $\alpha\in(0,\infty)$ such that $$|\cdot|=|\cdot|_p^\alpha\,,$$ where $|\cdot|_p$ is the standard $p$-adic absolute value. In particular, \[ [av_{NA};p]\cong(0,\infty). \]
		
	\end{enumerate}
\end{observation}

As written, Observation~\ref{obs:stillnonARCH} is standard elementary number theory;
we state it explicitly to emphasise that Theorem~\ref{thm:ostrowskiQ} ensures that it also holds geometrically (in the sense of Convention~\ref{conv:geometricity}). This yields two main consequences that will be important for this section.
\begin{itemize}
	\item  Item (ii) of the Observation says the non-Archimedean absolute values naturally decompose into a family of subspaces $(0,\infty)$ indexed by the primes, where each $(0,\infty)$ parametrises all absolute values belonging to the same non-Archimedean place.  This justifies ``localising'' Diagram~\eqref{eq:PLACES-topos} and analysing the non-Archimedean places prime by prime. 
	\item Item (i) invites us to review the decision to define the topos of places using lax descent (Definition~\ref{def:places}). Globally, exponentiation gives a monoidal action of $(0,1]$ on absolute values, making lax descent the natural choice. Once we restrict to a single prime, however, exponentiation of absolute values extends to a group action of $(0,\infty)$. Thus we may replace lax descent by standard descent in the non-Archimedean case; see Discussion~\ref{dis:GroupoidReflection}. 
\end{itemize}

\subsection{Local: At single $\frap$}\label{subsec:localNA} Throughout this subsection, we fix the following hypothesis. 

\begin{hypothesis} Denote $\ISpec(\Z)_{\neq (0)}$ as \emph{the space of non-trivial prime ideals of $\Z$}, i.e. prime ideals that possess a non-zero integer. We shall fix a single $$\frap\in\ISpec(\Z)_{\neq (0)}.$$ 
By \cite[Lemma 1.14]{NVOstrowski}, this is equivalent to fixing a prime $p\in\N_+$, obtained by $\frap=(p)$.\footnote{We quote this lemma explicitly to indicate there exists a geometric proof of this elementary fact. We remark that $\frap=(p)$ does not come for free since the standard proof (see e.g. \cite[\S 3.7]{vdW1}) is not geometric -- see \cite[Remark 1.15]{NVOstrowski} for details.} 
	\end{hypothesis}

By Observation~\ref{obs:stillnonARCH}, fixing a prime $p$ identifies a single family of non-Archimedean absolute values 
$$[av_{NA};p]\cong (0,\infty)\,,$$
on which the group $(0,\infty)$ acts by exponentiation. Under the parametrisation
$$|\cdot|=|\cdot|^\beta_p\,,$$
this action is simply multiplication
$$|\cdot|_p^\beta\longmapsto
\bigl(|\cdot|_p^\beta\bigr)^\alpha
=|\cdot|_p^{\alpha\beta}.$$
This suggests the following ``localisation'' of Diagram~\eqref{eq:PLACES-topos}:

\begin{construction}\label{cons:groupoidG} Define the localic groupoid $N$ as
	\begin{equation}\label{eq:simplicialLOCALP}
	\begin{tikzcd} \Big((0,\infty)\times (0,\infty) \Big) \times (0,\infty)\ar[r, bend right, shift right=1ex, swap, "\pi_{12}"] \ar[r,swap,"\m\times\id"] \ar[r,bend left, shift left=1ex,  "\pi_{02}"]& (0,\infty)\times (0,\infty) \ar[r, bend right,shift right=1ex, swap, "\M"] \ar[r, bend left, shift left=1ex, "\pi"] & (0,\infty) \ar[l, "s"] \end{tikzcd}\,\,.
	\end{equation}
Here $\pi(\alpha,\beta)=\beta$ is the target map and
$\M(\alpha,\beta)=\alpha\cdot \beta$ the source map. The unit is
$s(\beta)=(1,\beta)$, while $\m,\pi_0,\pi_1$ are the multiplication and projection maps implementing composition. We call $N$ the {\em action groupoid} associated to the localic group $(0,\infty)$ acting on the space $(0,\infty)$.
\end{construction}

Construction~\ref{cons:groupoidG} leads to the following key definition, and sets up the main test problem of this section.

\begin{definition}[Non-Archimedean Place]\label{def:STDdescent} Define $\mathcal{D}:=\mathrm{BN}$ to be the descent topos corresponding to the groupoid $N$ in Construction~\ref{cons:groupoidG}. Call $\D$ the \emph{topos of the non-Archimedean place associated to prime $p$}. 
\end{definition}

\begin{problem}\label{prob:STDdescent} Give a useful characterisation of $\D$.
\end{problem}


\begin{discussion}\label{dis:DefNA} Some orienting remarks. On the topos-theoretic side, Construction~\ref{cons:stdDESCENT} ensures that $\D$ is well-defined and is in fact a topos. On the number-theoretic side, the space $(0,\infty)$ parametrises the non-Archimedean absolute values associated to the prime $p$. The action groupoid $N$ records the fact that these absolute values are identified under the $\M$-action by $(0,\infty)$, and hence belong to a single place. Thus the descent topos $\D$ may be regarded as a point-free representation of the corresponding non-Archimedean place, justifying the terminology above.
\end{discussion}

In what follows, we work to improve our understanding of $[\thT_\D]$, the (quotient) space associated to the topos $\D$. To start, observe that $G_1$ represents a free transitive action of $(0,\infty)$ on $G_0$, and so there only exists a single $G_1$-orbit on $G_0$. Combined with the previous discussion of toposes as generalised spaces whose points may carry non-trivial automorphisms (Discussion~\ref{dis:descentAUTOMORPHISMS}), the following guess is reasonable.

\begin{guess}\label{guess:DhasAUTOMORPHISMS}  $[\thT_\D]$ is the singleton space $\{*\}$, with $(0,\infty)$ as the group of automorphisms acting on $\{*\}$.
\end{guess}

Very interestingly, Guess~\ref{guess:DhasAUTOMORPHISMS} turns out to be wrong. The fundamental reason behind this has to do with the misplaced expectation that  the quotient space ought to possess non-trivial automorphisms. As a baseline observation: non-triviality of the $G_1$-action does {\em not} imply non-triviality of the quotient space -- for instance, $\BG\simeq \Set$ for any connected localic group $G$ (see Example~\ref{ex:Bunge}). A similar issue arises in our setting: 

\begin{theorem:SINGLEPRIME}	$\mathcal{D}\simeq\Set$. Or, equivalently, $[\thT_\D]\cong \{*\}$.
\end{theorem:SINGLEPRIME}


In other words, the quotient space $[\thT_\D]$ is trivial: it is the singleton $\{*\}$ with \emph{no} non-trivial automorphisms. Comparing our groupoid $G$ with the previous example, one might then suspect that the connectedness of $(0,\infty)$ is the main culprit behind the trivialisation, but this is again a red herring.\footnote{The original proof of Theorem~\ref{thm:DESCENTsinglePRIME} did in fact rely on the observation that $(0,\infty)$ is a locally-connected connected space, but that approach required a fairly lengthy analysis of the sheaves on $M$. Some of this analysis has been repurposed by Section~\ref{sec:ARCHvsNONARCH}, which compares the sheaves of the Archimedean and non-Archimedean topos.} In fact, Theorem~\ref{thm:DESCENTsinglePRIME} follows from a more general result:

\begin{theorem}\label{thm:TRIVIALfreeTRANSaction} Let $G$ be a localic group acting on a localic space $M$. Consider its {\em action groupoid}, denoted $H$, depicted below with the obvious maps:
	\begin{equation}\label{eq:groupoidFREETRANSITIVE}
	\begin{tikzcd}  G\times G\times M \ar[r, bend right, shift right=1ex, swap, "\pi_{12}"] \ar[r,swap,"\m\times\id"] \ar[r,bend left, shift left=1ex,  "\pi_{02}"] & G\times M \ar[r, bend right,shift right=1ex, swap, "\M"] \ar[r, bend left, shift left=1ex, "\pi"] & M \ar[l, "s"] \end{tikzcd}\,.
	\end{equation}
Now suppose:
\begin{itemize}
	\item  The unique map $\rho\colon M\rightarrow \{*\}$ is an open surjection; and 
	\item $G$ induces a free transitive action on $M$, in the sense that 
\begin{align}\label{eq:free-trans-action}
\lranglet{\pi}{\M}\colon G\times M&\longrightarrow M\times M \\
(g,m)&\longmapsto(m,g\cdot m)\,\,\,, \nonumber 
\end{align}
	is an isomorphism.
\end{itemize}
\noindent \underline{Then},
$$\mathrm{BH}\simeq \Set\,.$$
\end{theorem}

\begin{proof} Following Example~\ref{ex:DescentGEOMMORPHISM}, construct the \v{C}ech groupoid $V$ of $\rho\colon M\rightarrow \{*\}$ as below:
	\begin{equation}\label{eq:KERNELpairGROUPOID}
	\begin{tikzcd} M\times M \times M \ar[r, bend right, shift right=1ex, swap,"\pi_{12}"] \ar[r,"\pi_{02}"] \ar[r,bend left, shift left=1ex,  "\pi_{01}"]& M\times M \ar[r, bend right,shift right=1ex, swap, "\pi_0"] \ar[r, bend left, shift left=1ex, "\pi_1"] & M \ar[l, swap, "\Delta"] \ar[r,"\rho"] & \{*\} 
	\end{tikzcd}\,.
	\end{equation}
Let $\mathrm{BV}$ denote the descent topos associated to this (localic) groupoid. By hypothesis, $M\xrightarrow{\rho}\{*\}$ is an open surjection, and thus satisfies effective descent by Theorem~\ref{thm:JTopenSURJ}. Hence, deduce that
\begin{equation}\label{eq:BV=Set}
\mathrm{BV}\simeq \Set\,.
\end{equation} 
	
On the other hand, given any localic groupoid $L:=(L_0,L_1)$, its corresponding descent topos $\mathrm{BL}$ may be equivalently characterised as the category of \'{e}tale $L$-spaces, whose objects are \'{e}tale bundles $E\xrightarrow{p} L_0$ equipped with an $L_1$-action
$$L_1\times_{L_0} E \xrightarrow{\bigcdot} E$$ 
satisfying the usual axioms; see \cite[\S 5]{Moerdijk} for details. Applying this characterisation to our localic groupoids $V$ and $H$, we see that the objects of $\mathrm{BV}$ and $\mathrm{BH}$ are both étale bundles
$$E\xrightarrow{p}M,$$
equipped, respectively, with an $M\times M$-action and a $G\times M$-action.


By hypothesis~\eqref{eq:free-trans-action}, the free transitive action of $G$ on $M$ yields an isomorphism of spaces
$$\lranglet{\pi}{\M}\colon G\times M \xrightarrow{\cong} M\times M\,.$$
A routine check shows that this isomorphism is compatible with the groupoid structure. [For example, to see that  $\lranglet{\pi}{\M}$ is compatible with the source and target maps, notice it makes the inner and outer triangles of the diagram below commute
\begin{equation}\label{eq:Step2Diagram}
\begin{tikzcd}
G\times M \ar[rr,dashed] \ar[dr, shift right = 2, swap, "\M"]\ar[dr,shift left=0.5,"\pi"] && M\times M  \ar[dl, shift right = 0.5, swap, "\pi_0"]\ar[dl,shift left=2,"\pi_1"]\\
& M 
\end{tikzcd}\,\,.
\end{equation}
since its definition yields the identities
	\[\pi(g,m)=m=\pi_0\circ \langle\pi,\M\rangle(g,m)\]
\[\M(g,m)=g\cdot m=\pi_1\circ \langle\pi,\M\rangle(g,m).\]
A bit more work shows the corresponding identities for the unit and multiplication maps.] In other words, the given isomorphism~\eqref{eq:free-trans-action} transports the $M\times M$ action on \'{e}tale-bundles over $M$ to a $G\times M$-action (and vice versa). This yields the equivalence of toposes 
\begin{equation}\label{eq:BV=BH}
 \mathrm{BV}\simeq \mathrm{BH}.
\end{equation}
Combined with the earlier equivalence~\eqref{eq:BV=Set}, we thus obtain 
$$\mathrm{BH}\simeq \Set\,.$$

\end{proof}

\begin{discussion}[On the hypothesis ``open"] The locale theorist may ask: why did we require $M\xrightarrow{\rho} \{*\}$ to be an \emph{open} surjection in Theorem~\ref{thm:TRIVIALfreeTRANSaction}? After all, as Borceux proves in \cite[Example 1.6.5c]{Borceux}, the unique map $L\rightarrow\{*\}$ is open for any given locale $L$. This appears to indicate that the additional hypothesis of openness is unnecessary.
	
The issue is that Borceux's proof is classical. Recall that $\Omega$, i.e. the frame of opens on $\{*\}$, corresponds to the frame of truth values. In his argument, Borceux identifies this with the classical frame of Boolean truth values $\Omega=\{\bot,\top\}$, a move which implicitly uses Law of Excluded Middle. In particular, while it is classically true that the unique map $L\to \{\ast\}$ is open for any locale $L$, constructively this is not the case; see \cite[\S 6.2]{ViLocCompII}. Since we want our result to be topos-valid, this justifies our original hypothesis.

\end{discussion}


As an application of Theorem~\ref{thm:TRIVIALfreeTRANSaction}, we obtain the more quotable result:

\begin{theorem}\label{thm:DESCENTsinglePRIME}  $\D\simeq \Set$. Or, equivalently, $[\thT_\D]\cong \{*\}$.
\end{theorem}

\begin{proof} We emphasise that our proof is geometric (in the sense of Convention~\ref{conv:geometricity}). Recall that $\D$ is the descent category associated to the action groupoid $N$ from Construction~\ref{cons:groupoidG}. Examining the hypotheses of Theorem~\ref{thm:TRIVIALfreeTRANSaction}, it thus suffices to show that:
	\begin{enumerate}[label=(\alph*)]
		\item The $\M$-action of $N$ is both free and transitive;
		\item The unique map $!\colon (0,\infty)\rightarrow \{*\}$ is an open surjection.
	\end{enumerate}

For Condition (a), construct the inverse to the canonical map $\lranglet{\pi}{\M}$ as below:
	\begin{align*}\label{eq:Step2DiagramINVERSE}
\nu\colon (0,\infty)\times (0,\infty) &\longrightarrow (0,\infty)\times (0,\infty)\\
	(x,y)&\longmapsto(yx^{-1},x).
	\end{align*}
A straightforward computation verifies that $\langle \pi,\M\rangle$ and $\nu$ are indeed inverses, yielding the desired isomorphism.
For Condition (b), we invoke Vickers' framework of generalised metric spaces. A localic space $Y$ is called {\em overt} if its unique map $!\colon Y\to \{\ast\}$ is open. Since $\R$ is the localic completion of a generalised metric space, it is overt by \cite[Corollary 6.9]{ViLocCompII}. Since $(0,\infty)\hookrightarrow\mathbb R$ is an open subspace, and overtness is preserved by open subspaces, $(0,\infty)$ is overt as well. Stated more plainly, the map $!\colon (0,\infty)\to\{*\}$ is indeed open. Finally, $!$ is a surjection since it admits a section
$$s\colon \{\ast\}\to (0,\infty)\qquad \{\ast\}\longmapsto 1\,.$$
Why? Notice $!\circ s=\id_{\{\ast\}}$ in $\Loc$, which yields
$s^{-1}\circ !^{-1}=\id_{\Omega_{\ast}}$ at the level of frame homomorphisms. That is, $!^{-1}$ is a split monomorphism in $\Frm$, and so $!$ defines a surjection as per Remark~\ref{rem:openSURJ}. 
\end{proof}

\subsection{Global: Over all $\frap$}\label{subsec:globalNA} Per the definitions of \cite{NVOstrowski}, an {\em ultrametric absolute value} $|\cdot|$ is one which satisfies the ultrametric inequality $|x+y|\leq \max\{|x|,|y|\}$. This includes the non-Archimedean absolute values (= the non-trivial ultrametric absolute values) as well as the {\em trivial} absolute value $|\cdot|_0$ given by $|x|=1$ for every $x\neq 0$.

By Theorem~\ref{thm:DESCENTsinglePRIME}, each ultrametric place (including the trivial place) corresponds to a singleton -- as classically expected. Combined with Discussion~\ref{dis:ISpec}, this suggests that the space of ultrametric places ought to correspond to $\ISpec(\Z)$. However, perhaps surprisingly, we presently do not know how to justify this identification geometrically. This has to do with topological subtleties in reconciling the trivial and non-trivial ultrametric absolute values, which was already an issue even before quotienting \cite[\S 6]{NVOstrowski}. Nonetheless, we can still recover the expected picture after removing the trivial place: the space of non-Archimedean places is precisely $\ISpec(\Z)_{\neq 0}$.

To see this, recall that action groupoid $N$ in Section~\ref{subsec:localNA} was constructed after fixing $\frap=(p)$. In the language of Convention~\ref{conv:fixingx}, this means we were working internally within the topos $\baseS(\ISpec(\Z)_{\neq (0)})$. To recover the entire space of non-Archimedean places, we externalise the descent construction. Start by constructing the pullback diagram of spaces
\begin{equation}\label{eq:pullbackISPEC}
\begin{tikzcd}
\Phi^\#(N) \ar[r] \ar[d] & N\ar[d]\\
\ISpec(\Z)_{\neq (0)} \ar[r,"\Phi"] & \{\ast\}
\end{tikzcd}
\end{equation}
where $\Phi\colon \ISpec(\Z)_{\neq (0)} \to \{\ast\}$ denotes the unique terminal map in $\Loc$. We define the topos of non-Archimedean places as $$\baseS[\mathrm{places}_{NA}]:=	\B(\baseS(\ISpec(\mathbb{Z})_{\neq (0)}), \Phi^\#(N)).$$ We shall regard this topos as the category of sheaves on some space $[\mathrm{places}_{NA}]$. The following theorem verifies that we obtain the expected characterisation.

\begin{theorem}\label{thm:ALLNAplace} $[\places_{NA}]\cong\ISpec(\mathbb{Z})_{\neq (0)}$.
\end{theorem}
\begin{proof} We divide the argument into two stages.
	
\subsubsection*{Step 1: Open Groupoid} We claim that the action groupoid $N$ of Construction~\ref{cons:groupoidG} is an open groupoid, i.e. its source and target maps are open. Consider the \v{C}ech groupoid of $\rho\colon (0,\infty)\to\{\ast\}$, as in Diagram~\eqref{eq:KERNELpairGROUPOID}. Its source and target maps are the two projections, arising from the pullback
\begin{equation}
\begin{tikzcd}
(0,\infty)\times (0,\infty) \ar[r,"\pi_1"] \ar[d,swap,"\pi_0"] & (0,\infty) \ar[d,"\rho"]\\
(0,\infty) \ar[r,"\rho"]& \{*\} 
\end{tikzcd}\,.
\end{equation}	
Since $\rho$ is open and open maps are stable under pullback \cite[\S V.4]{JoyalTierney}, deduce that $\pi_0,\pi_1$ must also be open. 

Next, by the free-and-transitive condition established by Theorem~\ref{thm:DESCENTsinglePRIME}, the map
$$\lranglet{\pi}{\M}\colon (0,\infty)\times(0,\infty)
\longrightarrow
(0,\infty)\times(0,\infty)$$
is an isomorphism. Moreover, examining Diagram~\eqref{eq:Step2Diagram}, this isomorphism satisfies the identities
$$\pi=\pi_0\circ \lranglet{\pi}{\M}  \qquad \M = \pi_1\circ  \lranglet{\pi}{\M}.$$
Since isomorphisms are open, and compositions of open maps are open, it follows that $\pi,\M$ are also open -- which are precisely the source and target maps of the action groupoid $N$.
	
\subsubsection*{Step 2: Base-change} Regard $N$ as a localic groupoid internal to the topos $\Set$. By Step 1, $N$ is an open localic groupoid. Hence, we can apply Moerdijk's Stability Theorem~\ref{thm:MoerdijkSTABILITY} to obtain the equivalence
$$	\B(\baseS(\ISpec(\mathbb{Z})_{\neq (0)}), \Phi^\#(N))\simeq \baseS(\ISpec(\mathbb{Z})_{\neq (0)})\times_{\Set} \B(\Set,N)\,.$$
	By Theorem~\ref{thm:DESCENTsinglePRIME}	, we know that $\B(\Set,N)\simeq \Set$. In the language of point-free space, this yields
$$	[\places_{NA}]\cong \ISpec(\Z)_{\neq (0)} \times \{\ast\} \cong\ISpec(\Z)_{\neq (0)}\,.$$
\end{proof}


\section{The Archimedean Place}\label{sec:Arch}

This section characterises the Archimedean place. Per the definitions of \cite{NVOstrowski}, an Archimedean absolute value is one which is {\em not} non-Archimedean; this technically includes the trivial absolute value. Nonetheless, we exclude the trivial absolute value by fixing the following hypothesis:

\begin{hypothesis} Unless stated otherwise, an ``Archimedean absolute value'' means a {\em non-trivial} Archimedean absolute value, i.e. an absolute value $|\cdot|$  such that  $|n|>1$ for some non-zero integer $n$. Denote $[av_A]$
as the space of non-trivial Archimedean absolute values. 	
\end{hypothesis}

The prototypical example of an Archimedean absolute value is the {Euclidean absolute value} $|\cdot|_\infty$, defined
$$|n|_\infty:= n\ \qquad\text{for all}\, n\in\N\,.$$
In fact, by Theorem~\ref{thm:ostrowskiQ}, we know that $$[av_A]\cong (0,1],$$
with $\beta\in(0,1]$ corresponding to the absolute value
$$|\cdot|=|\cdot|_{\infty}^\beta\,.$$
Thus exponentiation acts on the parameter $\beta$ by multiplication:
$$|\cdot|^\alpha=\left(|\cdot|_{\infty}^\beta\right)^\alpha=|\cdot|_{\infty}^{\alpha\cdot\beta}\,.$$
Hence, playing the same game as in the non-Archimedean case, we reformulate the algebraic action of exponentiation on $[av_A]$ as below.

\begin{construction}\label{cons:ArchLAXtopos} Define the following diagram in $\Loc$.
	
	\begin{equation}\label{eq:ArchLAXtopos}
	\begin{tikzcd} \Big((0,1]\times (0,1] \Big) \times (0,1]\ar[r, bend right, shift right=1ex, swap, "\pi_{12}"] \ar[r,swap,"\m\times \id"] \ar[r,bend left, shift left=1ex,  "\pi_{02}"]& (0,1]\times (0,1] \ar[r, bend right,shift right=1ex, swap, "\M"] \ar[r, bend left, shift left=1ex, "\pi"] & (0,1] \ar[l, "s"] \end{tikzcd}
	\end{equation}
	where the maps are defined analogously to those in Construction~\ref{cons:groupoidG}, with
$$	\pi(\alpha,\beta)=\beta,
\qquad
\M(\alpha,\beta)=\alpha\beta.$$

\end{construction}

The set-up is analogous to the non-Archimedean case, with one crucial difference: the action is not invertible. Indeed, the action cannot be extended to $\alpha>1$, since already for the Euclidean absolute value 
$$|1+1|_{\infty}^\alpha=2^\alpha>2=|1|_{\infty}^\alpha+|1|_{\infty}^\alpha.$$
In other words, $|\cdot|_\infty^\alpha$ fails the triangle inequality for every $\alpha>1$, and thus no longer defines an absolute value.

This technical distinction matters. As pointed out in Discussion~\ref{dis:GroupoidReflection}, the standard descent construction freely inverts the arrows of the groupoid, whereas here we wish to retain the non-invertibility of the $(0,1]$-action. We are thus led to use lax descent.

\begin{definition}[Archimedean Place]\label{def:LAXdescent} 
Define $\D'$ to be the lax descent category of Diagram~\eqref{eq:ArchLAXtopos}. Call $\D'$ the \emph{topos of the Archimedean place}. Denote $[\thT_{\D'}]$ to be the corresponding space of $\D'$.
\end{definition}

As in the non-Archimedean case, the diagram suggests regarding $[\thT_{\D'}]$ as a quotient of $(0,1]$ by the exponentiation action. It is thus natural to wonder if we get the same result as before, i.e. if $[\thT_{\D'}]$ also corresponds to the singleton space $\{*\}$. It does not. In fact, we get the following surprising result.

\begin{theorem:ARCHPRIME} $\D'\simeq\baseS\overleftarrow{[0,1]}$, or equivalently, $[\thT_{\D'}]\cong\overleftarrow{[0,1]}$.
\end{theorem:ARCHPRIME}

Aside from the obvious difference with the non-Archimedean case, why else might Theorem~\ref{thm:ARCHIMEDEANPLACE} be surprising? One answer would be the number-theoretic implications, which we postpone till Section~\ref{sec:strangewoods} for proper discussion. But even setting aside the number theory, Theorem~\ref{thm:ARCHIMEDEANPLACE} highlights some very interesting interactions between topology and algebra, made visible through the lax descent.



\begin{enumerate}[label=(\alph*)]
	\item The appearance of the upper reals is unexpected.\footnote{Although, in hindsight, perhaps less surprising once we step away from classical number theory and examine the lax descent construction by itself: the quotient converts actions by the monoid into 2-cells, which introduces the one-sidedness.} Informally, Theorem~\ref{thm:ARCHIMEDEANPLACE} says that if we quotient the real interval $(0,1]$ by the multiplicative action of the monoid $(0,1]$, then we kill off all the left Dedekind sections of the reals in $(0,1]$ --- something which is \textit{a priori} not obvious.
	\item Although we were careful to exclude the trivial absolute value --- note that we work with $(0,1]$ instead of the closed interval $[0,1]$ --- the resulting quotient space nevertheless extends to $\overleftarrow{[0,1]}$. This suggests that the (non-trivial) Archimedean place and trivial place cannot be definably separated.\footnote{Why? Recall that subspaces of upper reals must be closed under directed joins along the specialisation order (Section~\ref{subsec:localicreals}).} This raises interesting questions on how we should understand the generic Archimedean completion, especially since $\mathbb{Q}$ and $\mathbb{R}$ are clearly not homeomorphic.
\end{enumerate}

We prove Theorem~\ref{thm:ARCHIMEDEANPLACE} by constructing two mutually inverse functors
\begin{equation*}\label{eq:LAXinverseFUNCTORS}
\begin{tikzcd}
\J\colon\D'
\arrow[r, bend left=35] &
\baseS\overleftarrow{\text{$[0,1]$}}\colon\K
\arrow[l, bend left=35]
\end{tikzcd}\,.
\end{equation*}
Unlike the non-Archimedean case, whose analysis could be streamlined by powerful descent theorems, here we give an explicit analysis of which sheaves of $\baseS(0,1]$ are selected by the lax descent topos $\D'$. Although the details are technical, the analysis is (implicitly) guided an important distinction regarding the Archimedean vs. non-Archimedean case: \emph{$\D'$ witnesses non-trivial forking in the connected components of its sheaves whereas $\D$ does not}. Many of the arguments below can be understood as adjusting for this difference. We defer further discussion of this forking phenomena to Section~\ref{sec:ARCHvsNONARCH}; for now, we focus on establishing the proof of Theorem~\ref{thm:ARCHIMEDEANPLACE}.

\subsection{First Direction}\label{subsec:ElephantstDIRECTION} Here we construct the functor
\begin{align}
\J\colon \D' \longrightarrow & \, \baseS\overleftarrow{[0,1]}\\
\quad (F,\theta)  \mapsto & \quad  ? \qquad\nonumber 
\end{align}
by giving a concrete description of its action on objects. Fix an object $(F,\theta)\in\D'$. By Construction~\ref{cons:LAXdescent}, $F$ is a sheaf in the topos $\baseS(0,1]$ and $\theta$ is its descent data. As preparation for later analysis, it will be useful to replace this abstract description by more explicit presentations of both $F$ and $\theta$. 

\subsubsection*{Presentation of the sheaf.} By Fact~\ref{fact:sheaves},  $F$ may be equivalently characterised as:
\begin{itemize}
	\item  A sheaf over $(0,1]$; 
	\item  A map $F\colon(0,1]\rightarrow [\mathbb{O}]$ to the object classifier;
	\item An \'{e}tale bundle $$f\colon Y\rightarrow (0,1],$$
with fibre $F({\gamma})=f^{-1}(\gamma)$ over each $\gamma\in (0,1]$. In particular, each fibre $F(\gamma)$ defines a set since \'{e}tale bundles are fibrewise discrete (Discusion~\ref{dis:fibrewiseDISCRETE}).
\end{itemize}
Throughout this section, we shall move freely between these different descriptions of $F$, depending on convenience.

\subsubsection*{Presentation of the descent data.} The pullback of $f$ along $\pi$ and $\M$ gives
\[\begin{tikzcd} \pi^*(Y) \ar[r,"\phi"] \ar[d,"\delta"] & Y \ar[d,"f"]\\
(0,1]\times (0,1] \ar[r, "\pi"] & (0,1]
\end{tikzcd}\quad\quad \begin{tikzcd} \M^*(Y) \ar[r,"\phi'"] \ar[d,"\delta'"] & Y \ar[d,"f "]\\
(0,1]\times (0,1] \ar[r, "\M"] & (0,1]
\end{tikzcd}\]
which we represent as
\[\pi^*(Y)= Y\times (0,1]\qquad\qquad \M^*(Y)=\{(y,\alpha,\beta) \in Y\times (0,1]\times (0,1]\big|\, f(y)=\alpha\cdot  \beta)\}\,,\]
with the evident projections $\delta,\delta'$ to the base $(0,1]\times (0,1]$. 
Correspondingly, the descent data $$\theta\colon \pi^*(Y)\rightarrow\M^*(Y)$$ defines a bundle map 
\begin{equation*} 
\begin{tikzcd} Y\times (0,1] \ar[rr,"\theta"] \ar[dr,swap, "\delta"] && \M^*(Y) \ar[dl,"\delta'"]\\
& (0,1]\times(0,1]
\end{tikzcd}\,.
\end{equation*}
In particular, $\theta$ is determined by the first coordinate: 
\begin{align}\label{eq:theta-coordinate}
\theta\colon Y\times (0,1] &\longrightarrow \M^*(Y)\\
(y,\alpha)&\longmapsto (\theta_0(y,\alpha),\alpha,f(y)) \,\,.\nonumber 
\end{align}
Here the second and third coordinates are forced by the fact that $\theta$ commutes with the projections $\delta,\delta'$ to the base space: the point $(y,\alpha)\in Y\times (0,1]$ lies over $(\alpha,f(y))$ in $(0,1]\times (0,1]$  whereas $(y,\alpha,\beta)\in\M^*(Y)$ lies over $(\alpha, \beta)$. 
\smallskip

The first coordinate map $\theta_0$ of the descent data will be important for subsequent analysis. For clarity and later quotation, we record some useful observations about its descent structure. 
 
\begin{observation}\label{obs:theta-descent-prop} Let $\theta_0$ be the first coordinate map from Equation~\eqref{eq:theta-coordinate}. \underline{Then}, for $(y,\alpha)\in Y\times (0,1]$:
	\begin{enumerate}[label=(\roman*)]
		\item \textit{(Fibre Action). } $f(\theta_0(y,\alpha))=f(y)\cdot\alpha$.
		\item \textit{(Unit Condition). } $\theta_0(y,1)=y.$
		\item \textit{(Cocycle Condition). } $\theta_0(\theta_0(y,\alpha),\alpha')=\theta_0(y,\alpha\cdot \alpha'),$ for any $\alpha,\alpha'\in (0,1]$.
	\end{enumerate}
\end{observation}
\begin{proof}\begin{enumerate}[label=(\roman*):] 
		\item This is immediate by construction of $\M^\ast(Y)$.
		\item The unit map is given by 
		$$s\colon  (0,1]\longrightarrow (0,1]\times (0,1]\qquad \beta\longmapsto (1,\beta)\,.$$
		Taking pullbacks once more yields the bundle map
		\begin{equation*} 
		\begin{tikzcd} s^\ast \pi^*(Y) \ar[rr,"s^\ast\theta"] \ar[dr,swap, "s^\ast \delta"] && s^\ast \M^*(Y) \ar[dl,"s^\ast\delta'"]\\
		& (0,1]
		\end{tikzcd}
		\end{equation*}
		with the bundle subspaces
		$$s^\ast \pi^\ast (Y)=Y\times \{1\}\qquad \qquad s^\ast \M^\ast(Y)=\{(y,1,f(y))\in \M^\ast(Y) \}\,.$$
		Since the unit condition from the descent data requires $s^\ast\theta=\id$, this yields $\theta_0(y,1)=y$.
		
		\item Examining Construction~\ref{cons:ArchLAXtopos}, the coycle condition requires $$(\m\times \id)^*(\theta)=\pi^*_{02}(\theta)\circ \pi^*_{12}(\theta)\,, $$
		where $\pi_{02},\pi_{12}$ are the obvious projection maps, and the multiplication map is given by
		\begin{align*}
		\m\times \id\colon (0,1]\times (0,1] \times (0,1] &\longrightarrow  (0,1]\times (0,1]\qquad (\alpha,\alpha',\beta)\longmapsto  (\alpha\cdot\alpha',\beta)\,.
		\end{align*}
In English: the RHS of the cocycle condition applies $\theta$ successively with parameters $\alpha$ and $\alpha'$, whereas the LHS applies $\theta$ once with parameter $\alpha\alpha'$. Comparing their first coordinates gives $\theta_0(\theta_0(y,\alpha),\alpha')=\theta_0(y,\alpha\cdot \alpha'),$ as required.
	\end{enumerate}
\end{proof}


\subsubsection*{Constructing $\J(F,\theta)$} We now show how to construct a new sheaf $\overline{F}\in\baseS\overleftarrow{[0,1]}$ from the original $(F,\theta)\in\D'$. First, consider the canonical map 
\begin{align}\label{eq:canonical(0,1]}
\Psi: \qint &\longrightarrow (0,1] 
\end{align} 
which sends a (discrete) point of $\qint:=\{q\in\Q  | 0<q\leq 1\}$ to its canonical representative in $(0,1]$. Regarding $F$ as an \'{e}tale bundle $f\colon Y\to(0,1]$, we can pullback $f$ along $\Psi$ to obtain:
\[\begin{tikzcd}
Z \ar[d,swap,"f_{\text{res}}"] \ar[r] & Y \ar[d,"f"]\\
\protect\mathbb{Q}_{(0,1]} 
\ar[r,"\Psi"] 
& \protect(0,1]
\end{tikzcd}\]
Alternatively, the bundle $f_{\text{res}}\colon Z\rightarrow \mathbb{Q}_{(0,1]}$ can be viewed as restricting the domain of $F\colon (0,1]\rightarrow [\mathbb{O}]$ to the map $F_{\text{res}}\colon \mathbb{Q}_{(0,1]}\rightarrow [\mathbb{O}]$.\footnote{Warning: $\mathbb{Q}_{(0,1]}$ cannot be thought of as a naive subspace of $(0,1]$ since $\mathbb{Q}_{(0,1]}$ is discrete space, unlike $(0,1]$.} In particular, if $F_{\text{res}}$ satisfies the continuity conditions of the Lifting Lemma~\ref{lem:liftingsheaves}, then Observation~\ref{obs:CHARof[0,1]SHEAVES} tells us that
it canonically defines a sheaf on $\overleftarrow{[0,1]}$. We verify this is indeed the case in the following claim.

\begin{claim}[Key Claim]\label{claim:descent2upperreals} Given $(F,\theta)\in\mathcal{D}'$, the descent data $\theta$ induces a map on fibres 
	\[\theta_{\gamma'\gamma}\colon F(\gamma')\rightarrow F(\gamma),\]
	for any $\gamma,\gamma'\in (0,1]$ such that $\gamma'\geq \gamma$, satisfying the following conditions:
	\begin{enumerate}[label=(\roman*)] 	
		\item $\theta_{\gamma\gamma}=\id$ for all $\gamma\in(0,1]$;
		\item If $\gamma,\gamma',\gamma''\in (0,1]$ such that $\gamma''\geq \gamma'\geq \gamma$, then $\theta_{\gamma''\gamma}=\theta_{\gamma'\gamma}\circ\theta_{\gamma''\gamma'}$;
		\item For any $\gamma\in(0,1]$, denote 
		\[I_\gamma:=\{q\mid \gamma<q<1\}\cup\{1\}\]
		to be its associated rounded ideal in $\RIdl(\qint,\prec)$, as defined in Example~\ref{ex:ridlreals}. 
		
		\noindent \underline{Then}, the induced map \[\theta_{\gamma}\colon\displaystyle\colim_{q\in I_\gamma}F(q)\rightarrow F(\gamma)\] 
		is an isomorphism.
	\end{enumerate}
\end{claim}

\begin{proof} Given $\gamma,\gamma'\in (0,1]$ such that $\gamma'\geq \gamma$ and $z\in F(\gamma')$, define the following function of sets
	\begin{align}\label{eq:THETAspecMORPHISM}
	\theta_{\gamma'\gamma}\colon  F(\gamma')&\longrightarrow F(\gamma)\\
	y&\longmapsto \theta_0\bigg(y,\frac{\gamma}{\gamma'}\bigg)\nonumber 
	\end{align}
	where $\theta_0\colon Y\times(0,1]\rightarrow Y$ is the first coordinate map of $\theta$ as in Observation~\ref{obs:theta-descent-prop}. Item (i) of the Observation gives the identity
		\begin{equation}
	f(\theta_0(y,\frac{\gamma}{\gamma'}))=f(y)\cdot\frac{\gamma}{\gamma'}\,\,.
	\end{equation}
In English: the multiplicative action on the base space $(0,1]$ (i.e. mapping $\gamma'\mapsto \gamma$) lifts to an action on the bundle space $Y$ (i.e. mapping $F(\gamma')\rightarrow F(\gamma)$). Notice this also ensures that $\theta_{\gamma'\gamma}$  is well-defined. 
We now check the required conditions of the Claim.	
	
	\subsubsection*{Step 1: Verifying Conditions (i) and (ii)} For Condition (i), the unit condition from Observation~\ref{obs:theta-descent-prop} yields
	$$\theta_0(y,\frac{\gamma}{\gamma})=\theta_0(y,1)=y\,,$$
	which by Equation~\eqref{eq:THETAspecMORPHISM} translates to $\theta_{\gamma\gamma}=\id\,.$
	For Condition (ii), suppose  $\gamma''\geq \gamma'\geq \gamma $ in $(0,1]$, and denote $y''\in Y$ such that $f(y'')=\gamma''$. The cocycle condition from Observation~\ref{obs:theta-descent-prop} yields $$\theta_0(y'',\frac{\gamma}{\gamma''})=\theta_0(\theta_0(y'',\frac{\gamma'}{\gamma''}),\frac{\gamma}{\gamma'}),$$ which translates to $\theta_{\gamma''\gamma}=\theta_{\gamma'\gamma}\circ\theta_{\gamma''\gamma'}$.

	\subsubsection*{Step 2: Reformulating Condition (iii)} By Discussion~\ref{dis:fibrewiseDISCRETE}, \'{e}tale bundles are fibrewise discrete, and so the fibre map $\theta_{\gamma\gamma'}$ corresponds to a map of sets. In particular. $\displaystyle\colim_{q\in I_\gamma} F(q)$ is a filtered colimit in $\Set$, and so admits the canonical description  
	\begin{equation}
	\colim_{q\in I_\gamma}F(q) = \coprod_{q\in I_{\gamma}} F(q)/\sim
	\end{equation}
	as a coproduct quotiented by the equivalence relation 
	\[(x,F(q))\sim (y,F(q')) \leftrightarrow \exists r\in I_{\gamma}. \left(q\prec r\land q'\prec r\land  \theta_{qr}(x)=\theta_{q'r}(y)\right), \]
	where ``$(x,F(q))$'' denotes $x\in F(q)$ and ``$(y,F(q'))$'' denotes $y\in F(q')$. Notice the specialisation order reverses the numerical order since $q\prec r$ iff $q>r$ or $q=r=1$ (Example~\ref{ex:ridlreals}).
	
	\smallskip
The canonical map
$$\theta_\gamma\colon
\colim_{q\in I_\gamma}F(q)\longrightarrow F(\gamma)$$
sends the class represented by $y\in F(q)$ to $\theta_{q\gamma}(y)$. Thus to verify $\theta_\gamma$ is an isomorphism of sets, it suffices to establish the implications 
	\begin{enumerate}[label=(\alph*)]
		\item $x\in F(\gamma ) \implies \exists q\in I_\gamma.\left(\exists y\in F(q)\,\land \,x=\theta_{q\gamma }(y)\right)$
		\item $y,z\in F(q),\theta_{q\gamma}(y)=\theta_{q\gamma}(z) \implies \exists r\in I_\gamma. (q\prec r \land \theta_{qr}(y)=\theta_{qr}(z))$\,.
	\end{enumerate}
For clarity: Implications (a) and (b) correspond to verifying surjectivity and injectivity of $\theta_\gamma$ respectively.
	
	
	\subsubsection*{Step 3: Surjectivity} We prove Implication (a) by exploiting the topology of the \'{e}tale bundle.
	
	\subsubsection*{Step 3a: Setup} Fix $x\in F(\gamma)$. Since $f\colon Y\rightarrow (0,1]$ is a local homeomorphism, there exists an open neighbourhood $U\subset Y$ of $x$ such that $f|_U\colon U\xrightarrow{\cong}f(U)$.  
	Let
$$	u\colon f(U)\longrightarrow U$$
	denote its inverse, so that $u(\gamma)=x$.\footnote{For details regarding local homeomorphisms in the point-free setting, see \cite[\S 5.1]{Vi07}.} Since $(0,1]$ has a basis consisting of rational-ended open intervals, we may shrink $U$ so that
$$	f(U)=(\alpha,\beta)
\qquad\text{or}\qquad
f(U)=(\alpha,1],$$
	for appropriate rationals $\alpha,\beta$ with $0\leq\alpha<\beta\leq1$.

	\subsubsection*{Step 3b: Exploiting topology and the unit condition} Denote
	\[X:=f(U)\cap [\gamma,1]\]
	with the closed interval $[\gamma,1]:=\{\gamma'\in(0,1] \,|\, \gamma\leq\gamma' \}$.\footnote{Notice this definition allows for the degenerate case $\gamma=1$, in which case $[\gamma,1]=\{1\}$.} By Step 3a, this means either
	$$	X=[\gamma,1]
	\qquad\text{or}\qquad
	X=[\gamma,\beta)\,\,\text{for some}\, \beta\in (0,1].$$
	In particular, $X$ is connected and contains $\gamma$.

Next, use the inverse $u\colon f(U)\rightarrow U$ to define the following map:
	\begin{align*}
	\Theta \colon  X& \longrightarrow F(\gamma)\\
	a&\longmapsto \theta_{a\gamma}(u(a))
	\end{align*}
Since $X$ is a connected space and $F(\gamma)$ is a discrete set, the image of $\Theta (X)$ is constant. In particular, since
	\[\Theta(\gamma)=\theta_{\gamma\gamma}(u(\gamma))=\theta_0(x,1)=x,\]
	where the final equality is by the unit condition, this implies $\Theta(a)=x$ for all $a\in X$. Finally,  notice there always exists some rational $q\in X$ such that $q\in I_\gamma$.\footnote{Why? If $X=[\gamma,\beta)$, this is obvious. If $X=[\gamma,1]$, let $q=1$. Notice this works even when $X=\{1\}$ since we allow $1\prec 1$.} Thus for $u(q)\in F(q)$, we obtain the identity $\Theta(q)=\theta_{q\gamma}(u(q))=x$, proving Implication (a).
	
	\subsubsection*{Step 4: Verifying injectivity} Fixing $q\in I_{\gamma}$, suppose $y,z\in F(q)$ such that $\theta_{q\gamma}(y)=\theta_{q\gamma}(z)$. For explicitness, denote $$x:=\theta_{q\gamma}(y)=\theta_{q\gamma}(z) \in F(\gamma).$$
Next, define two maps $$v,v'\colon (0,q]\rightarrow Y$$ whereby $v(a):=\theta_{qa}(y)$ and $v'(a):=\theta_{qa}(z)$ respectively. Notice:
	\begin{itemize}
			\item $v$ and $v'$ are partial sections of $f$ since
		\[f\circ v(a) = f\circ(\theta_{qa}(y))=a = f\circ v'(a),\qquad \text{	for any $a\in(0,q]$.}\]
		
		\item The images of $v$ and $v'$ coincide on $\gamma$, since
		\[v(\gamma)=\theta_{q\gamma}(y)=x=\theta_{q\gamma}(z)=v'(\gamma).\]
	
	\end{itemize}
	Finally, just as in Step 3, let $U\subset Y$ be an open neighbourhood of $x$ such that $f|_U\colon U\xrightarrow{\cong}f(U)$; let $u\colon f(U)\xrightarrow{\cong}U$ denote its inverse, so that $u(\gamma)=x$. 
	
	
	\subsubsection*{Step 4a: Refining the sections} Our objective here is to identify an open subspace of $(0,1]$ on which all the section maps $u,v,v'$ all agree. Since $q\in I_{\gamma}$ by assumption, we have $\gamma<q$ or $q=1$ (or both). Hence, define the subspace
	\begin{equation}
	V:=\begin{cases} f(U) \cap (0,q), \qquad \text{if $\gamma<q$}\\
	f(U), \qquad\quad\,\,\qquad\text{if $q=1$}
	\end{cases}
	\end{equation}
 In either case, $V$ is an open neighbourhood of $\gamma$ in $(0,1]$. Moreover, by design, both $u$ and $v$ are well-defined on the whole of $V$. Hence, their restrictions
$$ u,v\colon V\longrightarrow Y$$
yield sections of $f$ satisfying the identity
\begin{equation}\label{eq:u=v}
 u(\gamma)=v(\gamma)=x.
\end{equation}

Next, construct the obvious pullback:
	\begin{equation}
	\begin{tikzcd}
	V_u \ar[r] \ar[d,swap,"v_u"] \ar[dr, phantom, "\usebox\pullback", very near start, color=black] & V \ar[d,hook,"u"]\\
	V \ar[r,"v"] & Y
	\end{tikzcd}
	\end{equation}
Since $u\colon V\to Y$ is a homeomorphism onto the open subspace $u(V)\subseteq Y$, and open inclusions are preserved by pullback, the map $V_u\hookrightarrow V$ is an open inclusion. In other words, $V_u$ is the open subspace of $V$ on which $v$ and $u$ agree. In particular, this means that $\gamma\in V_u$ by Equation~\eqref{eq:u=v}.	

	Apply the same construction to obtain an open subspace $V'_{u}$ of $V$ on which $u$ and $v'$ agree (and thus, also $\gamma\in V'_{u}$). Repeat one last time to obtain an open subspace
	$$P:=V_u\cap V'_{u}$$
on which $u,v,v'$ all agree, and also $\gamma\in P$. Shrinking $P$ if necessary, we may assume that $P$ is an open interval of the form $(\alpha,\beta)$ or $(\alpha,1]$ with rationals $\alpha,\beta\in\Q$ such that $0\leq \alpha<\beta\leq q$.
	
	\subsubsection*{Step 4b: Finish} It remains to show that there exists $r\in P\cap I_{\gamma}$ such that $q\prec r$. Since $q\in I_\gamma$, either $\gamma<q$ or $q=1$. We therefore check the following cases.
	\begin{itemize}
		\item \underline{\textbf{Case 1:} $\gamma<q$}. Since $P$ is an open neighbourhood of $\gamma$, pick a rational	$r\in P$ such that $\gamma<r<q.$ 
		
		\smallskip 
		\item \underline{\textbf{Case 2a:} $q=1$ and $P=(\alpha,\beta)$}. Pick a rational $r$ where $\alpha<\gamma<r<\beta\leq q$. 
		
		\smallskip
		\item \underline{\textbf{Case 2b:} $q=1$ and $P=(\alpha,1]$}. Pick $r=1=q$. Here, $r\in I_\gamma$ and $q\prec r$ since the specialisation order allows for $1\prec1$ (Example~\ref{ex:ridlreals}).\footnote{Notice this covers the edge case $\gamma=1$.}
	\end{itemize}
In every case, we found a rational $r\in P\cap I_{\gamma}$ such that $q\prec r$. Since $q\prec r$, the map $\theta_{qr}$ is well-defined.  Moreover, since $u=v=v'$ on $P$ by Step 4a's construction, we obtain the desired identity $$\theta_{qr}(y)=v(r)=u(r)=v'(r)=\theta_{qr}(z),$$
thus proving Implication (b).	
\end{proof}

As an immediate corollary of  Claim~\ref{claim:descent2upperreals} and the Lifting Lemma~\ref{lem:liftingsheaves}, we get:

\begin{corollary}\label{cor:GETSHEAFoverUPPER} Any $(F,\theta)\in\D'$ determines a sheaf $\overline{F}\in\baseS\overleftarrow{[0,1]}$ by its restriction to the rationals.
\end{corollary}
\begin{proof} Represent $F\in\baseS(0,1]$ as a map $F\colon (0,1]\rightarrow [\mathbb{O}]$ to the object classifier, and consider its restriction $F_{\text{res}}\colon\mathbb{Q}_{(0,1]}\rightarrow [\mathbb{O}]$. By Claim~\ref{claim:descent2upperreals}, $F_{\text{res}}$ satisfies the continuity conditions required by item (iii) of Observation~\ref{obs:CHARof[0,1]SHEAVES}. The Lifting Lemma~\ref{lem:liftingsheaves} therefore extends $F_{\mathrm{res}}$ to a sheaf $\overline{F}$ on $\overleftarrow{[0,1]}$.
\end{proof}

\begin{discussion}[Choice vs. Existential Quantifiers] We briefly clarify the use of ``choice'' and ``pick'' in the preceding arguments. In Steps 3 and 4, we argued: given a point $x\in F(\gamma)$, since $f\colon Y\to (0,1]$ is a local homeomorphism, there exists an open $U\subset Y$ such that $x\in U\cong f(U)$. Choosing such a $U$ does not invoke the Axiom of Choice. Constructively, once the existence of an object satisfying a given property has been established, then invoking this object is simply the elimination of an existential quantifier; for more context, see \cite[\S 1.3]{BauerConstructive}. In our setting, the required existence follows constructively from the fact that all \'{e}tale maps $f\colon Y\to X$ yield a basis of opens admitting unique sections; see \cite[Proposition 1.49]{Vi07}. 
	
	
\end{discussion}

\subsubsection*{The Main Construction} We can now define our functor $\mathfrak{J}$. On the level of objects, we map:
\begin{align*}
\mathfrak{J}\colon \mathcal{D}' \longrightarrow & \,\mathcal{S}\overleftarrow{[0,1]}\\
\quad (F,\theta)  \mapsto & \quad  \overline{F}
\end{align*}
where $\overline{F}$ is the sheaf over $\overleftarrow{[0,1]}$ associated to $(F,\theta)$ by Corollary~\ref{cor:GETSHEAFoverUPPER}, via the Lifting Lemma~\ref{lem:liftingsheaves}. Explicitly, for an upper real $I_{\gamma}\in \RIdl (\qint,\prec)$, we define 
$$\overline{F}(I_{\gamma}):=\colim_{q\in I_{\gamma}}F(q)\,.$$

The same procedure can be carried out on the level of morphisms. Consider a $\calD'$-morphism
$$
u\colon(F,\theta)\longrightarrow(G,\xi)\,\,.
$$
This, of course, defines a natural transformation of sheaves, but let us reformulate the details within point-free topology. Representing $F,G\colon (0,1]\rightarrow[\mathbb{O}]$ as maps to the object classifier, $u$ gives pointwise, specialisation morphisms
$$
u_q\colon F(q)\longrightarrow G(q)
$$
compatible with the transition maps
$$
\begin{tikzcd}
F(q') \ar[r,"u_{q'}"] \ar[d,swap,"\theta_{q'q}"]&
G(q') \ar[d,"\xi_{q'q}"]\\
F(q) \ar[r,swap,"u_q"]&
G(q).
\end{tikzcd}\,.
$$
Thus, for each $I_\gamma\in\RIdl(\qint,\prec)$, the maps $u_q$ form a cocone from the filtered diagram
$$
\{F(q')\xrightarrow{\theta_{q'q}}F(q)\}_{q,q'\in I_\gamma}
$$
to $\overline G(I_\gamma)=\colim_{q\in I_\gamma}G(q)$. By the universal property of colimits, they induce a morphism
$$
\mathfrak J(u)(I_\gamma)\colon
\overline F(I_\gamma)
\longrightarrow
\overline G(I_\gamma).
$$
These morphisms are natural in $I_\gamma$, so define
$$
\mathfrak J(u)\colon\overline F\longrightarrow\overline G.
$$
Identity and composition are preserved by the universal property of the colimits, so $\mathfrak J$ is a functor.

\subsection{Second Direction}\label{subsec:K2ndDIRECTION} We now construct a functor inverse to $\J$, i.e. 
\begin{align}
\mathfrak{K}\colon \mathcal{S}\overleftarrow{[0,1]} \longrightarrow & \,\mathcal{D}'\\
F \mapsto & \,\,  ? \,\,\,\,\,\,.\nonumber 
\end{align}

Start with the diagram of spaces:
\[ \begin{tikzcd}
\protect(0,1]\times \protect(0,1] \ar[r,shift left=.75ex,"\pi"]
\ar[r,shift right=.75ex,swap,"\M"] & \protect(0,1]\ar[r,"r"] & \overleftarrow{\protect[0,1]}\end{tikzcd}\]
where the map
\begin{align}\label{eq:r-map}
r\colon (0,1]&\longrightarrow\overleftarrow{[0,1]}\\
\gamma&\longmapsto I_{\gamma}:=\{q\mid \gamma<q\} \nonumber
\end{align}
 sends a Dedekind $\gamma$ to its right Dedekind section. On the level of sheaves, the arrows reverse and the associated inverse image functors yield the following diagram
\begin{equation}\label{eq:KsetupDIAGRAM}
\begin{tikzcd}
\mathcal{S}(0,1]\times\baseS(0,1] & \mathcal{S}\protect(0,1] \ar[l,shift right=.75ex,swap,"\pi^*"]
\ar[l,shift left =.75ex,"\M^*"] & \mathcal{S}\protect\overleftarrow{[0,1]} \ar[l,swap, "r^*"].\end{tikzcd}
\end{equation}
That is, given a sheaf $F$ on $\overleftarrow{[0,1]}$, the above functor $r^*$ sends it to a sheaf 
\begin{equation}\label{eq:rF}
\widehat{F}:=r^*(F)
\end{equation}
 over $(0,1]$,  essentially for free. The non-trivial task then is to equip  $\widehat{F}$ with descent data
$$\hat{\theta}\colon \pi^*(\widehat{F})\to \M^*(\widehat{F})\,.$$
Our guiding observation is that the descent data is naturally induced by the specialisation order between points of the upper reals.\footnote{The attentive reader may notice that this is broadly analogous to the proof of the Lifting Lemma~\ref{lem:liftingsheaves}.} We now develop this remark. 

\begin{construction}[Descent Data]\label{cons:inverse-setup} Let $F\in\baseS\overleftarrow{[0,1]}$ be a sheaf, and write $\widehat{F}:=r^*(F)$ as above. Invoking Fact~\ref{fact:sheaves} once more, represent both sheaves as the below maps to the object classifier
	$$F\colon \overleftarrow{[0,1]}\to \OC \qquad\text{and}\qquad  \widehat{F}\colon (0,1]\to \OC\,.$$
We construct the descent data in stages.
\begin{enumerate}
	\item Fix $\alpha,\beta\in (0,1].$ Since $\pi(\alpha,\beta)=\beta$ and $\M(\alpha,\beta)=\alpha\cdot \beta$, we obtain the expected characterisation of the respective pullbacks:
	\begin{itemize}
		\item $\pi^*(\widehat{F})\colon  (0,1]\times(0,1]\rightarrow [\mathbb{O}]$ is a map that sends $(\alpha,\beta)\mapsto \widehat{F}(\beta)$;
		\item $\M^*(\widehat{F})\colon  (0,1]\times(0,1]\rightarrow [\mathbb{O}]$ is a map that sends $(\alpha,\beta)\mapsto \widehat{F}(\alpha\cdot\beta)$.
	\end{itemize}		
\item Recall from Definition~\ref{def:ptFREEspace} that any map between point-free spaces is functorial with respect to its points. Thus, for any $\gamma,\gamma',\in(0,1]$ such that 
$$I_{\gamma'}\sqsubseteq I_{\gamma} \qquad\text{in $\overleftarrow{[0,1]}$}\,$$ 
the map $F$ sends this specialisation order between upper reals to a morphism
\[F(I_{\gamma'})\xrightarrow{s_{\gamma' \gamma}} F(I_{\gamma}).\]
\item Fix $\alpha,\beta\in(0,1]$. Clearly, we have
$$\alpha\cdot\beta\leq\beta,
\qquad\text{and hence}\qquad
I_\beta\sqsubseteq I_{\alpha\cdot\beta}.$$
By Step (2), this gives the morphism  
\begin{equation}\label{eq:specMORPH1} s_{\beta(\alpha\cdot\beta)}\colon F(I_\beta)\longrightarrow F(I_{\alpha\cdot\beta}).
\end{equation}
Under the identification $\widehat{F}=r^* F$, this yields the morphism
\[s_{\beta(\alpha\cdot\beta)}\colon \widehat{F}(\beta)\longrightarrow \widehat{F}(\alpha\cdot\beta)\]
By Step~(1), these components determine a
$\baseS(0,1]\times\baseS(0,1]$-morphism
\[\widehat{\theta}\colon \pi^*(\widehat{F})\longrightarrow \M^*(\widehat{F}).\]
The unit and cocycle conditions follow formally from the functoriality of $F$, since the identity and composition laws for the induced morphisms correspond respectively to the unital and associativity laws of multiplication
$$1\cdot\beta=\beta
\qquad\text{and}\qquad
\alpha\cdot (\alpha'\beta)=(\alpha\alpha')\cdot \beta.$$
\end{enumerate}

\end{construction}

We now define our functor $\K$. On the level of objects, we map:
\begin{align*}
\mathfrak{K}\colon \mathcal{S}\overleftarrow{[0,1]} \longrightarrow & \,\mathcal{D}'\\
F \longmapsto & \,(\widehat{F},\widehat{\theta})
\end{align*}
where $\widehat{F}:=r^\ast(F)$ as in Equation~\eqref{eq:rF}, and $\widehat{\theta}$ as in Construction~\ref{cons:inverse-setup}. This is well-defined, since $\widehat{F}$ is a sheaf over $(0,1]$ and $\widehat{\theta}$ satisfies the required descent conditions.

On the level of morphsims: for any $\baseS\overleftarrow{[0,1]}$-morphism $u\colon F\rightarrow G$ where $F,G\colon\overleftarrow{[0,1]}\rightarrow [\mathbb{O}]$, define 
\[\K(u):= r^*(u)\colon \widehat{F}\rightarrow\widehat{G}.\]
with $r^\ast$ as in Diagram~\eqref{eq:KsetupDIAGRAM}. This defines an $\baseS\overleftarrow{[0,1]}$-morphism functorially, so it remains to check that $r^\ast (u)$ is also a $\D'$-morphism. Explicitly, this means: given $(\widehat{F},\widehat{\theta}_F), (\widehat{G},\widehat{\theta}_G)\in \mathcal{D}'$ with the descent data $\widehat{\theta}_F,\widehat{\theta}_G$ from Construction~\ref{cons:inverse-setup}, the following identity holds:
\begin{equation}\label{eq:D'morphism}
\M^*(\K(u))\circ \widehat{\theta}_F =\widehat{\theta}_{G}\circ \pi^*(\mathfrak{K}(u)).
\end{equation}
To see this, it is helpful to work pointwise. Equation~\eqref{eq:D'morphism} says that the following square commutes
\begin{equation}\label{eq:D'-2}
\begin{tikzcd}
\widehat{F}(\beta) \ar[r,"\widehat{\theta}_F"] \ar[d,swap, "\mathfrak{K}(u)"] & \widehat{F}(\alpha\cdot \beta) \ar[d, "\mathfrak{K}(u)"] \\
\widehat{G}(\beta) \ar[r,"\widehat{\theta}_G"] & \widehat{G}(\alpha\cdot \beta)
\end{tikzcd}
\end{equation}
for a generic point $(\alpha,\beta)\in (0,1]\times (0,1]$. By Construction~\ref{cons:stdDESCENT}, the horizontal maps of Diagram~\eqref{eq:D'-2} are obtained by applying $F$ and $G$, respectively, to the specialisation order
$$I_\beta\sqsubseteq I_{\alpha\beta}\,,$$
before restricting along $r$. As such, commutativity of Diagram~\eqref{eq:D'-2} follows from the naturality square \[\begin{tikzcd} F(I_{\gamma'}) \ar[r,"s_{\gamma'\gamma}"] \ar[d,swap, "u"] & F(I_\gamma) \ar[d, "u"] \\ G(I_{\gamma'}) \ar[r,"t_{\gamma'\gamma}"] & G(I_\gamma) \end{tikzcd}\,\,.\]
for $I_{\gamma'}\sqsubseteq I_\gamma$. In particular, taking $\gamma'=\beta$ and $\gamma=\alpha\cdot \beta$ gives the commutativity of Diagram~\eqref{eq:D'-2}.

%



\subsection{Assemble and Finish} With the key constructions completed, all the gears line up and we now prove the main result of the section. 

\begin{theorem}\label{thm:ARCHIMEDEANPLACE} $\D'\simeq\baseS\overleftarrow{[0,1]}$. 	
\end{theorem}

\begin{proof} Let $\J\colon\mathcal{D}'\rightarrow \overleftarrow{[0,1]}$ and $\K\colon\mathcal{S}\overleftarrow{[0,1]}\rightarrow\mathcal{D}'$ be the two functors as defined in Sections~\ref{subsec:ElephantstDIRECTION} and \ref{subsec:K2ndDIRECTION}. We now check that $\J,\K$ are mutually inverse.
	
	\subsubsection*{Step 1: Verifying $\K\circ\J\cong\id_{\D'}$} Let $(F,\theta)\in\mathcal{D}'$, and write
$$(\widehat{\overline{F}},\widehat{\theta}):=\K\circ\J(F,\theta).$$
We shall regard the sheaves as maps to the object classifier, i.e. as maps $F,\widehat{\overline{F}}\colon (0,1]\to\OC\,.$

The core mechanism of our argument is  Key Claim~\ref{claim:descent2upperreals} (iii), which gives an isomorphism
	\[\theta_\gamma \colon \colim_{q\in I_\gamma}F(q)\xrightarrow{\sim} F(\gamma)\,,\qquad\text{for}\, \gamma\in (0,1]\,.\] 
In what follows, we show that $\theta_{\gamma}$ extends to a natural isomorphism of objects $(F,\theta)\cong(\widehat{\overline{F}},\widehat{\theta})$: we proceed by first identifying the underlying sheaves, before verifying compatibility with the descent data. Finally, we observe that the cocone-induced construction is functorial in $F$, giving the required naturality.

\subsubsection*{Step 1a: $\widehat{\overline{F}}\cong F$ as sheaves.}

Reviewing Equation~\eqref{eq:r-map}, the map $r\colon (0,1]\to \overleftarrow{[0,1]}$ sends a Dedekind $\gamma\in (0,1]$ to its right Dedekind section, represented as a rounded ideal $I_{\gamma}$. Thus, unwinding the definition of $\mathfrak K\circ\mathfrak J$, we get	
\[\widehat{\overline{F}}(\gamma)=\overline{F}\big(r(\gamma)\big) = \overline{F}(I_\gamma)= \displaystyle\colim_{q\in I_\gamma}F(q)\cong F(\gamma)\,,\]
 where the final isomorphism is by Key Claim~\ref{claim:descent2upperreals}. Since the isomorphism was constructed geometrically  (Convention~\ref{conv:geometricity}) and done for generic $\gamma\in (0,1]\,$, this extends to an isomorphism of point-free maps $\widehat{\overline F}\cong F,$ and thus an isomorphism of sheaves.

\subsubsection*{Step 1b: Compatibiity with descent data} To show that the isomorphism of sheaves in Step 1a extends to an isomorphism in the descent category $\D'$, we need to verify its compatibility with the descent data. In the language of Construction~\ref{cons:inverse-setup}, this amounts to showing that the diagram \footnote{A remark on notation: We have been careful to denote the multiplication $\M(\alpha,\beta)=\alpha\cdot\beta$ as opposed to just $\M(\alpha,\beta)=\alpha\beta$. One reason for this is to emphasise that an algebraic action has taken place. Another reason is to reduce potential confusion between the morphism $\theta_{(\alpha\cdot\beta)}$ in Diagram~\eqref{eq:descentISO1} and the  morphism $\theta_{\alpha\beta}\colon F(\alpha)\rightarrow F(\beta)$ as defined in Equation~\eqref{eq:THETAspecMORPHISM}.}: 
\begin{equation}\label{eq:descentISO1}
\begin{tikzcd}
\widehat{\overline{F}}(\beta) \ar[r,"\theta_\beta"] \ar[d,swap,"\widehat{\theta}"] & F(\beta) \ar[d,"\theta"]\\
\widehat{\overline{F}}(\alpha\cdot \beta) \ar[r,"\theta_{\alpha\cdot\beta}"]& F(\alpha\cdot\beta)
\end{tikzcd}
\end{equation}
commutes for generic $(\alpha, \beta)\in(0,1]\times (0,1]$. Notice  the horizontal arrows in Diagram~\eqref{eq:descentISO1} are well-defined by Step 1a, which yields the identity  \[
\widehat{\overline F}(\beta)
=
\colim_{q\in I_\beta}F(q) \qquad\text{and}\qquad \widehat{\overline F}(\alpha\cdot\beta)
=
\colim_{q\in I_{\alpha\cdot\beta}}F(q) \,.
\]
In particular, by the universal property of colimits, it suffices to show that the two composites $\theta_{\alpha\cdot\beta}\circ\widehat\theta$ and $\theta\circ\theta_\alpha$ agree after precomposition with each component map for $\widehat{\overline F}(\beta)$, which we denote
\[
t_{q\beta}\colon F(q)\longrightarrow
\colim_{q\in I_\beta}F(q),
\qquad q\in I_\beta.
\]

These component maps were left implicit in our proof of Key Claim~\ref{claim:descent2upperreals}, so it is worth making explicit how they determine the horizontal isomorphisms in Diagram~\eqref{eq:descentISO1}. Recall: given $(F,\theta)\in\D'$, we defined a system of transition maps on the fibres 
$$\theta_{\gamma'\gamma}\colon F(\gamma')\to F(\gamma).$$
Read in our context, this assembles into the cocone diagrams
\begin{equation}\label{eq:descentISO2}
\begin{tikzcd}
F(q) \ar[rr,"t_{q\beta}"] \ar[drr,"\theta_{q\beta}",swap]&& \displaystyle\colim_{q\in I_{\beta}}F(q)\ar[d,"\theta_{\beta}"]\\
&& F(\beta)
\end{tikzcd}\qquad \qquad \begin{tikzcd}
F(q) \ar[rr,"t_{q(\alpha\cdot\beta)}"] \ar[drr,"\theta_{q(\alpha\cdot\beta)}",swap]&& \displaystyle\colim_{q\in I_{\alpha\cdot\beta }}F(q) \ar[d,"\theta_{\alpha\cdot\beta}"]\\
&& F(\alpha\cdot\beta)
\end{tikzcd}\,\,\,,
\end{equation}
whose cocone maps $\theta_{\alpha},\theta_{\alpha\cdot\beta}$ are now known to be isomorphisms. 
Moreover, for any $q\in I_{\alpha}$, we also get the cocone diagram
\begin{equation}\label{eq:descentISO3}
\begin{tikzcd}
F(q) \ar[rr,"t_{q\beta}"] \ar[drr,"t_{q(\alpha\cdot\beta)}",swap]&& \displaystyle\colim_{q\in I_{\beta}}F(q)\ar[d,dashed,"\widehat{\theta}" yshift=0.4em]  = \overline{F}(I_{\beta})\\
&& \displaystyle\colim_{q\in I_{\alpha\cdot\beta}}F(q) = \overline{F}(I_{\alpha\cdot\beta})
\end{tikzcd}\,.
\end{equation}
To identify the induced cocone map with the descent data $\widehat{\theta}$, recall from Construction~\ref{cons:inverse-setup} that
$$\widehat{\theta}\colon \overline{F}(I_{\beta})\to \overline{F}(I_{\alpha\cdot\beta})$$
is obtained by applying the map $\overline{F}$ to the specialisation order
\begin{equation}\label{eq:Ibab}
I_\beta\sqsubseteq I_{\alpha\cdot\beta}
\end{equation}
on the upper reals. However, the map $\overline{F}$ itself is
constructed from the system of fibres $F(q)$ via colimits,
\[
\overline F(I)=\colim_{q\in I}F(q),
\]
and so its action on the component maps $t_{q\beta}$ is given by the
corresponding component maps $t_{q(\alpha\cdot\beta)}$. In other words, 
$\widehat{\theta}$ is precisely the morphism between the corresponding
colimits induced by the inclusion~\eqref{eq:Ibab}. In particular, this yields the identity 
\begin{equation}\label{eq:descentISO4}
\widehat{\theta}\circ t_{q\beta}
=
t_{q(\alpha\cdot\beta)}
\qquad(q\in I_\beta).
\end{equation}
\smallskip

We now assemble the pieces to establish the desired identity 
\begin{equation}\label{eq:1b-GOAL}
\theta_{\alpha\cdot\beta}\circ\widehat\theta=\theta\circ\theta_\alpha\,.
\end{equation}
For the LHS, combine Diagram~\eqref{eq:descentISO2} and Equation~\eqref{eq:descentISO4} to obtain the identity 
\begin{equation}\label{eq:1b-one}
\theta_{q(\alpha\cdot\beta)}=\theta_{\alpha\cdot\beta} \circ t_{q(\alpha\cdot\beta)} =\theta_{\alpha\cdot\beta}\circ  \widehat{\theta}\circ t_{q\beta} \qquad \qquad(q\in I_\beta).
\end{equation}
For the RHS, the descent morphism $\theta$, by definition, induces the following map on fibres 
\[
\theta_{\beta(\alpha\cdot\beta)}\colon
F(\beta)\longrightarrow F(\alpha\cdot\beta).
\]
Therefore, for any $q\in I_\beta$, we compute
\begin{align}\label{eq:1b-two}
\theta\circ\theta_\beta\circ t_{q\beta} &= \theta_{\beta(\alpha\cdot\beta)}\circ\theta_\beta\circ t_{q\beta} & \text{[since $\theta=\theta_{\beta(\alpha\cdot\beta)}$]}\nonumber\\
&= \theta_{\beta(\alpha\cdot\beta)} \circ \theta_{q\beta} & \text{[by Diagram~\eqref{eq:descentISO2}]}\nonumber \\
&= \theta_{q(\alpha\cdot\beta)} & \text{[by Key Claim~\ref{claim:descent2upperreals} (ii)]}
\end{align}
Hence, by Equations~\eqref{eq:1b-one} and ~\eqref{eq:1b-two}, the composites $\theta_{\alpha\cdot\beta}\circ\widehat\theta$ and $\theta\circ\theta_\alpha$ agree after precomposition with $t_{q\beta}$ for every $q\in I_\beta$. By the universal property of the
colimit, this proves~\eqref{eq:1b-GOAL}. In English: this shows that the isomorphism of sheaves in Step 1a is in fact compatible with the descent data, and thus yields an isomorphism of objects in $\D'$
\[
(F,\theta)\cong
(\widehat{\overline F},\widehat\theta).
\]
\subsubsection*{Step 1c: Naturality.}  It remains to show that the isomorphisms constructive above are also {\em natural} in $(F,\theta)$. On this front, let $u\colon (F,\theta)\rightarrow (G,\xi)$ be a $\D'$-morphism. We must show that 
\begin{equation}\label{eq:KJnaturality}
\K\J(u)\circ\eta_{(F,\theta)}
=
\eta_{(G,\xi)}\circ u,
\end{equation}
where $\eta_{(F,\theta)}$ and $\eta_{(G,\xi)}$ are the isomorphisms
constructed in Steps~1a and~1b. In fact, it suffices to verify~\eqref{eq:KJnaturality} on the fibres. Fixing some $\gamma\in (0,1]$, this amounts to the commutativity of the diagram 
	\begin{equation}
\begin{tikzcd}\label{eq:KJnaturality-fibre}
F(\gamma) \ar[rr,"u_\gamma"] && G(\gamma)\\
\displaystyle\colim_{q\in I_\gamma} F(q) \ar[u,"\theta_{\gamma}"] \ar[rr,"\K\J(u)"]&& \displaystyle\colim_{q\in I_\gamma} G(q) \ar[u,"\xi_{\gamma}"]
\end{tikzcd}
\end{equation}

Recall that $\J(u)$ was constructed by applying the universal property
of the colimits to the fibre maps
\[
u_q\colon F(q)\longrightarrow G(q).
\]
Thus, writing
\[
t_{q\gamma}\colon F(q)\longrightarrow
\colim_{q\in I_\gamma}F(q),
\qquad
t'_{q\gamma}\colon G(q)\longrightarrow
\colim_{q\in I_\gamma}G(q)
\]
for the respective cocone maps, we have
\begin{equation}\label{eq:KJ-on-cocone}
\K\J(u)\circ t_{q\gamma}
=
t'_{q\gamma}\circ u_q.
\end{equation}
On the other hand, as seen in Diagram~\eqref{eq:descentISO2}, the defining property of the maps
$\theta_\gamma$ and $\xi_\gamma$ gives
\begin{equation}\label{eq:1c-one}
\theta_\gamma\circ t_{q\gamma}=\theta_{q\gamma},
\qquad
\xi_\gamma\circ t'_{q\gamma}=\xi_{q\gamma}.
\end{equation}
Since $u$ is a $\D'$-morphism, its fibre maps are compatible with the
transition maps, so
\begin{equation}\label{eq:1c-two}
u_\gamma\circ\theta_{q\gamma}
=
\xi_{q\gamma}\circ u_q.
\end{equation}
Consequently, plugging in the relevant identities, compute for every $q\in I_\gamma$:
\begin{align*}
u_{\gamma}\circ \theta_\gamma\circ t_{q\gamma}&= u_{\gamma}\circ \theta_{q\gamma} && \text{[by Equation~\eqref{eq:1c-one}]} \\
&= \xi_{q\gamma}\circ u_q && \text{[by Equation~\eqref{eq:1c-two}]} \\
&=\xi_{\gamma}\circ t'_{q\gamma}\circ u_q && \text{[by Equation~\eqref{eq:1c-one}]}\\
&= \xi_{\gamma}\circ \K\J(u)\circ t_{q\gamma} && \text{[by Equation~\eqref{eq:KJ-on-cocone}]}
\end{align*}
Hence, the two morphisms in Diagram~\eqref{eq:KJnaturality-fibre}
agree after precomposition with every cocone map $t_{q\gamma}$, where
$q\in I_\gamma$. By the universal property of the colimit, the
morphisms are therefore equal. Since Steps~1a--1c hold at the generic
point $\gamma\in(0,1]$, these pointwise isomorphisms yield a corresponding natural isomorphism of functors 
\[
\eta\colon\K\J\xrightarrow{\sim}\id_{\D'}.
\]
This completes Step 1.
	
	\subsubsection*{Step 2: Verifying $\J\circ\K\cong\id_{\mathcal{S}\protect\overleftarrow{[0,1]}}$}  Let $F$ be a sheaf over $\overleftarrow{[0,1]}$, and write 
	$$(\widehat{F},\widehat{\theta}):=\mathfrak{K}(F)\qquad\qquad \overline{\widehat{F}}:=\mathfrak{J}\circ\mathfrak{K}(F)\,.$$
	\smallskip 
Here we verify that $\overline{\widehat F}\cong F$ naturally in $F$.
	\begin{enumerate}[label=(\alph*)]
		\item \emph{$\overline{\widehat{F}}\cong F$ as sheaves.} For any upper real $I_\gamma\in\overleftarrow{[0,1]}$, one easily verifies that
		\[\overline{\widehat{F}}(I_\gamma)=\displaystyle\colim_{q\in I_\gamma}\widehat{F}(q)=\displaystyle\colim_{q\in I_{\gamma}}F(r(q))=\displaystyle\colim_{q\in I_\gamma}F(I_q)\cong F(I_\gamma),\]
		where the final isomorphism follows from the fact that maps preserve filtered colimits of points (Lemma~\ref{lem:filteredCOLIMITpts}).\footnote{Compare this, perhaps, with our use of Key Claim~\ref{claim:descent2upperreals} in Step 1a.} 
		\item \emph{Naturality.} Naturality essentially follows from the colimit structure built into the definition of $\J\K$. To elaborate, given a morphism $u\colon F\to G$, the morphism $\J\K(u)$ is, by construction,
		the map induced on the above colimits by the morphisms
		\[
		u_{I_q}\colon F(I_q)\longrightarrow G(I_q).
		\]
	Thus, writing
	\[
	t_{q\gamma}\colon F(I_q)\longrightarrow
	\colim_{q\in I_\gamma}F(I_q),
	\qquad
	t'_{q\gamma}\colon G(I_q)\longrightarrow
	\colim_{q\in I_\gamma}G(I_q)
	\]
	for the respective cocone maps, we have
	\[
	\J\K(u)\circ t_{q\gamma}
	=
	t'_{q\gamma}\circ u_{I_q}.
	\]
	On the other hand, the naturality of $u$ with respect to the morphism
	$I_q\sqsubseteq I_\gamma$ gives
	\[
	u_{I_\gamma}\circ t_{q\gamma}
	=
	t'_{q\gamma}\circ u_{I_q}.
	\]
	Hence $\J\K(u)$ and $u_{I_\gamma}$ agree after precomposition with every
	cocone map $t_{q\gamma}$. By the universal property of the colimit, they
	are therefore equal under the identifications
	\[
	\colim_{q\in I_\gamma}F(I_q)\cong F(I_\gamma),
	\qquad
	\colim_{q\in I_\gamma}G(I_q)\cong G(I_\gamma).
	\]
	Thus the isomorphisms
	$\overline{\widehat F}\cong F$ are natural in $F$.	

	\end{enumerate}
	This completes Step 2, and we are done.
\end{proof}

Let's step back for a moment. What would happen if we used standard
descent to define the Archimedean place instead of lax descent? Given
the association between group completion and standard descent
(cf.~Discussion~\ref{dis:GroupoidReflection}), one may expect
a loss of information. The following
observation confirms this.

\begin{observation}\label{obs:grpcompletion}
	Let $Z$ denote the space corresponding to the standard descent topos of
	Construction~\ref{cons:ArchLAXtopos}. Then
	\[
	Z\cong\{*\}.
	\]
\end{observation}
\begin{proof} The argument is short, but it will be helpful to organise it into stages.
	
\subsubsection*{Step 0: Setup} By construction, $Z$ is the coequaliser of $$\pi,\M\colon (0,1]\times (0,1] \to (0,1]\,$$
subject to descent conditions. Let
$$p\colon (0,1]\to Z$$
denote the universal quotient map, and let $Z'$ denote the image of $(0,1]\times(0,1]$ under the map
$p\circ\pi$ (equivalently, under $p\circ\M$). This comes equipped with an inclusion map
\[
i\colon Z'\hookrightarrow Z.
\]
\subsubsection*{Step 1: Constant image} We show that $Z'=\{\ast\}$. Since $p\circ\pi=p\circ\M$, we have
\[
p(\beta)
=
p(\beta\beta')
=
p(\beta') \qquad\text{for any}\, \beta,\beta'\in(0,1]\,.
\]
The first equality follows by evaluating
$p\circ\pi=p\circ\M$ at $(\beta',\beta)$, while the second follows by
evaluating it at $(\beta,\beta')$. Thus $p$ is constant. Since $Z'$ is the image of $p\circ\pi$, it follows that $Z'=\{*\}$, as claimed.

\subsubsection*{Step 2: Trivialisation}

Let $p'\colon(0,1]\to\{*\}$ be the unique map to the singleton space. Assemble our data into the following diagram
\begin{equation}
\begin{tikzcd}
(0,1] \times (0,1] \ar[r,shift left=.75ex,"\pi"]
\ar[r,shift right=.75ex,swap,"\M"]
&
(0,1] \ar[r,"p"] \ar[ddr,swap, shift right=1,"p'"] 
& Z \ar[dd,  bend right, swap,dashed,"j"]\\
\\
& & \{*\}= Z \ar[uu,hook, bend right, swap,"i"]
\end{tikzcd}
\end{equation}
where $j$ is obtained via the universal property of coequalisers. In particular, this gives $j\circ p = p'$.

We claim that $Z\cong\{*\}$. The fact that $j\circ i=\id_{\{*\}}$ is obvious. For the converse direction, compute for any $\beta\in(0,1]$:
\[i\circ j\circ p(\beta)=i\circ p'(\beta)= p(\beta),\]	
where $i\circ p'=p$ by definition of $Z'$. Since $p$ is an epimorphism, this implies $i\circ j =\id_{Z}$.
\end{proof}

\begin{remark} The proof strategy of Observation~\ref{obs:grpcompletion} can be adapted to give an alternative proof that a single non-Archimedean place corresponds to a singleton. Notice, however, we still need to verify that Diagram~\eqref{eq:groupoidFREETRANSITIVE} is an {\em open} groupoid in order to apply Moerdijk's Stability Theorem~\ref{thm:MoerdijkSTABILITY}, which was used in Theorem~\ref{thm:ALLNAplace}'s proof that $[\mathrm{places}_{NA}]\cong \ISpec(\Z)_{\neq (0)}$.  Perhaps the simplest way to show this is by comparing Diagram~\eqref{eq:groupoidFREETRANSITIVE} with the obvious \v{C}ech groupoid~\eqref{eq:KERNELpairGROUPOID}. In which case, it is more natural to use the \v{C}ech groupoid and its associated descent data directly to establish the trivialisation, as in Theorem~\ref{thm:DESCENTsinglePRIME}.
 \end{remark}

\section{Non-Trivial Forking of Sheaves}\label{sec:ARCHvsNONARCH} This section investigates the topos-theoretic differences between the Archimedean vs. non-Archimedean place. By Theorems~\ref{thm:DESCENTsinglePRIME} and \ref{thm:ARCHIMEDEANPLACE}, we already know that
$$\D\simeq\Set\qquad\D'\simeq \baseS\overleftarrow{[0,1]}\,.$$
Motivated by this, we ask: 
\begin{question}\label{qn:forkLAXdescent} What kinds of sheaves are eliminated by standard vs. lax descent? Alternatively, how wild or complicated are the sheaves of $\D'$ compared to those of $\D$? 
\end{question}

The following basic observation tells us where to start looking.

\begin{observation}\label{obs:YdisjointDECOMPOSITION} Let $X$ be a locally connected localic space, and let $f\colon Y\to X$ be an \'{e}tale bundle on $X$. \underline{Then}, $Y$ is locally connected. In particular, $Y$ admits a pairwise disjoint decomposition
		\begin{equation}\label{eq:Ydecomposition}
	Y=\displaystyle\coprod_{i\in I} Y_i.
	\end{equation}
into open connected subspaces. 
\end{observation}
\begin{proof} Following \cite[pp.~525]{Elephant}, recall that a localic space is
	\emph{locally connected} if every open subspace is expressible as a
	union of connected open subspaces. Since $f$ is a local homeomorphism, every point of $Y$ has an open neighbourhood $V\subseteq Y$ such that $f|_V\colon V\to f(V)$ is a homeomorphism onto an open subspace of $X$. As $X$ is locally connected, each $f(V)$ is a union of connected open subspaces, and hence so is $V$. Thus $Y$ is locally connected.  Applying \cite[Lemma~C.1.5.8]{Elephant}, we obtain a pairwise disjoint
	decomposition of $Y$ into open connected subspaces.
\end{proof}

In particular, recall that:
\begin{itemize}
	\item Any $(F,\theta)\in\D$  defines a sheaf $F$ on $(0,\infty)$;
	\item Any $(F',\theta')\in\D'$ defines a sheaf $F'$ on $(0,1]$. 
\end{itemize}
Since $(0,\infty)$ and $(0,1]$ are both locally connected localic spaces, Observation~\ref{obs:YdisjointDECOMPOSITION} suggests that analysis of sheaves in $\D$ or $\D'$ 
ought to be reducible to analysis of their sheaves' connected components. Leveraging this insight, we establish the next series of observations.

\begin{observation}\label{obs:connNAsheaf} Let $(F,\theta)\in\D$, with $f\colon Y\to (0,\infty)$ as the corresponding \'{e}tale bundle. \underline{Then}, 
		\[f\cong \coprod_{I} \id_{(0,\infty)}\]
for some set $I$. 		
\end{observation}
\begin{proof} By Theorem~\ref{thm:DESCENTsinglePRIME}, there is an equivalence
	\[
	\varphi^*\colon\Set\xrightarrow{\sim}\D.
	\]
Hence, there exists a set $I$ such that $$\varphi^*(I)\cong F.$$ 
We now translate this isomorphism into the language of \'{e}tale bundles.
	\begin{enumerate}[label=(\alph*)]
		\item Represent the singleton $\{*\}\in\Set$ as the bundle $\id_{\{*\}}\colon\{*\}\rightarrow \{*\}$. Then, any set $S$ can be represented as the disjoint coproduct $\coprod_{S}\id_{{\{*\}}}$;	
	\item  The functor $\varphi^*$ acts on bundles by pulling them back along the unique map $(0,\infty)\to \{\ast\}$. In particular, we get $\varphi^*(\id_{\{*\}})\cong \id_{(0,\infty)}$, as exhibited by the pullback square below
	\[\begin{tikzcd}
	(0,\infty) \ar[r] \ar[d,swap,"\id_{(0,\infty)}"] & \{*\} \ar[d,"\id_{\{*\}}"]\\
	(0,\infty) \ar[r] & \{*\}
	\end{tikzcd}\,\,.\]
	\item By item (a), the isomorphism $\varphi^*(I)\cong F$ may be written as
	$f\cong \varphi^*(\coprod_{I} \id_{\{\ast\}}).$ 
	\end{enumerate}

	Since $\gamma^*$ preserves arbitrary colimits, Items (a) - (c) assemble to give
	\begin{equation*}\label{eq:coprodARGUMENT}
	f\cong \varphi^*(\coprod_{I}\id_{\{*\}}) \cong 
	\coprod_{I}\varphi^*(\id_{\{*\}})\cong \coprod_{I}\id_{(0,\infty)}.
	\end{equation*}
\end{proof}

Here is the punchline. Let $(F,\theta)\in\D$ be a connected sheaf, i.e. $F$ corresponds to an \'{e}tale bundle $f\colon Y\rightarrow (0,\infty)$ where $Y$ is connected. Observation~\ref{obs:connNAsheaf} then forces $Y\cong (0,\infty)$, as illustrated in Figure~\ref{fig:flattenedsheafna}.
\begin{figure}[H]
	\centering
	\includegraphics[width=0.42\linewidth]{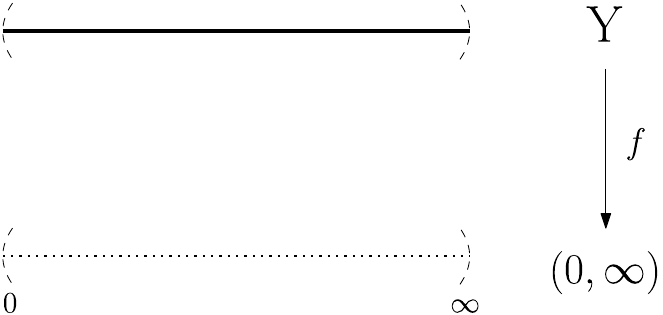}
	\caption{A connected sheaf $F$ of $\D$}
	\label{fig:flattenedsheafna}
\end{figure}

 Since $\D'\not\simeq\Set$ (Theorem~\ref{thm:ARCHIMEDEANPLACE}), one naturally suspects that the connected sheaves of $\D'$ are no longer quite as simple. The following example gives the first indications of this.

\begin{example}\label{ex:(0,alpha)} Define $(F,\theta)\in\D'$ where
	\begin{itemize}
		\item $F$ corresponds to the inclusion map $f\colon (0,\alpha)\to(0,1]$, regarded as an \'{e}tale bundle over $(0,1]$;
		\item $\theta$ is the descent data whose first coordinate map is defined as
		\begin{align}\label{eq:theta0LAXsheaf}
		\theta_0\colon (0,\alpha) \times (0,1] & \longrightarrow (0,\alpha)\\
		(y,\beta) &\longmapsto (y\cdot\beta) \nonumber
		\end{align}
		[Why is this sufficient? Recall from Section~\ref{subsec:ElephantstDIRECTION} that descent data $\theta$ is determined by the first coordinate map $\theta_0$. One then easily verifies that our $\theta_0$ satisfies the unit and cocycle conditions.]
	\end{itemize} 
In particular, notice $(0,\alpha)$ is connected, yet $f((0,\alpha))\cong(0,\alpha)\not\cong(0,1]$. 
\end{example}

Example~\ref{ex:(0,alpha)} signals an interesting difference with $\D$: the connected sheaves of $\D'$ need not be homeomorphic to the base space $(0,1]$. In fact, they can turn out to be much more complicated:

\begin{example}[Tuning Fork Sheaf]\label{ex:tuningfork} Following Example~\ref{ex:(0,alpha)}, 
	define $(F,\theta),(F',\theta')\in\D'$ whereby:
	\begin{itemize}
		\item $F$ corresponds to the inclusion $f\colon (0,\frac{1}{2})\hookrightarrow (0,1]$; and
		\item $F'$ corresponds to the identity map $f'\colon (0,1]\to (0,1]$,
		\item The descent data $\theta$ and $\theta'$ are both given by multiplication, analogous to Equation~\eqref{eq:theta0LAXsheaf}.
	\end{itemize}
Now observe that the inclusion $(0,\frac{1}{2})\hookrightarrow(0,1]$ also induces a bundle map between $f$ and $f'$. Since $\D'$ is a topos, and toposes possess all pushouts, the cokernel pair of this bundle map exists, which we illustrate in Figure~\ref{fig:tuning-fork-sheaf1}. For obvious reasons, we call this
the \emph{Tuning Fork Sheaf}. Moreover, notice that the resulting bundle space is connected: the two branches are glued along the common subspace $(0,\frac12)$.\footnote{\emph{Details.} Let $g\colon Z\to(0,1]$ denote the corresponding \'{e}tale bundle. Define two global sections $p_1,p_2\colon (0,1]\to Z$, one which maps the subspace $[\frac{1}{2},1]$ to the lower branch of $Z$, while the other maps it to the upper branch. Now suppose $h\colon Z\to S$ is a map to a discrete space. Since $(0,1]$ is connected, $h\circ p_1$ and $h\circ p_2$ are constant.  Now let $\gamma\in[\frac{1}{2},1]$ and $\gamma'\in(0,\frac{1}{2})$. Since $p_1(\gamma')=p_2(\gamma')$, conclude that $h$ is constant by observing: $h\circ p_1 (\gamma)=h\circ p_1(\gamma')=h\circ p_2(\gamma')=h\circ p_2(\gamma)$.
}


	
	\begin{figure}[ht!]
		\centering
		\includegraphics[width=0.3\linewidth]{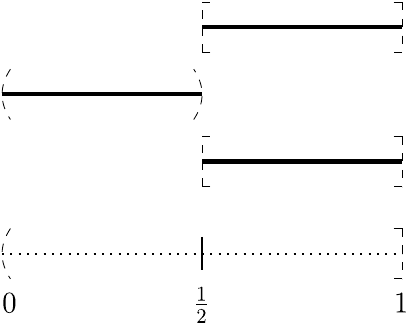}
		\caption{The Tuning Fork Sheaf of $\D'$}
		\label{fig:tuning-fork-sheaf1}
	\end{figure}
\end{example}
\begin{discussion}\label{dis:AvsNAforksheaf} The construction in Example~\ref{ex:tuningfork} is fairly flexible, and can be used to construct a wide variety of forking structures in the [connected components of the] sheaves of $\D'$. This gives a geometric way of reading the difference between standard descent vs. lax descent. In the non-Archimedean case, the rich sheafy structure is completely flattened by standard descent: as shown by Observation~\ref{obs:connNAsheaf}, all connected sheaves must be homeomorphic to the base space. By contrast, for lax descent, non-trivial forking still persists in the connected sheaves, as exhibited by the Tuning Fork Sheaf. 
\end{discussion}

Discussion~\ref{dis:AvsNAforksheaf} gives an insight into the kind of sheaves present in $\D'$
but absent in $\D$. We now turn to the complementary question: which forking configurations are ruled out by lax descent?

\begin{example}\label{ex:cantor} Reviewing Example~\ref{ex:tuningfork}, totice there was nothing special about our choice of inclusion map $(0,\frac{1}{2})\hookrightarrow(0,1]$. In fact, one can iterate the argument to obtain the sheaf as illustrated in Figure~\ref{fig:morebranching}. There is, however, a limit to this construction. If the forking is iterated indefinitely along every branch, as in Figure~\ref{fig:cantorset}, the fibre over $1$ becomes the Cantor space. Since this fibre is not discrete, the resulting bundle cannot be \'{e}tale. In other words, Figure~\ref{fig:cantorset} does not define a sheaf over $(0,1]$.
	
	\begin{figure}[H]
		\centering
		\subfloat[Iterated Forking]{	\includegraphics[width=0.44\linewidth]{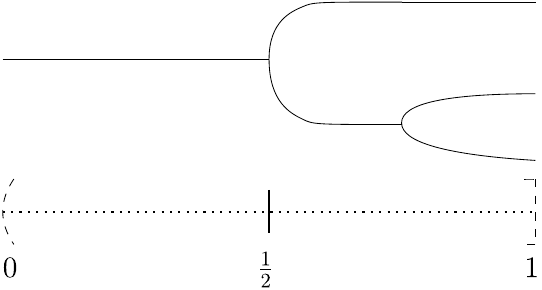}
			\label{fig:morebranching}}\quad
		\subfloat[A Forbidden Configuration]{	\includegraphics[width=0.44\linewidth]{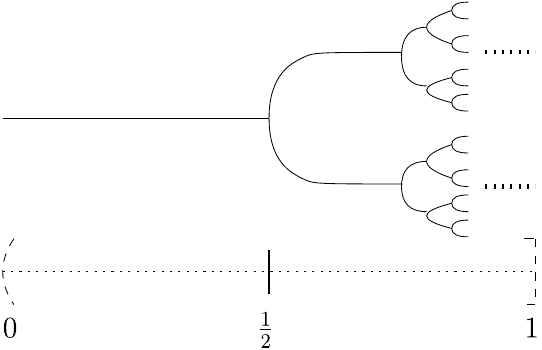}
			\label{fig:cantorset}}	
		\caption{}\label{fig:Label}
	\end{figure}
	
\end{example}

Let us sharpen our terminology. Given a sheaf $F\in\baseS(0,1]$, say that $F$ witnesses \emph{upper bound forking} if some connected component has two branches on the right of a branching point and a single on its left (as illustrated in Figure~\ref{fig:upper-bound-forking1}). Analogously, say that $F$ witnesses \emph{lower bound forking} if there exists two branches to the left of the branching point and one on its right (as illustrated in Figure~\ref{fig:lowerboundforking1}). In principle, there may be multiple instances of forking (see, e.g. Figure~\ref{fig:morebranching}), but we shall always assume that the branches of the fork do not ``join'' back up. 

\begin{figure}[H]
	\centering
	\subfloat[Upper Bound Forking]{	\includegraphics[width=0.35\linewidth]{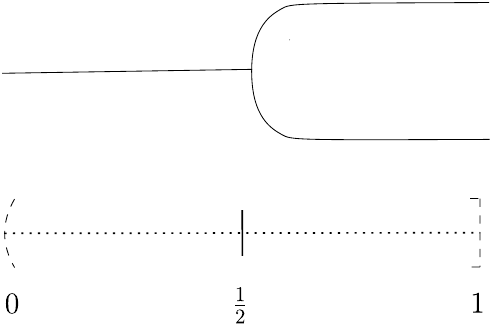}
		\label{fig:upper-bound-forking1}}\quad
	\subfloat[Lower Bound Forking]{	\includegraphics[width=0.35\linewidth]{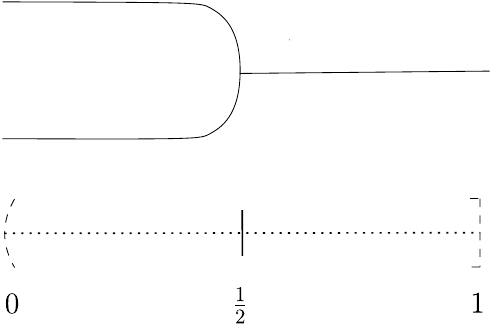}
		\label{fig:lowerboundforking1}}	
	\caption{Two Types of Forking with Branching Point at $\frac{1}{2}$}\label{fig:Forking}
\end{figure}

\begin{claim}\label{claim:ELIMlowerFORK} 	For any $(F,\theta)\in\D'$, $F\in\baseS(0,1]$ does not witness lower bound forking.
\end{claim}
\begin{proof} Suppose, for contradiction, that $F$ witnesses an instance of lower bound forking. Let $Y_i$ be a connected component in which this occurs, at a branching point $\gamma_0\in(0,1]$.  Thus, there exists a connected neighbourhood $V_{\gamma_0}$ of a point $\gamma_0\in (0,1]$ such that
	\[
	|Y_i\cap F(\alpha)|=
	\begin{cases}
	2,&\alpha\leq\gamma_0,\\
	1,&\alpha>\gamma_0,
	\end{cases}\qquad \alpha\in V_{\gamma_0}\,.
	\]
See Figure~\ref{fig:lowerforkingACTUAL} for an illustration. We shall fix this neighhbourhood $V_{\gamma_0}$ throughout the proof. Informally, $V_{\gamma_0}$ records the local geometry of the forking, which we shall use to fix a ``witnessing point'' $\beta\in V_{\gamma_0}$; the subsequent
argument will then propagate this local information to the rest of
$(\gamma_0,1]$ via the descent data.
	
%
	
	
	\begin{figure}[ht!]
		\centering
		\subfloat[The Connected Component $Y_i$]{	\includegraphics[width=0.37\linewidth]{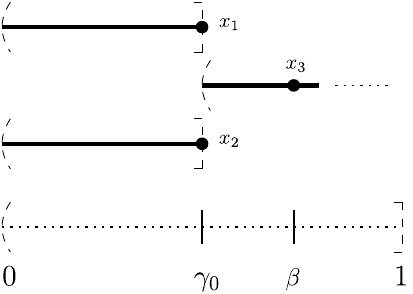}
			\label{fig:lowerforkingACTUAL}}\quad
		\subfloat[Another Forbidden Configuration]{\label{fig:lowerforkingFORBIDDEN}	\includegraphics[width=0.37\linewidth]{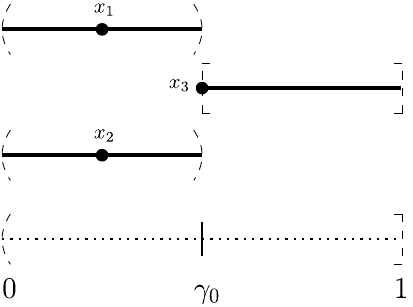}
		}	
		\caption{}	
		\label{fig:LowerBoundForking}
	\end{figure}
	
	\subsubsection*{Step 1: The fibre at $\gamma_0$} Moving from right to left, the fork begins \emph{at} $\gamma_0$: necessarily
	\[
	|Y_i\cap F(\gamma_0)|=2.
	\]
	Indeed, the alternative configuration in
	Figure~\ref{fig:lowerforkingFORBIDDEN}, where the two branches merge
	at $\gamma_0$, does {\em not} define an étale bundle. Why? If $x_3$ denotes the
	merged point, every open neighbourhood of $x_3$ meets both branches.
	Hence no such neighbourhood can map homeomorphically onto an open
	neighbourhood of $\gamma_0$ in $(0,1]$. 

	\subsubsection*{Step 2: Analysis of Descent Data} We now examine how the descent data interacts with the forking
	structure of $Y_i$. Recall from Key Claim~\ref{claim:descent2upperreals}
	that $\theta$ induces maps
	\[
	\theta_{\gamma'\gamma}\colon F(\gamma')\longrightarrow F(\gamma),
	\qquad \gamma\leq\gamma',
	\]
	satisfying the cocycle and unit conditions.

%
	
	\subsubsection*{Step 2a: Descent action restricts to $Y_i$} For any $\beta\in(\gamma_0,1]$, pick $x_3\in F(\beta)\cap Y_i$. Consider the map
	\[
	\widetilde\theta\colon(0,\beta]\longrightarrow Y,
	\qquad
	\gamma\longmapsto\theta_{\beta\gamma}(x_3).
	\] 
	By the unit condition, $\widetilde\theta(\beta)=\theta_{\beta\beta}(x_3)=x_3\in Y_i$. Since $(0,\beta]$ is connected, $\phi((0,\beta])$ is also connected, and is thus contained in
	the connected component $Y_i$. Thus $\theta_{\beta\gamma}(x_3)\in Y_i$ for every $\gamma\in (0,\beta]$. 
	
	We leverage this to exploit the lower bound forking at $\gamma_0$. By our setup, choose $\beta\in V_{\gamma_0}$ such that
	$|Y_i\cap F(\beta)|=1$. By Step~1, there exist distinct points
	$x_1,x_2,x_3\in Y_i$ such that
$$	x_1,x_2\in F(\gamma_0),
\qquad
x_3\in F(\beta),$$
as depicted in Figure~\ref{fig:lowerforkingACTUAL}. The preceding argument then gives
\begin{equation}\label{eq:Step2a-forking}
	\theta_{\beta\gamma_0}(x_3)\in Y_i\cap F(\gamma_0)
=\{x_1,x_2\}.
\end{equation}
	
	
	\subsubsection*{Step 2b: No jumps} Folllowing our choice of $\beta$ from Step 2a, assume without loss of generality that $\theta_{\beta\gamma_0}(x_3)=x_1$. Now let $\gamma\in (\gamma_0,1]$, and pick any $y\in Y_i\cap F(\gamma)$. There are two cases to check.\footnote{This case-splitting is constructive by
		\cite[Proposition~1.37]{NV}; in particular, the interval $(\gamma_0,1]$ can be obtained
		(constructively) as the pushout of $(\gamma_0,\beta]$ and $[\beta,1]$
		along their common point $\beta$. 
	
} 
	\begin{itemize}
		\item[] \textbf{Case 1:} $\gamma\in (\gamma_0,\beta]$. In which case, $\gamma\in V_{\gamma_0}$ and so $|Y_i\cap F(\gamma)|=1$. Since Step 2a gives $\theta_{\beta\gamma}(x_3)\in Y_i\cap F(\gamma),$ deduce that $\theta_{\beta\gamma}(x_3)=y$. The cocycle condition therefore gives 
		\[x_1=\theta_{\beta\gamma_0}(x_3)=\theta_{\gamma\gamma_0}\circ\theta_{\beta\gamma}(x_3)=\theta_{\gamma\gamma_0}(y).\]
	\item[] \textbf{Case 2:} $\gamma\in [\beta,1]$. In which case, Step 2a gives $\theta_{\gamma\beta}(y)\in Y_i\cap F(\beta)$ and so $\theta_{\gamma\beta}(y)=x_3$ (notice the reversal of $\beta$ and $\gamma$). The cocycle condition again gives
	$$x_1=\theta_{\beta\gamma_0}(x_3)=\theta_{\beta\gamma_0}\circ\theta_{\gamma\beta}(y)=\theta_{\gamma\gamma_0}(y). $$
	\end{itemize}
 In sum: once we restrict to the connected component $Y_i$, every point in the fibre over $(\gamma_0,\beta]$ descends to $x_1$ on the upper branch; the descent data cannot ``jump'' from one branch to the other.
 
	
	\subsubsection*{Step 2c: A contradiction} Recall that the proof of Key Claim~\ref{claim:descent2upperreals} (iii) involved verifying two implications, which we reproduce below for the reader's convenience:
	\begin{enumerate}[label=(\alph*)]
		\item $x\in F(\gamma )\implies\exists q\in I_\gamma  .\big(\exists y\in F(q).(x=\theta_{q\gamma }(y)\big)$
		\item $y,z\in F(q),\theta_{q\gamma}(y)=\theta_{q\gamma}(z)\implies \exists r\in I_\gamma. (q\prec r \land \theta_{qr}(y)=\theta_{qr}(z))$
	\end{enumerate}
	Read in our context, Implication (a) says: given $x_2\in F(\gamma_0)$, which lives on the lower branch of Figure~\ref{fig:lowerforkingACTUAL},  there exists some $q>\gamma_0$, and some $y\in F(q)$ such that $x_2=\theta_{q\gamma_0}(y)$.\footnote{Notice by definition of lower bound forking, $\gamma_0\in(0,1)$, so we avoid the edge case of $\gamma_0=1$. Hence the translation of ``$\,\exists q\in I_{\gamma}$'' as `` there exists $q>\gamma_0$'' is justified.} But Step 2b forces the identity  $\theta_{q\gamma_0}(y)=x_1\neq x_2$, giving a contradiction. This finishes the proof.
	

\end{proof}

Reviewing our work in this section, we present the following summary answer to Question~\ref{qn:forkLAXdescent}.

\begin{conclusion}\label{conc:forking}\hfill 
	\begin{enumerate}[label=(\roman*)]
		\item Standard descent eliminates all forms of forking in the sheaves of $\D$. 
		\item Although upper bound forking persists in the sheaves of $\D'$, lax descent eliminates lower bound forking. 
		\item Lax descent forces the image of every nonempty sheaf to be downward closed. More precisely, if $(F,\theta)\in\D'$ and $F$ corresponds to an \'{e}tale bundle $f\colon Y\rightarrow (0,1]$, then either $$f(Y)\cong (0,1] \qquad\text{or}\qquad f(Y)\cong (0,\alpha) \,\,\,\,\text{for some}\,\, \alpha\in (0,1].$$
	\end{enumerate}
\end{conclusion}
\begin{proof} (i) is by Observation~\ref{obs:connNAsheaf}, (ii) is Example~\ref{ex:tuningfork} and Claim~\ref{claim:ELIMlowerFORK}. For (iii), suppose $\gamma\in f(Y)$, so there exists some $y\in Y$ such that $f(y)=\gamma$. Recall from Observation~\ref{obs:theta-descent-prop} that the lax descent data gives $$\theta_0(y,\beta)\in Y \qquad\text{for any}\, \beta\in (0,1]\,.$$ 
Moreover, $(0,1]$ has the following action on fibres:  $$f(\theta_0(y,\beta))=f(y)\cdot \beta\qquad\text{for any}\, \beta\in (0,1]\,.$$
Now suppose $\gamma'\in(0,\gamma]$. Setting $\beta:=\gamma^{-1}\cdot \gamma'\in (0,1]$, the lax descent data gives the point $y':=\theta_0(y,\beta)$ satisfying
	\[f(y')=f(y)\cdot\gamma^{-1}\cdot\gamma'=\gamma\cdot\gamma^{-1}\cdot\gamma'=\gamma'.\]
Thus $\gamma'\in f(Y)$, proving that $f(Y)$ is downward closed.

Notice: in contrast to the standard descent case, we do not get upward closure of $f(Y)$ since we only have a (non-invertible) monoidal action induced by $(0,1]$.
\end{proof}

\section{A Strange Woods}\label{sec:strangewoods}
\begin{fquote}[Andr\'{e} Weil, letter to his sister \cite{WeilSister}]
	That is what one achieves, and in a very satisfactory manner, too, in the theory of “valuations” [\dots]  To define a prime ideal in a field (a field given
	abstractly) is to represent the field “isomorphically” in a p-adic field: to represent it in the same
	way in the field of real or complex numbers, is (in
	this theory) to define a “\textbf{prime ideal at infinity}”.
\end{fquote}

In mathematics, difficult problems are often approached by breaking them into smaller, more manageable pieces. This approach leads to two fundamental questions:

\begin{enumerate}
	\item How do we account for all the different pieces of the problem?
	\item How and when can we glue the pieces to form a global solution, and what are the obstructions to this reassembly?
\end{enumerate}

The Hasse local-global principle from number theory provides a particularly striking example. Here one seeks to understand whether solutions to polynomial equations over $\Q$ can be reassembled from its solutions over completions of $\Q$, i.e. the reals $\R$ and the $p$-adics $\Q_p$. In the classical picture, apparently going back to Hasse and/or Artin \cite{WeilSister}, the non-Archimedean completions are indexed by the nonzero prime ideals of $\Z$, while the real place is regarded as an additional ``prime ideal at infinity''. This viewpoint underlies a number of powerful frameworks, including Arakelov geometry, as discussed in the Introduction.

The results of this paper provide an important nuance to this perspective. When re-examined from the topos-theoretic perspective, we find that the Archimdean place does not behave like an individual point with no intrisic features; rather, it is identified with the space $\overleftarrow{[0,1]}$ of upper reals (Theorem~\ref{thm:ARCHIMEDEANPLACE}). Thus, accounting for the Archimedean place is not simply a matter of adjoining an additional point at infinity because the quotient carries an intrinsic topological structure, invisible in the classical point-set formulation. 

\smallskip

We emphasise that this is not just an artefact of passing to point-free topology. Rather, it arises from analysing the internal struture of the places as equivalence classes, which was accomplished through the use of descent techniques. Very informally, the difference can be understood in terms of the symmetry of the algebraic actions underlying the Archimedean and non-Archimedean cases. As discussed in Section~\ref{sec:NonArch}, after fixing a non-zero prime ideal $\frap$ of $\Z$, exponentiation induces a free and transitive group action on the corresponding space of non-Archimedean absolute values. Thus, in a sense, all its points are related to one another by invertible symmetries. Passing to the quotient through descent consequently collapses these symmetries, leaving a singleton  (Theorem~\ref{thm:DESCENTsinglePRIME}). By contrast, the space of Archimedean absolute values carries only a monoid action. This action introduces non-invertible comparisons between the points, producing an intrinsic asymmetry that survives the descent process. This asymmetry is reflected in the order topology of the resulting quotient space $\overleftarrow{[0,1]}$ (Theorem~\ref{thm:ARCHIMEDEANPLACE}).
\smallskip 

We are now at a very interesting mathematical juncture. We began this paper with a question originating in number theory: how might one build a framework that unifies the reals and $p$-adic fields, whilst also accommodating their differences (Question~\ref{qn:TENSION})?  Various subtleties surrounding this tension were discussed. Over the course of this paper, our understanding of the interplay between the Archimedean and non-Archimedean perspectives has begun to shift. Our topos-theoretic analysis has brought some aspects of the picture into focus that were previously obscured. Along the way, the analysis also raises a more fundamental question about the tools themselves: what information, precisely, does topos theory capture?

Looked at from a certain distance, these questions arising from number theory and topos theory start to converge on a common theme: \emph{how should the connected and the disconnected interact?} The remainder of this (primarily expository) section explores this question from both perspectives.

\subsection{The View from Topos Theory}

\subsubsection{What do Classifying Toposes Classify?} An important clue in our investigation of the non-Archimedean places was the following example by Bunge \cite{Bunge}, which we now discuss more fully:

\begin{example}\label{ex:Bunge} Let $G:=(G_0,G_1)$ be a connected localic group. Then $BG\simeq \Set$.
\end{example}
\begin{proof} Recall $\BG$ may be characterised as the category of \'{e}tale $G$-spaces, whose objects are \'{e}tale bundles $E\xrightarrow{p} G_0$ equipped with a (continuous) $G_1$-action $G_1\times_{G_0}E\xrightarrow{\bigcdot} E$ satisfying the usual axioms \cite[\S 5]{Moerdijk}. First, notice that $G_0=\{*\}$ since $G$ is a group. Since \'{e}tale bundles are fibrewise discrete (Discussion~\ref{dis:fibrewiseDISCRETE}), the bundle space of any \'{e}tale $G$-space $E\xrightarrow{p} G_0$ must therefore be discrete. 
	
Next, given $e\in E$, the $G_1$-action induces a (continuous) map
	\begin{align*}
	\argu \bigcdot e\colon G_1&\longrightarrow E\\
	g&\longmapsto g\bigcdot e.
	\end{align*}
Since $E$ is discrete and $G_1$ is connected, this map must be constant; in fact, $g\bigcdot e = e$ for all $g\in G_1$ since the $G_1$-action forces the identity $s(\ast)\bigcdot e =e$ where $s\colon G_0\to G_1$ is the unit map. Hence, since the objects of $\BG$ are just sets equipped with trivial $G_1$-action, conclude that $\BG\simeq \Set$. 
\end{proof}

The following comments give some context as to why Example~\ref{ex:Bunge} is interesting. 

\begin{discussion} Example~\ref{ex:Bunge} gave us our first indication that the topos of a single non-Archimedean place may in fact be trivial --- contrary to the expectations of Guess~\ref{guess:DhasAUTOMORPHISMS}. Although the eventual proof of Theorem~\ref{thm:DESCENTsinglePRIME} did not require the hypothesis that the $(0,\infty)$-action of Groupoid~\eqref{eq:simplicialLOCALP} is connected, one can in fact adapt the argument of Example~\ref{ex:Bunge} to give an alternate (though much more involved) proof to show that its descent topos trivialises.

\end{discussion}

\begin{discussion}\label{dis:classifyDIStorsors} The trivialisation result is also striking because it contravenes a basic expectation from the discrete setting. By Diaconescu's Theorem, the presheaf topos $\BG\simeq [G,\Set]$ classifies all $G$-torsors for any discrete group $G$ \cite[B3.2]{Elephant}. Yet when $G$ is a connected group, the fact that  $\BG\simeq \Set$ implies that for each topological space $X$ there exists essentially only one geometric morphism $\baseS X\to \BG$, even though for suitable $X$ and $G$ there may exist many non-isomorphic $G$-torsors. 
\end{discussion}

Reviewing Example~\ref{ex:Bunge}, the core mechanism of the argument rests on two general facts:
\begin{enumerate}[label=(\alph*)]
	\item All sheaves over localic spaces can be characterised as fibrewise discrete bundles (Discussion~\ref{dis:fibrewiseDISCRETE}).
	\item All maps from connected spaces into discrete sets must be constant. 
\end{enumerate}
Put together, these observations suggest that the trivialisation we encountered is not an accident, but instead a more general feature of sheaf toposes. The following observation by Lurie confirms this suspicion, and makes it precise.


\begin{observation}[Lurie's Observation {{\cite{LuriePrincipalBundle}}}]\label{obs:Lurie} Let $G$ be a localic group such that there exists a non-constant continuous map $\nu\colon \R\to G$. \underline{Then}, there cannot exist a topos $\baseS\E$ that classifies $G$-torsors over localic spaces, i.e. there does not exist a topos $\baseS\E$ such that
	\[\textbf{Geom}(\baseS X,\baseS\E)\simeq \Tor_{G}(X),\] 
	where $X$ is a localic space and $\Tor_{G}(X)$ denotes the category of $G$-torsors over $X$. 	
\end{observation}

\begin{proof} We give a sketch of the proof to highlight the role of connectedness; for details, see \cite[Observation 6.5.5]{NgThesis}.\footnote{We emphasise that the credit goes to Lurie.} Since $\R$ is connected, the unique projection map $p\colon \R\to\{*\}$ induces a fully faithful embedding on the level of their sheaf toposes
	\begin{equation*}
	p^*\colon \Set\hookrightarrow  \baseS\R.
	\end{equation*}
 Lurie's Observation then follows from the following two claims:
 \begin{enumerate}[label=(\alph*)]
 	\item For any topos $\baseS\E$, $p^*$ induces a fully faithful embedding 
 	\begin{equation*}\label{eq:p*inducedmap}
 	\textbf{Geom}(\Set,\baseS\E)\hookrightarrow \textbf{Geom}(\baseS\R,\baseS\E).
 	\end{equation*}	
 	\item There does not exist a fully faithful embedding 
 	\begin{equation*}
 	\Tor_G(\{*\})\hookrightarrow \Tor_{G}(\R).
 	\end{equation*}
 \end{enumerate}
Now suppose there exists a topos $\baseS\E$ such that 
	\[\textbf{Geom}(\baseS X,\baseS\E)\simeq \Tor_{G}(X),\]
	for any localic space $X$. By Claim (a), this equivalence yields a fully faithful embedding
	\[ \Tor_G(\{*\})\simeq \textbf{Geom}(\Set,\baseS\E)\hookrightarrow \textbf{Geom}(\baseS\R,\baseS\E) \simeq \Tor_G(\R),\]
	contradicting Claim (b).
\end{proof}



\begin{discussion}[{{$\infty$-toposes}}]\label{dis:inftyTOPOSclassification} Interestingly, Lurie \cite{LuriePrincipalBundle} points out that the argument of Observation~\ref{obs:Lurie} can be extended to show that $\infty$-toposes also do not classify $G$-torsors for all topological groups $G$ either. This is \emph{a priori} surprising: one may have expected that the generality of $\infty$-toposes would resolve the previous issues faced by the standard topos.\footnote{This becomes less surprising when we re-examine the motivations behind higher topos theory. Higher topos theory, as set out in \cite{LurieHigherTopos}, aims to develop a categorical framework for interpreting higher cohomology classes, analogous to how $G$-torsors describe first cohomology classes. Thus, the objective is distinct from classifying all $G$-torsors when $G$ is topological/localic.
	} In any case, Lurie's remark suggests that other frameworks are needed if we wish to deal with $G$-torsors for general topological/localic $G$.
\end{discussion}

Lurie's Observation~\ref{obs:Lurie} is powerful because it is no longer about the failure of the particular construction $\BG$; rather it rules out the possibility of classifying topos altogether for a group $G$ satisfying its hypotheses. Of course, other frameworks for classifying $G$-torsors exist  (e.g. topological stacks \cite{Noohi}), but it is less clear if they have a point-free interpretation in the sense that was important to this paper (cf. Definition~\ref{def:ptFREEspace}). More specifically, we may ask:

\begin{problem}\label{prob:Gbundles} Can the classification of $G$-torsors for general topological/localic groupoids $G$ be given a point-free interpretation, compatible with geometric logic?
\end{problem}


\subsubsection{Lax Descent} Given Lurie's Observation~\ref{obs:Lurie}, the reader may be forgiven for thinking that toposes are generally incapable of retaining meaningful
information about connected spaces. The reality, however, is more nuanced. Consider
Theorem~\ref{thm:ARCHIMEDEANPLACE} once more: although our descent Diagram~\eqref{eq:ArchLAXtopos} was constructed using the interval $(0,1]$ -- ostensibly a connected space -- the resulting quotient space is certainly non-trivial (in contrast to the non-Archimedean case). In a different vein, Conclusion~\ref{conc:forking} highlights some of the subtle and non-trivial effects that lax descent has on the connected components of sheaves. Together, these results highlight a striking parallel between the logical and sheaf-theoretic perspectives. At the level of point-free spaces,  the lax descent eliminates the left Dedekind sections of $(0,1]$ but retains the right Dedekind sections as upper reals. At the level of sheaves, the lax descent eliminates lower-bound forking but permits upper-bound forking in the connected components of the sheaves on $(0,1]$. 

\smallskip
Stepping back for a moment, one may ask if the lax descent calculations underlying Theorem~\ref{thm:ARCHIMEDEANPLACE} and Conclusion~\ref{conc:forking} point to a broader topos-theoretic principle. In particular, the observations above suggest that lax descent may admit a more general description in terms of its effects on the connected components of sheaves. Here is one possible approach to this problem. In the non-Archimedean case, the proof of Theorem~\ref{thm:DESCENTsinglePRIME} essentially reduced to observing that the map $$(0,\infty)\to \{\ast\}$$
appearing in Diagram~\eqref{eq:simplicialLOCALP} is an open surjection, and thus satisfies effective descent. This suggests the following test problem:

\begin{problem} Can the characterisation of the Archimedean place in Theorem~\ref{thm:ARCHIMEDEANPLACE} likewise be obtained from a general lax descent theorem?
\end{problem}

\begin{discussion} The existing literature on effective lax descent is a natural first place to start looking for answers; see e.g. \cite{Pitts,MoVer00,BungePittsLax}. In particular, Bunge \cite{BungePittsLax} gives a fairly general approach unifying various effective lax results via her framework of Pitts KZ-monads. Our explicit calculations raise the question of how the behaviour of lax
descent on connected components exhibited here fits into Bunge's framework, or more generally into the overall theory of effective lax descent.
	
Notice, however, that such an interpretation would only address the mechanism underlying our calculation of the Archimedean place. One would still need to understand how the resulting Archimedean and non-Archimedean pieces assemble into a single point-free space of places. This brings us to the more fundamental Problem~\ref{prob:places}, which we discuss next.
\end{discussion}

\subsection{Local-Global Questions} Section~\ref{sec:global-places} defined the topos of places $\calP$ as the lax descent topos of Diagram~\eqref{eq:PLACES-topos}. Thus far, however, our calculations have been local in nature: the trivial place and non-Archimedean places of $\Q$ are identified as singlestons, with the Archimedean place as the unit interval of upper reals: 
\[
\begin{array}{rcl}
\text{trivial place} &\longleftrightarrow& \{*\}\\
\text{non-Archimedean place at }\frap &\longleftrightarrow& \{*\}\\
\text{Archimedean place} &\longleftrightarrow& \overleftarrow{[0,1]} 
\end{array}
\]
This only gives a piecewise account of the space of places of $\Q$. Taking seriously our objective of treating the places of $\Q$ as an actual space (as opposed to an indexing set), let us restate our original motivating problem from the Introduction.

\begin{problem}\label{prob:places} Characterise the (entire) point-free space of places of $\Q$. 
\end{problem}

\begin{discussion}[The trivial place] As stated, Question~\ref{qn:TENSION} frames the problem as reconciling the differences between the Archimedean and non-Archimedean perspectives. Our analysis highlights a more immediate difficulty of reconciling the trivial place with the (non-trivial) non-Archimedean places.\footnote{We remind the reader that this paper defines the non-Archimedean absolute values as absolute values $|\cdot|$ such that $|p|<1$ for some prime $p$; in particular, they are non-trivial by definition.}
	
 Theorem~\ref{thm:ALLNAplace} shows that the space of non-Archimedean places behaves like the set of points indexed by the non-zero prime ideals of $\Z$. However, as already remarked, our methods (essentially, Moerdijk's Stability Theorem~\ref{thm:MoerdijkSTABILITY}) do not allow us to extend this to the trival place. What is interesting is that this does not appear to be just a technical limitation of our descent tools; in fact, the same difficulty is already visible at the level of absolute values.
 
This was the very subject of \cite[Problem 6.1]{NVOstrowski} from an earlier manuscript, which asks for the characterisation of the point-free space of {\em absolute values on $\Z$}. In that paper, we were able to give a complete characterisation of the space of {\em multiplicative seminorms on $\Z$} valued in the upper reals (Theorem 4.1, {\em loc. cit.}). Curiously, when we considered the space of {\em Dedekind-valued} absolute values, we were unable to even characterise the space of ultrametric absolute values.\footnote{Recall from Section~\ref{subsec:globalNA} that the ultrametric absolute values are defined as the absolute values $|\cdot|$ satisfying the ultrametric inequality -- this includes {\em both} the trivial and non-Archimedean absolute values.} The closest approximation we found was an Ostrowski-type result, stated as Theorem~\ref{thm:ostrowskiQ} here, which gives a local
 account that identifies the space of absolute values belonging to the same place.
\end{discussion}

\begin{discussion}[Archimedean place as a parameter space]\label{dis:Arch-parameter} Recall Artin-Whaples's classical result \cite{ArtinWhaples} that $\Q$ (in fact, all global fields) satisfies a product formula
	\begin{equation}\label{eq:productformula}
	\prod_{v\in\Lambda_\Q}|x|_v=1, \qquad \text{for all}\, x\neq 0 \qquad ,
	\end{equation}
	where $v$ ranges over all the places of $\Q$, including the Archimedean place. Notice the product formula is preserved under simultaneously rescaling the absolute values by any choice of exponent
	$\alpha\in[0,1]$. In light of our Theorem~\ref{thm:ARCHIMEDEANPLACE}, this suggests viewing the Archimedean place as a parameter space, capturing the normalisation across {\em all} places. 
	
This perspective has interesting precedents in other approaches to arithmetic geometry --- e.g. in the role of the deformation parameter $0<\alpha\leq 1$ in Connes--Consani's framework of {\em Absolute Algebraic Geometry} \cite[\S4.1]{ConnesConsaniATIYAH}, and, in a rather different way, in the role of the Archimedean error $\epsilon\geq0$ in the framework of {\em Globally Valued Fields}	\cite{GVF}. In our point-free context, fixing a normalisation parameter from the Archimedean place amounts to working internally over the topos $\baseS\overleftarrow{[0,1]}$  (Convention~\ref{conv:fixingx}). Equivalently, as illustrated by Figure~\ref{fig:candidatepicture}, one may conjecturally regard the Archimedean place as lying \emph{below} $\Spec(\Z)$:
	\begin{figure}[H]
		\centering
		\includegraphics[width=0.7\linewidth]{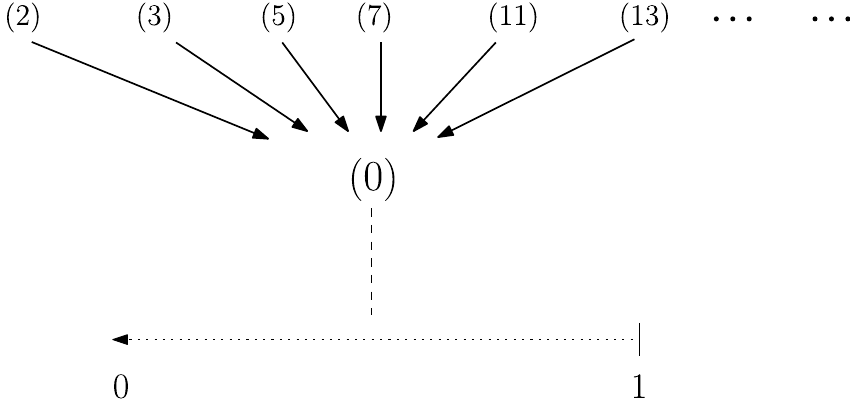}
		\caption{A candidate picture for the space of places}
		\label{fig:candidatepicture}
	\end{figure}
\noindent If this picture is correct, it should have interesting implications for our understanding of topos-theoretic descent. On the side of number theory, the perspective that the Archimedean place lives below $\Spec(\Z)$ may be suggestive to readers familiar with $\mathbb{F}_1$-geometry; see e.g. \cite{PenaLorscheid}.
\end{discussion}

Let us also take the opportunity to state two follow-up problems to Problem~\ref{prob:places}. The first is the original motivating test problem of my thesis \cite{NgThesis}:

\begin{problem}\label{prob:strangeCOMPLETIONS} Over the topos of places $\calP$, characterise the space of completions of $\Q$. In particular, how should we understand the generic Archimedean completion?
\end{problem}

\begin{discussion}\label{dis:Archcompletion} When restricted to the Archimedean place, we expect the construction from Problem~\ref{prob:strangeCOMPLETIONS} to produce what is morally a parametrised family of completions, interpolating between $\Q$ and $\R$.\footnote{There is a technicality in that we cannot direectly define an absolute value $|\cdot|^\alpha$ for a generic upper real $\alpha\in \overleftarrow{[0,1]}$: this would give an upper-valued absolute value on $\Q$, creating the same issues regarding multiplicative inverses raised in \cite[Obsservation 0.1]{NVOstrowski}. One possible
		constructive workaround is to consider the family $|\cdot|^q$ for $q\in(0,1]$.}
\begin{figure}[H]
	\centering
	\includegraphics[width=0.5\linewidth]{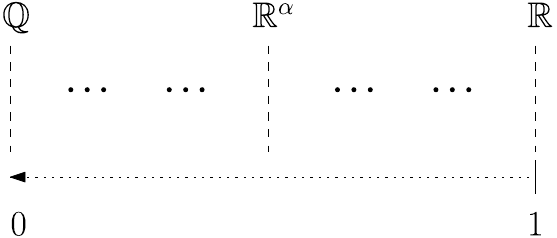}
	\caption{A parametrised family of completions over the Archimedean place}
	\label{fig:parametrisedfamilyofcompletions}
\end{figure}
\noindent Recall that the unit interval $\overleftarrow{[0,1]}$ of upper reals {\em must} contain 0 since subspaces of upper reals are closed under the Scott topology (Section~\ref{subsec:localicreals}). This suggests that that the trivial completion $\Q$ cannot be definably separated from the rest of this family. It is interesting to notice then that the Archimedean completions  $\R^\alpha$ becoming increasingly ``disconnected'' as $\alpha\to 0$ .	
\end{discussion}

\begin{discussion} Discussion~\ref{dis:Archcompletion} gives another way to appreciate some of the subtleties underlying Question~\ref{qn:TENSION}. Whenever one wants to import analytical methods from the Archimedean setting (e.g. complex analytification of varieties) to the non-Archimedean setting, the barriers to this translation are clear: non-Archimedean fields are totally disconnected, and so the naive analytification is of limited usefulness. Following Berkovich \cite{BerkovichMonograph}, however, we
now have a robust solution, which essentially involves ``filling in'' the gaps responsible for the disconnectedness of non-Archimedean spaces; see \cite{Payne} for a nice survey.
	
The opposite direction is less clear. In order to relate analytic structures over $\R$ to the (disconnected) arithmetic setting, Theorem~\ref{thm:ARCHIMEDEANPLACE} and Discussion~\ref{dis:Archcompletion} suggest that we should somehow parametrise families of analytic structures $\{\M_q\}$, perhaps over $\qint$, and examine their behaviour as $q\to0$. Ideas from condensed mathematics, in particular $p$-liquidification for $0<p\leq1$, appear relevant but more investigation is needed; see e.g. \cite[Theorem 3.11]{ClausenScholzeCondensed}.
\end{discussion}

Finally, recall that Grothendieck toposes trace their origins back to Grothendieck's attack on the Weil conjectures: the need to develop \'{e}tale cohomology (and other related cohomology theories) led him to replace ordinary topological spaces with the more general notion of a topos. Given the context of this paper, the following problem is natural:

\begin{problem} Characterise the sheaf cohomology $H^\bullet(\calP,F)$ of the topos of
	places, for suitable classes of coefficients $F$.  What geometric or arithmetic information do these cohomology groups capture? What local obstructions do they detect?
\end{problem}

\bibliography{thesis}

@article{MoVer00, title={Proper maps of toposes}, volume={148}, url={http://dx.doi.org/10.1090/memo/0705}, DOI={10.1090/memo/0705}, number={705}, journal={Memoirs of the American Mathematical Society}, publisher={American Mathematical Society (AMS)}, author={Moerdijk, I. and Vermeulen, J. J. C.}, year={2000}, language={en} }

@book{Elephant,
	author = {Johnstone, Peter T.},
	isbn = {0-19-853425-6},
	mrclass = {18B25 (18-02)},
	mrnumber = {1953060},
	mrreviewer = {Colin\ McLarty},
	pages = {xxii+468+71},
	publisher = {The Clarendon Press, Oxford University Press, New York},
	series = {Oxford Logic Guides},
	title = {Sketches of an {E}lephant: a topos theory compendium. {V}ol. 1 and 2},
	volume = {43},
	year = {2002}}

@article{NgBerk,
	author = {Ming Ng},
	journal = {Journal of Symbolic Logic, to appear},
	title = {Logical {B}erkovich Geometry: a point-free perspective},
	year = {2026}}

@article{BY08,
	abstractnote = {<jats:p> We present an adaptation of continuous first order logic to unbounded metric structures. This has the advantage of being closer in spirit to C. Ward Henson's logic for Banach space structures than the unit ball approach (which has been the common approach so far to Banach space structures in continuous logic), as well as of applying in situations where the unit ball approach does not apply (i.e. when the unit ball is not a definable set). </jats:p><jats:p> We also introduce the process of single point emboundment (closely related to the topological single point compactification), allowing to bring unbounded structures back into the setting of bounded continuous first order logic. </jats:p><jats:p> Together with results from [4] regarding perturbations of bounded metric structures, we prove a Ryll--Nardzewski style characterization of theories of Banach spaces which are separably categorical up to small perturbation of the norm. This last result is motivated by an unpublished result of Henson. </jats:p>},
	author = {Ben Yaacov, Ita{\"I}},
	doi = {10.1142/s0219061308000737},
	journal = {Journal of Mathematical Logic},
	language = {en},
	month = dec,
	number = {02},
	pages = {197--223},
	publisher = {World Scientific Pub Co Pte Lt},
	title = {CONTINUOUS FIRST ORDER LOGIC FOR UNBOUNDED METRIC STRUCTURES},
	url = {http://dx.doi.org/10.1142/S0219061308000737},
	volume = {08},
	year = {2008}}

@article{Tsi18,
	author = {Tsimerman, Jacob},
	doi = {10.4007/annals.2018.187.2.2},
	journal = {Annals of Mathematics},
	month = mar,
	number = {2},
	publisher = {Annals of Mathematics},
	title = {The Andr{\'e}-Oort conjecture for {$\mathcal{A}_g$}},
	url = {http://dx.doi.org/10.4007/annals.2018.187.2.2},
	volume = {187},
	year = {2018}}

@article{Pil11,
	author = {Pila, Jonathan},
	doi = {10.4007/annals.2011.173.3.11},
	journal = {Annals of Mathematics},
	language = {en},
	month = may,
	number = {3},
	pages = {1779--1840},
	publisher = {Annals of Mathematics},
	title = {O-minimality and the Andr{\'e}-Oort conjecture for {$\mathbb{C}^n$}},
	url = {http://dx.doi.org/10.4007/annals.2011.173.3.11},
	volume = {173},
	year = {2011}}

@unpublished{GVF,
	author = {Ita\"{i} Ben-Yaacov and Pablo Destic and Ehud Hrushovski and Micha{\l} Szachniewicz},
	note = {arXiv:2409.04570},
	title = {Globally valued fields: foundations},
	year = {2024}}

@phdthesis{NgThesis,
	author = {Ming Ng},
	school = {University of Birmingham},
	title = {Adelic Geometry via Topos Theory},
	year = {2023}}

@article{NVOstrowski,
	abstractnote = {<jats:p>This paper investigates the absolute values on $Z$ valued in the upper reals (i.e. reals for which only a right Dedekind section is given). These necessarily include multiplicative seminorms corresponding to the finite prime fields $mathbb{F}_p$. As an Ostrowski-type Theorem, the space of such absolute values is homeomorphic to a space of prime ideals (with co-Zariski topology) suitably paired with upper reals in the range $[-infty, 1]$, and from this is recovered the standard Ostrowski's Theorem for absolute values on $Q$. Our approach is fully constructive, using, in the topos-theoretic sense, geometric reasoning with point-free spaces, and that calls for a careful distinction between Dedekinds vs. upper reals. This forces attention on topological subtleties that are obscured in the classical treatment. In particular, the admission of multiplicative seminorms points to connections with Berkovich and adic spectra. The results are also intended to contribute to characterising a (point-free) space of places of $Q$.</jats:p>},
	author = {Ng, Ming and Vickers, Steven},
	doi = {10.4115/jla.2025.17.fds6},
	journal = {Journal of Logic and Analysis},
	month = apr,
	publisher = {Journal of Logic and Analysis},
	title = {A Point-free Look at {O}strowski's Theorem and Absolute Values},
	url = {http://dx.doi.org/10.4115/jla.2025.17.FDS6},
	volume = {17},
	year = {2025}}

@article{ConnesConsaniATIYAH,
	author = {Alain Connes and Caterina Consani},
	journal = {Quarterly Journal of Mathematics},
	pages = {1-29},
	title = {{S}egal's {G}amma rings and universal arithmetic},
	year = {2020}}

@unpublished{ClausenScholzeCondensed,
	author = {Dustin Clausen and Peter Scholze},
	note = {\url{https://people.mpim-bonn.mpg.de/scholze/Complex.pdf}},
	title = {Condensed Mathematics and Complex Geometry},
	url = {https://people.mpim-bonn.mpg.de/scholze/Complex.pdf},
	year = {2022}}

@inbook{PenaLorscheid,
	author = {Javier L\'{o}pez Peña and Oliver Lorscheid},
	chapter = {Mapping $\mathbb{F}_1$-Land: An Overview of Geometries over the field with one element},
	editor = {Caterina Consani and Alain Connes},
	publisher = {The Johns Hopkins University Press},
	title = {Noncommutative Geometry, Arithmetic, and Related Topics},
	year = {2011}}

@article{ArtinWhaples,
	author = {Emil Artin and George Whaples},
	journal = {Bulletin of the American Mathematical Society},
	pages = {469-492},
	title = {Axiomatic characterization of fields by the product formula for valuations},
	volume = {51},
	year = {1945}}

@unpublished{Noohi,
	author = {Behrang Noohi},
	note = {arXiv:0503247},
	title = {Foundations of Topological Stacks I},
	year = {2005}}

@book{LurieHigherTopos,
	author = {Jacob Lurie},
	publisher = {Princeton University Press},
	series = {Annals of Mathematics Studies},
	title = {Higher Topos Theory},
	year = {2009}}

@electronic{LuriePrincipalBundle,
	author = {Jacob Lurie},
	title = {Reply to Mathoverflow: classifying $\infty$-toposes for topological/localic groupoids?},
	url = {https://mathoverflow.net/questions/155166/classifying-infty-toposes-for-topological-localic-groups},
	year = {2014}}

@article{WeilSister,
	author = {Andr\'{e} Weil},
	journal = {Notices of the AMS},
	number = {3},
	pages = {334-341},
	title = {Letter to {S}imone {W}eil (1940) (transl. by {M}artin {H}. {K}rieger)},
	volume = {52},
	year = {2005}}

@article{BauerConstructive,
	author = {Andrej Bauer},
	journal = {Bulletin of the American Mathematical Society},
	month = {July},
	number = {3},
	pages = {481-498},
	title = {Five Stages of Accepting Constructive Mathematics},
	volume = {54},
	year = {2017}}

@article{ViPowerlocaleEXP,
	author = {Steven Vickers},
	journal = {Theory andApplications of Categories},
	number = {13},
	pages = {372-422},
	title = {The double powerlocale and exponentiation: A case study in geometric logic},
	volume = {12},
	year = {2004}}

@unpublished{Pitts,
	author = {A. M. Pitts},
	note = {Slides for a talk at the Category Theory Conference, Cambridge},
	title = {Lax descent for essential surjections},
	year = {1986}}

@article{BungePittsLax,
	author = {Marta Bunge},
	journal = {Tbilisi Mathematical Journal},
	number = {1},
	pages = {1-29},
	title = {{P}itts monads and a lax descent theorem},
	volume = {8},
	year = {2015}}

@inbook{TierneyM,
	author = {Myles Tierney},
	chapter = {On the spectrum of a ringed topos},
	editor = {A. Heller and M. Tierney},
	pages = {189-210},
	publisher = {Academic Press, New York},
	title = {Algebra, Topology and Category Theory: A Collection of Papers in Honor of Samuel Eilenberg},
	year = {1976}}

@article{ColeSpectra,
	author = {Julian Cole},
	journal = {Reprints in Theory and Applications of Categories},
	pages = {1-16},
	title = {The bicategory of topoi and spectra},
	volume = {25},
	year = {2016}}

@unpublished{VickersPtfreePtwise,
	author = {Steven Vickers},
	note = {arXiv:2206.01113},
	title = {Generalized Point-free Spaces, Pointwise},
	year = {2022}}

@article{Payne,
	author = {Sam Payne},
	journal = {Bull. Amer. Math. Soc; Bull. Amer. Math. Soc.},
	number = {2},
	pages = {223-247},
	title = {Topology of Non-{A}rchimedean analytic spaces and relations to complex algebraic geometry},
	volume = {52},
	year = {2015}}

@article{Arakelov,
	author = {S. Yu. Arakelov},
	journal = {Izv. Akad. Nauk SSSR Ser. Mat.},
	number = {6},
	pages = {1179-1192},
	title = {Intersection Theory of Divisors on an Arithmetic Surface},
	volume = {38},
	year = {1974}}

@book{ArakelovDiophantine,
	editor = {Emmanuel Peyre and Ga{\"e}l R{\'e}mond},
	publisher = {Springer, Cham},
	title = {Arakelov Geometry and Diophantine Applications},
	year = {2021}}

@book{Borceux,
	author = {Francis Borceux},
	number = {52},
	publisher = {Cambridge University Press},
	series = {Encyclopedia of Mathematics and its Applications},
	title = {Handbook of Categorical Algebra 3: Categories of Sheaves},
	year = {1994}}

@article{Bunge,
	author = {Marta Bunge},
	journal = {Mathematical Proceedings of the Cambridge Philosophical Society},
	pages = {59-79},
	title = {An Application of Descent to a Classifcation Theorem for Toposes},
	volume = {107},
	year = {1990}}

@book{J0,
	author = {P.T. Johnstone},
	publisher = {Academic Press},
	title = {Topos Theory},
	year = {1977}}

@article{Moerdijk,
	author = {Ieke Moerdijk},
	journal = {Transactions of the AMS},
	number = {2},
	title = {The Classifying Topos of a Continuous Groupoid {I}},
	volume = {310},
	year = {1988}}

@article{BG,
	author = {Scott Balchin and J.P.C. Greenlees},
	journal = {Advances in Mathematics},
	pages = {107339},
	title = {Adelic models of tensor-triangulated categories},
	volume = {375},
	year = {2020}}

@article{Maz,
	author = {Barry Mazur},
	journal = {Bulletin of the AMS},
	number = {1},
	pages = {14-50},
	title = {On the Passage from Local to Global in Number Theory},
	volume = {29},
	year = {1993}}

@book{BerkovichMonograph,
	author = {Vladimir Berkovich},
	publisher = {American Mathematical Society},
	title = {Spectral Theory and Analytic Geometry over Non-Archimedean Fields},
	year = {1990}}

@book{HruLoe,
	author = {Ehud Hrushovski and Fran{\c c}ois Loeser},
	publisher = {Princeton University Press},
	title = {Non-archimedean Tame Topology and Stably Dominated Types},
	year = {2016}}

@article{JohnstoneSpectra,
	author = {P. T. Johnstone},
	journal = {Journal of Algebra},
	pages = {238-260},
	title = {Rings, Fields, and Spectra},
	volume = {49},
	year = {1977}}

@book{JoyalTierney,
	author = {Andr\'{e} Joyal and Myles Tierney},
	publisher = {Memoirs of the AMS},
	title = {An Extension of the Galois Theory of Grothendieck},
	year = {1984}}

@article{NV,
	author = {Ming Ng and Steven Vickers},
	journal = {Logical Methods in Computer Science},
	number = {3},
	pages = {1-32},
	publisher = {Logical Methods in Computer Science},
	title = {Point-free Construction of Real Exponentiation},
	volume = {18},
	year = {2022}}

@phdthesis{R,
	author = {Guillaume Raynaud},
	school = {School of Computer Science, University of Birmingham},
	title = {Fibred Contextual Quantum Physics},
	year = {2014}}

@article{Smyth,
	author = {M.B. Smyth},
	journal = {Theoretical Computer Science},
	number = {3},
	pages = {257-274},
	title = {Effectively Given Domains},
	volume = {5},
	year = {1977}}

@book{vdW1,
	author = {B.L. van der Waerden},
	publisher = {Springer-Verlag},
	title = {Algebra},
	volume = {1},
	year = {1991}}

@article{Vi3,
	author = {Steven Vickers},
	journal = {Journal of Applied Logic},
	number = {1},
	pages = {14-27},
	title = {Continuity and Geometric Logic},
	volume = {12},
	year = {2014}}

@incollection{Vi07,
	author = {Steven Vickers},
	booktitle = {Handbook of Spatial Logics},
	editor = {M Aiello and I E Pratt-Hartmann and J F van Benthem},
	pages = {429-496},
	publisher = {Springer},
	title = {Locales and Toposes as Spaces},
	year = {2007}}

@article{Vi8,
	author = {Steven Vickers},
	journal = {Theoretical Computer Science},
	pages = {201-229},
	title = {Information Systems for Continuous Posets},
	volume = {114},
	year = {1993}}

@article{ViLocCompII,
	author = {Steven Vickers},
	journal = {Journal of Logic and Analysis},
	number = {11},
	pages = {1-48},
	title = {Localic completion of generalized metric spaces II: Powerlocales},
	volume = {1},
	year = {2009}}

@article{ViSublocales,
	author = {Steven Vickers},
	journal = {Journal of Symbolic Logic},
	number = {2},
	pages = {463-482},
	title = {Sublocales in Formal Topology},
	volume = {72},
	year = {2007}}

\end{document}